\documentclass[11pt,oneside,reqno]{amsart}
\pdfoutput=1

\usepackage[OT2,T1]{fontenc}
\usepackage[utf8]{inputenc}
\usepackage{amssymb,bm}

\usepackage[dvipsnames]{xcolor}
\usepackage[
    colorlinks=true,
    linkcolor=Maroon,
    citecolor=JungleGreen,
    urlcolor=NavyBlue]{hyperref}

\usepackage{xcolor,colortbl}
\usepackage[capitalize]{cleveref}

\usepackage{enumitem}
\usepackage{dsfont}
\usepackage{tikz}
\usetikzlibrary{patterns}
\usepackage{subcaption}
\usetikzlibrary{arrows}
\usepackage{multirow}

\usepackage{lipsum}

\usepackage[normalem]{ulem}
\usepackage{cases}

\newtheorem{theorem}{Theorem}[section]

\newtheorem{lemma}[theorem]{Lemma}
\newtheorem{proposition}[theorem]{Proposition}
\newtheorem{corollary}[theorem]{Corollary}

\newtheorem{remark}[theorem]{Remark}

\numberwithin{equation}{section}

\DeclareMathOperator{\Ad}{Ad}
\DeclareMathOperator{\im}{Im}
\DeclareMathOperator{\re}{Re}
\DeclareMathOperator{\id}{id}
\DeclareMathOperator{\supp}{supp}
\DeclareMathOperator{\const}{const.}
\DeclareMathOperator{\Hess}{Hess}

\newcommand{\OpNorme}[3]{
\|{#1}\|_{L^{{#2}}(\mathbb{X})
\rightarrow{L^{{#3}}(\mathbb{X})}}}

\newcommand\frakACL{\overline{\mathfrak{a}^{+}}}
\newcommand\Deltarad{\Delta^{\textnormal{rad}}}
\newcommand\DeltaradP{
\Delta_{\mathfrak{p}}^{\textnormal{rad}}}
\newcommand\DeltaA{\Delta_{\mathfrak{a}}}
\newcommand\DeltaP{\Delta_{\mathfrak{p}}}
\newcommand\nablaA{\nabla_{\mathfrak{a}}}
\newcommand\nablaP{\nabla_{\mathfrak{p}}}

\makeatletter
\def\@tocline#1#2#3#4#5#6#7{\relax
  \ifnum #1>\c@tocdepth 
  \else
    \par \addpenalty\@secpenalty\addvspace{#2}%
    \begingroup \hyphenpenalty\@M
    \@ifempty{#4}{%
      \@tempdima\csname r@tocindent\number#1\endcsname\relax
    }{%
      \@tempdima#4\relax
    }%
    \parindent\z@ \leftskip#3\relax
    \advance\leftskip\@tempdima\relax
    \rightskip\@pnumwidth plus4em \parfillskip-\@pnumwidth
    #5\leavevmode\hskip-\@tempdima
      \ifcase #1
       \or\or \hskip 2em \or \hskip 2em \else \hskip 3em \fi%
      #6\nobreak\relax
    \dotfill\hbox to\@pnumwidth{\@tocpagenum{#7}}\par
    \nobreak
    \endgroup
  \fi}
\makeatother

\usepackage[
      backend=biber,
      style=alphabetic,
      giveninits=true,
      maxbibnames=10,
      sorting=nyt,
      url=false,
      eprint=false,
      isbn=false
      ]{biblatex}

\DeclareFieldFormat[article,incollection]{title}{#1}
\DeclareFieldFormat{date}{#1}
\DeclareFieldFormat{pages}{{#1}}
\DeclareFieldFormat{doi}{
\href{https://www.doi.org/#1}{Doi}}
\DeclareFieldFormat{url}{
\href{#1}{Doi}}

\renewbibmacro{in:}{%
  \ifentrytype{article}{}{\printtext{\bibstring{in}\intitlepunct}}}

\begin{document}
\title[Wave equation on general 
noncompact symmetric spaces]
{Wave equation on general 
noncompact symmetric spaces}

\author{Jean-Philippe ANKER \& Hong-Wei ZHANG}

\subjclass[2020]
{22E30, 35L05, 43A85, 43A90}

\keywords{
noncompact symmetric space of higher rank,
semi-linear wave equation, dispersive property,
Strichartz inequality, global well-posedness}

\thanks{The second author acknowledges financial supports from the University of Orléans during his Ph.D. and from the Methusalem Programme \textit{Analysis and Partial Differential Equations (Grant number 01M01021)} during his postdoc stay at Ghent University.}

\maketitle

\vspace{-10pt}
\begin{center}
\it
Dedicated to the memory of Robert S. Strichartz (1943-2021)
\end{center}

\begin{abstract}
We establish sharp pointwise kernel estimates 
and dispersive properties for the wave equation 
on noncompact symmetric spaces of general rank.
This is achieved by combining the stationary phase method
and the Hadamard parametrix, and in particular,
by introducing a subtle spectral 
decomposition, which allows us to overcome a
well-known difficulty in higher rank analysis, 
namely the fact that the Plancherel density
is not a differential symbol in general.
Consequently, we deduce the Strichartz inequality
for a large family of admissible pairs and prove 
global well-posedness results for the corresponding
semi-linear equation with low regularity data as 
on hyperbolic spaces.
\end{abstract}

\setcounter{tocdepth}{1}
\tableofcontents

\section{Introduction}
This paper is devoted to prove sharp-in-time kernel estimates 
and dispersive properties for the wave equation on 
noncompact symmetric spaces of higher rank.
Consequently, we prove the Strichartz inequality and 
study their applications to associated semi-linear Cauchy
problems. 
Relevant theories are well established on 
Euclidean spaces, see for instance 
\cite{Kap1994,LiSo1995,GLS1997,KeTa1998,DGK2001}, 
and the references therein.

Given the rich Euclidean results, several works have 
been made in other settings. We are interested in
Riemannian symmetric spaces of noncompact type, where 
relevant questions are now well answered in rank one, 
see for instance 
\cite{Fon1997,Ion2000,Tat2001,MeTa2011,MeTa2012,APV2012,AnPi2014} 
on hyperbolic spaces, and \cite{APV2015} on Damek-Ricci 
spaces. 
A first study of the wave equation on general symmetric
spaces of higher rank was carried out in \cite{Has2011},
where some non optimal estimates were obtained under 
a strong smoothness assumption. 
Recently, sharp-in-time kernel estimates and dispersive 
properties have been proven in \cite{Zha2021} 
on noncompact symmetric spaces $G/K$, with $G$ complex. 
In this case, the Harish-Chandra $\mathbf{c}$-function 
and the spherical functions have elementary expressions,
which is not the case in general.

In this paper, we establish pointwise wave kernel 
estimates and dispersive properties for the wave equation 
on general noncompact symmetric spaces, which are
sharp in time and extend previous results obtained
on real hyperbolic spaces \cite{APV2012,AnPi2014} 
to higher rank. 
The main challenge is that the Plancherel density 
involved in the wave kernel is not a polynomial, 
nor even a differential symbol in general. To bypass 
this problem, we consider barycentric decompositions 
of the Weyl chambers into subcones and differentiate 
in each subcone along a well chosen direction.

For suitable $\sigma \in \mathbb{C}$, 
we consider the wave operator 
$W_{t}^{\sigma}=(-\Delta)^{- \frac{\sigma}{2}}
e^{it \sqrt{- \Delta}}$ 
associated to the Laplace-Beltrami operator $\Delta$ 
on a $d$-dimensional noncompact symmetric space $\mathbb{X}=G/K$. 
To avoid possible singularities (see \cref{subsection 2}),
we consider actually the analytic family of operators 
\begin{align}\label{intro-wave operator}
    \widetilde{W}_{t}^{\sigma}
    = \frac{e^{\sigma^2}}
        {\Gamma (\frac{d+1}{2}-\sigma)}
        (- \Delta)^{- \frac{\sigma}{2}} 
        e^{it \sqrt{- \Delta}}
\end{align}
in the vertical strip $0\le\re\sigma\le\frac{d+1}{2}$. 
Let us denote by $\widetilde{\omega}_{t}^{\sigma}$ its 
$K$-bi-invariant convolution kernel. Our first main result
is the following pointwise estimate, which summarizes Theorems
\ref{theorem subsection 1}, \ref{theorem subsection 2},
and \ref{theorem subsection 3} proved in 
\cref{section kernel estiamtes}.

\begin{theorem}
[Pointwise kernel estimates]\label{main thm1}
Let $d\ge3$ and $\sigma\in\mathbb{C}$ with
$\re\sigma=\frac{d+1}{2}$. There exist $C>0$ and
$N\in\mathbb{N}$ such that the following estimates 
hold for all $t\in\mathbb{R}^{*}$ and $x\in\mathbb{X}$:
\begin{align*}
    |\widetilde{\omega}_{t}^{\sigma}(x)| 
    \le C (1+|x^{+}|)^{N} 
        e^{-\langle\rho,x^{+}\rangle}
    \begin{cases}
        |t|^{-\frac{d-1}{2}} 
        &\textnormal{if \,} 0<|t|<1,\\
        |t|^{-\frac{D}{2}} 
        &\textnormal{if \,} |t|\ge1,
    \end{cases}
\end{align*}
where $x^{+}\in\overline{\mathfrak{a}^{+}}$ denotes 
the radial component of $x$ in the Cartan decomposition,
and $D=\ell+2|\Sigma_{r}^+|$ is the so-called dimension 
at infinity of $\mathbb{X}$.
\end{theorem}

\begin{remark}
These kernel estimates are sharp in time and similar
results hold obviously in the easier case where
$\re\sigma>\frac{d+1}{2}$. 
The value of $N$ will be specified in 
\cref{section kernel estiamtes}. However, the polynomial
$(1+|x^{+}|)^{N}$ is not crucial for further computations
because of the exponential decay $e^{-\langle\rho,x^{+}\rangle}$.
\end{remark}

By interpolation arguments, 
we deduce our second main result.
\begin{theorem}[Dispersive property]\label{main thm2}
Assume that $d\ge3$, $2<q,\widetilde{q}<+\infty$ and 
$\sigma\ge(d+1)\max(\frac{1}{2}-\frac{1}{q},
\frac{1}{2}-\frac{1}{\widetilde{q}})$.
Then there exists a constant $C>0$ such that 
the following dispersive estimates hold:
\begin{align*}
    \OpNorme{W_{t}^{\sigma}}{\widetilde{q}'}{q}
    \le C
    \begin{cases}
        |t|^{-(d-1)\max(\frac{1}{2}-\frac{1}{q},
        \frac{1}{2}-\frac{1}{\widetilde{q}})} 
        &\textnormal{if \,} 0<|t|<1, \\
        |t|^{-\frac{D}{2}} 
        &\textnormal{if \,} |t|\ge1.
    \end{cases}
    \end{align*}
\end{theorem}

\begin{remark}
At the endpoint $q=\widetilde{q}=2$,
$t\mapsto{e^{it\sqrt{-\Delta}}}$ is a one-parameter 
group of unitary operators on $L^2(\mathbb{X})$. 
\end{remark}

\begin{remark}
\cref{main thm1} and \cref{main thm2} generalize 
earlier results obtained for real hyperbolic spaces 
\cite{APV2012,AnPi2014} (which extend straightforwardly 
to all noncompact symmetric spaces of rank one), 
or for noncompact symmetric spaces $G/K$ with $G$ 
complex \cite{Zha2021}. Notice that $D=3$ in rank one 
and that $D=d$ if $G$ is complex.
\end{remark}

\begin{remark}
For simplicity, we omit the $2$-dimensional case 
where the small time bounds in \cref{main thm1} and 
\cref{main thm2} involve an additional logarithmic factor,
see \cite[Theorem 3.2 and 4.2]{AnPi2014}. 
Notice that $d\ge4$ in higher rank, 
see \eqref{dimensions}.
\end{remark}

Let us sketch the proofs of our main results. 
We prove the dispersive properties of 
$W_{t}^{\sigma}$ by using interpolation 
arguments based on pointwise estimates of 
$\widetilde{\omega}_{t}^{\sigma}$, which are sharp in time.
By the way, let us point out that the kernel analysis 
carried out on hyperbolic spaces \cite{AnPi2014} 
can not be extended straightforwardly in higher rank, 
since the Plancherel density is not a differential symbol 
in general.
Consider the Poisson operator 
$\mathcal{P}_{\tau}= e^{-\tau \sqrt{-\Delta}}$, 
for all $\tau\in\mathbb{C}$ with $\re\tau\ge0$.  
Along the lines of \cite{Sch1988, GiMe1990,CGM2002}, 
we can write formally our wave operator 
\eqref{intro-wave operator} as
\begin{align*}
    \widetilde{W}_{t}^{\sigma} 
    = \frac{e^{\sigma^2}}
        {\Gamma(\frac{d+1}{2}-\sigma)}
        \frac{1}{\Gamma(\sigma)}
        \int_{0}^{+\infty} ds\,s^{\sigma-1}
        \mathcal{P}_{s-it}.
\end{align*}
Our analysis is focused on kernel estimates of the Poisson 
operator $\mathcal{P}_{s-it}$ where $s\in \mathbb{R}^{+}$ and 
$t\in \mathbb{R}^{*}$. We adopt different methods depending 
whether $s$, $|t|$ and $\frac{|x|}{|t|}$ ($x\in\mathbb{X}$) are 
small or large. Specifically, 
\begin{itemize}[leftmargin=*]
    \item 
    If $s$ is bounded from above and $\frac{|x|}{|t|}$ is 
    sufficiently small with $|t|$ large, we develop an effective stationary phase method based on barycentric decompositions
    of Weyl chambers described in \cref{subsection decomposition}. 
    In each subdivision, the Plancherel density becomes a differential symbol for a well chosen directional derivative, 
    see \cref{subsection 1}.
    \item
    If $s$ is bounded from above but $\frac{|x|}{|t|}$ is large 
    (with $|t|$ small or large), we estimate the kernel along the
    lines of \cite{CGM2001}, where Cowling, Guilini and Meda have
    studied the Poisson operator $\mathcal{P}_{\tau}$ for  $\tau\in\mathbb{C}$ with $\re\tau\ge0$.
    Unfortunately, their estimates are not sharp when $\tau$ is large
    and nearly imaginary, which happens in our context when $s$
    is small and $|t|$ is large. To deal with this case, we resume
    and improve slightly their method by writing down more explicitly
    the Hadamard parametrix on noncompact symmetric spaces along
    the lines of \cite{Ber1977}, see \cref{subsection 2}.
    \item
    If $s$ is large, the kernel is estimated by using the standard
    stationary phase method, which is similar to the rank one 
    analysis, see \cref{subsection 3}.
\end{itemize}

This paper is organized as follows. We recall spherical Fourier 
analysis on noncompact symmetric spaces and introduce the 
barycentric decomposition of Weyl chambers in 
\cref{section preliminaires}. 
Next, we derive pointwise wave kernel estimates in 
\cref{section kernel estiamtes}. 
By using interpolation arguments, we prove in 
\cref{section dispersive estimates} the dispersive property 
for the wave operator. 
As consequences, we establish the Strichartz inequality for 
a large family of admissible pairs and obtain well-posedness 
results for the associated semi-linear wave equation in 
\cref{section applications}. 
We give further results about the Klein-Gordon equation
in \cref{section Klein-Gordon}.
Finally, we collect in the appendices some useful results:
in \cref{Appendix A}, we study by the stationary phase method
an oscillatory integral occurring in the wave kernel analysis; 
next we describe in \cref{Appendix B} the Hadamard parametrix on
noncompact symmetric spaces and consider its application 
to the Poisson operator in \cref{Appendix C}.

\section{Preliminaries}\label{section preliminaires}
In this section, we first review briefly spherical 
Fourier analysis on noncompact symmetric spaces. 
Next we introduce a barycentric decomposition for 
Weyl chambers, which will be crucial for the 
forthcoming kernel estimates.

\subsection{Notations}\label{subsection notations}
We adopt the standard notation and refer to 
\cite{Hel1978, Hel2000} for more details. 
Let $G$ be a semisimple Lie group, connected, noncompact,
with finite center, and $K$ be a maximal compact subgroup
of $G$. The homogeneous space $\mathbb{X}=G/K$ is 
a Riemannian symmetric space of noncompact type. 
Let $\mathfrak{g} = \mathfrak{k} \oplus \mathfrak{p}$ be
the Cartan decomposition of the Lie algebra of $G$. 
There is a natural identification between $\mathfrak{p}$
and the tangent space of $\mathbb{X}$ at the origin. 
The Killing form of $\mathfrak{g}$ induces a $K$-invariant
inner product on $\mathfrak{p}$, hence a $G$-invariant
Riemannian metric on $\mathbb{X}$.

Fix a maximal abelian subspace $\mathfrak{a}$ in
$\mathfrak{p}$. The rank of $\mathbb{X}$ is the dimension
$\ell$ of $\mathfrak{a}$. Let $\Sigma\subset\mathfrak{a}$
be the root system of $( \mathfrak{g},\mathfrak{a})$ and
denote by $W$ the Weyl group associated to $\Sigma$. 
Once a positive Weyl chamber 
$\mathfrak{a}^{+}\subset\mathfrak{a}$ has been selected,
$\Sigma^{+}$ (resp. $\Sigma_{r}^{+}$ or $\Sigma_{s}^{+}$)
denotes the corresponding set of positive roots 
(resp. positive reduced roots or simple roots). 
Let $d$ be the dimension of $\mathbb{X}$ and $D$ be the
dimension at infinity of $\mathbb{X}$: 
\begin{align}\label{dimensions}
\textstyle
    d= \ell + \sum_{\alpha \in \Sigma^{+}} m_{\alpha}
    \quad \text{and} \quad
    D= \ell + 2|\Sigma_{r}^{+}|,
\end{align}
where $m_{\alpha}$ is the dimension of the positive 
root subspace $\mathfrak{g}_{\alpha}$. 
Notice that one cannot compare $d$ and $D$ without 
specifying the geometric structure of $\mathbb{X}$. 
For example, when $G$ is complex, we have $d=D$; 
but when $\mathbb{X}$ is a normal real form, 
we have $d=\ell+|\Sigma_{r}^{+}|$ which is strictly 
smaller than $D$. Since we focus on the higher rank
analysis, we may assume that $d\ge3$.

Let $\mathfrak{n}$ be the nilpotent Lie subalgebra 
of $\mathfrak{g}$ associated to $\Sigma^{+}$ 
and let $N = \exp \mathfrak{n}$ be the corresponding 
Lie subgroup of $G$. We have the decompositions 
\begin{align*}
    \begin{cases}
        G=N\left(\exp\mathfrak{a}\right) K 
        \quad& \textnormal{(Iwasawa)}, \\
        G= K(\exp\overline{\mathfrak{a}^{+}})K
        \quad& \textnormal{(Cartan)}.
    \end{cases}
\end{align*}
In the Cartan decomposition, the Haar measure 
on $G$ writes
\begin{align*}
    \int_{G} dx \,f(x) 
    =   
    \const \int_{K}dk_{1}\
    \int_{\mathfrak{a}^{+}}dx^{+}\ \delta(x^{+})  
    \int_{K}dk_{2}\ f(k_{1}(\exp x^{+})k_{2}),
\end{align*}
with 
\begin{align*}
    \delta(x^{+})
    = \prod_{\alpha\in\Sigma^{+}} 
        \left(\sinh \alpha(x^{+})\right)^{m_{\alpha}}
    \asymp \Big\lbrace 
        \prod_{\alpha\in\Sigma^{+}} 
        \frac{\langle\alpha,x^{+}\rangle}
        {1+\langle\alpha,x^{+}\rangle}
    \Big\rbrace^{m_{\alpha}}
        e^{\langle2\rho,x^{+}\rangle}
    \quad \forall\,x^{+}\!\in\frakACL. 
\end{align*}
Here $\rho\in\mathfrak{a}^{+}$ denotes the half sum 
of all positive roots $\alpha \in \Sigma^{+}$ 
counted with their multiplicities $m_{\alpha}$:
\begin{align*}
\textstyle
    \rho 
    = \frac{1}{2} \sum_{\alpha\in\Sigma^{+}} 
        m_{\alpha} \alpha.
\end{align*}

\subsection{Spherical Fourier analysis on $\mathbb{X}$}
\label{subsection spherical fourier analysis}

Let $\mathcal{S}(K \backslash{G}/K)$ be the Schwartz space of $K$-bi-invariant functions on $G$. 
The spherical Fourier transform $\mathcal{H}$ is defined by
\begin{align*}
    \mathcal{H} f(\lambda) 
    = \int_{G}dx\ \varphi_{-\lambda} (x) f(x) 
    \quad\forall\,\lambda\in\mathfrak{a},\ 
    \forall\,f\in\mathcal{S}(K\backslash{G/K}),
\end{align*}
where $\varphi_{\lambda}\in\mathcal{C}^{\infty}
(K \backslash{G/K})$ denotes the spherical function 
of index $\lambda \in \mathfrak{a}_{\mathbb{C}}$, 
which is a smooth $K$-bi-invariant eigenfunction 
for all invariant differential operators on $\mathbb{X}$,
in particular for the Laplace-Beltrami operator:
\begin{equation*}
    - \Delta\varphi_{\lambda}(x) 
    = \left(|\lambda|^2+|\rho|^2\right)         
        \varphi_{\lambda}(x).
\end{equation*}
In the noncompact case, spherical functions have 
the integral representation
\begin{align}\label{integral expression of phi lambda}
    \varphi_{\lambda}(x) 
    = \int_{K}dk\ 
        e^{\langle{i\lambda+\rho,A(kx)}\rangle}
    \quad\forall\,
        \lambda\in{\mathfrak{a}_{\mathbb{C}}},
\end{align}
where $A(kx)$ denotes the $\mathfrak{a}$-component 
in the Iwasawa decomposition of $kx$. 
It satisfies the basic estimate
\begin{align*}
    |\varphi_{\lambda}(x)| 
    \le \varphi_{0}(x)
    \quad\forall\,\lambda\in\mathfrak{a},\
    \forall\,x\in G,
\end{align*}
where
\begin{align*}
    \varphi_{0} (\exp x^{+}) 
    \asymp \Big\lbrace 
        \prod_{\alpha\in\Sigma_{r}^{+}} 
        1+\langle\alpha,x^{+}\rangle
        \Big\rbrace 
        e^{-\langle\rho, x^{+}\rangle} 
        \quad\forall\,x^{+}\in\frakACL.
\end{align*}

Denote by $\mathcal{S}(\mathfrak{a})^{W}$ the subspace 
of $W$-invariant functions in the Schwartz space
$\mathcal{S}(\mathfrak{a})$. Then $\mathcal{H}$ is an
isomorphism between $\mathcal{S}(K\backslash{G/K})$ 
and $\mathcal{S}(\mathfrak{a})^{W}$. 
The inverse spherical Fourier transform is given by
\begin{align*}
    f(x) 
    = C_0 \int_{\mathfrak{a}}d\lambda\
        |\mathbf{c(\lambda)}|^{-2} 
        \varphi_{\lambda}(x)\,
        \mathcal{H}f(\lambda) 
    \quad \forall\,x\in{G},\ 
        \forall\,f\in\mathcal{S}(\mathfrak{a})^{W},
\end{align*}
where $C_0 > 0$ is a constant depending only 
on the geometric structure of $\mathbb{X}$, 
and which has been computed explicitly for instance in
\cite[Theorem 2.2.2]{AnJi1999}. 
By using the Gindikin \& Karpelevič formula of 
the Harish-Chandra $\mathbf{c}$-function 
(see \cite{Hel2000} or \cite{GaVa1988}), 
we can write the Plancherel density as
\begin{align}\label{Plancherel density}
    |\mathbf{c}(\lambda)|^{-2}
    = \prod_{\alpha\in\Sigma_{r}^{+}}
        |\mathbf{c}_{\alpha}
        (\langle\alpha,\lambda\rangle)|^{-2},
\end{align}
with
\begin{align*}
\textstyle
    \mathbf{c}_{\alpha}(v) =
    \frac{\Gamma(\frac{\langle{\alpha,\rho}\rangle}
        {\langle{\alpha,\alpha}\rangle}
        +\frac{1}{2} m_{\alpha})}
        {\Gamma(\frac{\langle{\alpha,\rho}\rangle}
        {\langle{\alpha,\alpha}\rangle})}
    \frac{\Gamma(\frac{1}{2}
        \frac{\langle{\alpha,\rho}\rangle}
        {\langle{\alpha,\alpha}\rangle} 
        +\frac{1}{4} m _{\alpha} 
        + \frac{1}{2} m_{2\alpha})}
        {\Gamma(\frac{1}{2}
        \frac{\langle{\alpha,\rho}\rangle}
        {\langle {\alpha,\alpha}\rangle} 
        + \frac{1}{4} m_{\alpha})}
    \frac{\Gamma(iv)}
        {\Gamma(iv+ \frac{1}{2}m_{\alpha})}
    \frac{\Gamma(\frac{i}{2}v 
        + \frac{1}{4} m_{\alpha})}
        {\Gamma(\frac{i}{2}v 
        +\frac{1}{4} m_{\alpha} 
        + \frac{1}{2} m_{2\alpha})}.
\end{align*}
Since $|\mathbf{c}_{\alpha}|^{-2}$ is a homogeneous 
symbol on $\mathbb{R}$ of order $m_{\alpha}+m_{2\alpha}$ 
for every $\alpha\in\Sigma_{r}^{+}$, then
$|\mathbf{c}(\lambda)|^{-2}$ is a product of 
one-dimensional symbols, but not a symbol on 
$\mathfrak{a}$ in general. 
The Plancherel density satisfies
\begin{align*}
    |\mathbf{c}(\lambda)|^{-2} 
    \asymp \prod_{\alpha\in\Sigma_{r}^{+}}
        \langle\alpha,\lambda\rangle^2 
        (1+|\langle\alpha,\lambda\rangle|
        )^{m_{\alpha}+m_{2\alpha}-2}
    \lesssim
    \begin{cases}
        |\lambda|^{D-\ell} 
        & \textnormal{if \,} |\lambda|\le1,\\
        |\lambda|^{d-\ell} 
        & \textnormal{if \,} |\lambda|\ge1.
    \end{cases}
\end{align*}
Occasionally we split up
\begin{align*}
  \mathbf{c}(\lambda)=\boldsymbol{\pi}(i\lambda)^{-1}\,\mathbf{b}(\lambda)
\quad\text{and}\quad
|\mathbf{c}(\lambda)|^{-2}=\boldsymbol{\pi}(\lambda)^2\,|\mathbf{b}(\lambda)|^{-2},
\end{align*}
where $\boldsymbol{\pi}(\lambda)=\prod_{\alpha\in\Sigma_r^+}\langle\alpha,\lambda\rangle$.

\subsection{Barycentric decomposition of the Weyl chamber}
\label{subsection decomposition}
Let $\Sigma_{s}^{+}=\lbrace \alpha_{1},\dots,
\alpha_{\ell} \rbrace$ be the set of positive 
simple roots, and let 
$\lbrace \Lambda_{1},\dots, \Lambda_{\ell} \rbrace$ 
be the dual basis of $\mathfrak{a}$, 
which is defined by
\begin{align}\label{definition dual basis}
    \langle\alpha_{j},\Lambda_{k}\rangle 
    = \delta_{jk}
    \quad\forall\,1\le{j,k}\le\ell.
\end{align}
Notice that $\overline{\mathfrak{a}^{+}} 
=\mathbb{R}^{+}\Lambda_{1}+\cdots+\mathbb{R}^{+} 
\Lambda_{\ell}$ and recall that
\begin{align}\label{property dual basis}
    \begin{cases}
        \langle \alpha_{j},\alpha_{k} \rangle \le 0
        \quad& \forall\, 1 \le j\neq k \le \ell\\[5pt]
        \langle \Lambda_{j},\Lambda_{k} \rangle \ge 0
        \quad& \forall\, 1 \le j,k \le \ell
    \end{cases}
\end{align}
(see \cite[Chap.VII, Lemmas 2.18 and 2.25]{Hel1978},
see also \cite[p.590]{Kor1993}).
Let $\mathfrak{B}$ be the convex hull of 
$W.\Lambda_{1} \sqcup \cdots \sqcup W.\Lambda_{\ell}$, 
and let $\mathfrak{S}$ be its polyhedral boundary. 
Notice that 
$\mathfrak{B} \cap \overline{\mathfrak{a}^{+}}$ 
is the $\ell$-simplex with vertices 
$0,\Lambda_{1},\dots,\Lambda_{\ell}$, and 
$\mathfrak{S} \cap \overline{\mathfrak{a}^{+}}$ 
is the $(\ell-1)$-simplex with vertices 
$\Lambda_{1},\dots,\Lambda_{\ell}$. 
The following tiling is obtained by regrouping the
barycentric subdivisions of the simplices 
$\mathfrak{S} \cap \overline{w.\mathfrak{a}^{+}}$:
\begin{align}\label{tiling S gothic}
    \mathfrak{S}
    = \bigcup_{w \in W} \bigcup_{1\le j\le \ell} w.     
        \mathfrak{S}_{j}
\end{align}
where
\begin{align*}
    \mathfrak{S}_{j} 
    = \lbrace
        \lambda\in\mathfrak{S}\cap\frakACL
        \ |\ \langle \alpha_{j},\lambda\rangle
        = \max_{1\le k \le \ell} 
        \langle\alpha_{k},\lambda \rangle
    \rbrace
\end{align*}
(see Figure 1).

\begin{figure}
    \vspace{-0.75cm}
    \begin{subfigure}[t]{0.5\textwidth}
    \centering
    \begin{tikzpicture}[line cap=round,line join=round,>=triangle 45,x=1.0cm,y=1.0cm,scale=0.9]
    \clip(-2.5,-4.75) rectangle (8,5);
    \draw [line width=1.pt,color=gray] (0.,0.)-- (-0.001924075134753982,4.1036328374460025);
    \draw [line width=1.pt,color=gray] (0.,0.)-- (3.506570339054844,2.043790363273668);
    \draw [line width=1.pt,color=red] (0.,3.)-- (2.5980762113533165,1.5);
    \draw [->,line width=1.pt] (0.,0.) -- (1.522051206786732,2.5771087087401936);
    \draw [->,line width=1.pt] (0.,0.) -- (3.022051206786732,-0.0209675026131233);
    \draw [->,line width=1.pt] (0.,0.) -- (-1.4779487932132676,2.577108708740194);
    \draw [->,line width=1.pt,color=blue] (0.,0.) -- (0.,3.);
    \draw [->,line width=1.pt,color=blue] (0.,0.) -- (2.6282805793082487,1.5060448360454495);
    \begin{scriptsize}
    \draw (0.,0.) [fill] circle (1.5pt);
    \draw (-0.3,-0.3) node {\large $O$};
    \draw (1,4) node {\large $\overline{\mathfrak{a}^{+}}$};
    \draw (-2,2.6) node {\large $\alpha_{2}$};
    \draw (3.3,-0.4) node {\large $\alpha_{1}$};
    \draw (-0.4,3.0) node {\large $\Lambda_{2}$};
    \draw (2.85,1.2) node {\large $\Lambda_{1}$};
    \draw (0.9,2.9) node {$\mathfrak{S}_{1}$};
    \draw (2.2,2.1) node {$\mathfrak{S}_{2}$};
    \draw (1.33,2.23) [red,fill=red] circle (1.5pt);
    \draw (1.4,1.9) node {$B$};
    \end{scriptsize}
    \end{tikzpicture}
    \end{subfigure}
    ~ 
    \begin{subfigure}[t]{0.5\textwidth}
    \centering
    \begin{tikzpicture}[line cap=round,line join=round,>=triangle 45,x=1.0cm,y=1.0cm, scale=0.7]
    \clip(-2.5,-8) rectangle (8,5);
    \draw [line width=1.2pt,gray] (-0.7284086688607592,0.35576067574813774)-- (4.985446010438759,3.3119338855607467);
    \draw [line width=1.2pt,gray] (-0.7284086688607592,0.35576067574813774)-- (4.139419219106337,-2.6566112861679647);
    \draw [line width=1.2pt,dash pattern=on 4pt off 4pt,gray] (-0.7284086688607592,0.35576067574813774)-- (6.005313923277843,-1.1499882331102318);
    \draw [->,line width=1.pt,dash pattern=on 4pt off 4pt,color=blue] (-0.7284086688607592,0.35576067574813774) -- (5.031706380167742,-0.934680821414303);
    \draw [line width=1.2pt,color=red] (2.9507617363248198,0.15475396250089046)-- (3.6404547647133185,-0.20881041335112677);
    \draw [line width=1.2pt,color=red] (3.6404547647133185,-0.20881041335112677)-- (3.973134704848195,0.6501522010641021);
    \draw [line width=1.2pt,color=red] (3.6404547647133185,-0.20881041335112677)-- (4.006958916900773,-1.437959161361985);
    \fill[line width=1.pt,color=red,fill=red,fill opacity=0.10000000149011612] (2.914584356650141,2.2349532937949) -- (2.987171397456459,-1.9388015525683633) -- (5.031706380167742,-0.934680821414303) -- cycle;
    \draw [line width=1.2pt,color=red] (2.987171397456459,-1.9388015525683633)-- (5.031706380167742,-0.934680821414303);
    \draw [line width=1.2pt,color=red] (5.031706380167742,-0.934680821414303)-- (2.914584356650141,2.2749532937949);
    \draw [line width=1.2pt,color=red] (2.914584356650141,2.2349532937949)-- (2.987171397456459,-1.9555);
    \draw [red,fill=red] (2.925584356650141,2.2449532937949) circle (1.5pt);
    \draw [red,fill=red] (5.031706380167742,-0.934680821414303) circle (1.5pt);
    \draw [red,fill=red] (2.987171397456459,-1.9388015525683633) circle (1.5pt);
    \begin{scriptsize}
    \draw [fill] (-0.7284086688607592,0.35576067574813774) circle (2pt);
    \draw (-1.2,0.4) node {\large $O$};
    \draw (3,-2.5) node {\large $\Lambda_2$};
    \draw (3,2.7) node {\large $\Lambda_1$};
    \draw (5.3,-0.6) node {\large $\Lambda_3$};
    \draw [red,fill=red] (3.6404547647133185,-0.20881041335112677) circle (1.5pt);
    \draw (3.55,0.1) node {$B$};
    \draw (3.4,0.7) node {$\mathfrak{S}_{1}$};
    \draw (3.4,-0.8) node {$\mathfrak{S}_{2}$};
    \draw (4.2,-0.4) node {$\mathfrak{S}_{3}$};
    \draw (4.7,2.5) node {\large $\overline{\mathfrak{a}^{+}}$};
    \draw [->,line width=1.pt,color=blue] (-0.7284086688607592,0.35576067574813774) -- (2.987171397456459,-1.938801552568363);
    \draw [->,line width=1.pt,color=blue] (-0.7284086688607592,0.35576067574813774) -- (2.914584356650141,2.2349532937949);
    \end{scriptsize}
    \end{tikzpicture}
    \end{subfigure}
    \vspace{-3.75cm}
    \caption{Examples of barycentric subdivisions in $A_2$ and in $A_3$.}
    \end{figure}
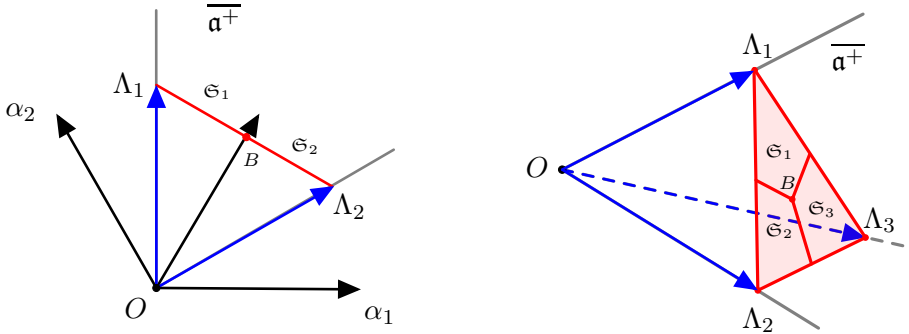

\begin{remark}
$\mathfrak{S}_{j}$ is the convex hull of the points 
\begin{align*}
    \frac{\Lambda_{k_{1}}+\cdots+\Lambda_{k_{r}}}{r}
\end{align*}
where 
$\lbrace \Lambda_{k_{1}},\dots,\Lambda_{k_{r}}\rbrace$ 
runs through all subsets of 
$\lbrace \Lambda_{1},\dots,\Lambda_{\ell} \rbrace$
containing $\Lambda_{j}$.
\end{remark}

\begin{lemma}\label{properties on S gothic}
Let $w\in W$ and $1\le j\le\ell$. Then 
\begin{enumerate}[leftmargin=*, label=(\roman*)]
    \item a root $\alpha \in \Sigma$ is orthogonal 
    to some vectors in the tile $w.\mathfrak{S}_{j}$ 
    if and only if $\alpha$ is orthogonal to 
    its vertex $w.\Lambda_{j}$.
    \vspace{5pt}
    \item $\langle w.\Lambda_{j},\lambda\rangle
    \ge\frac{1}{\ell}|\Lambda_{j}|^2$ for every 
    $\lambda \in w.\mathfrak{S}_{j}$.
\end{enumerate}
\end{lemma}

\begin{proof}
{\it(i)} Let us show that 
$\langle\alpha,w.\Lambda_{j}\rangle=0$ if there exists
$\lambda\in w.\mathfrak{S}_{j}$ such that
$\langle\alpha,\lambda\rangle=0$. By symmetry, we may
assume that $w=\id$ and that $\alpha$ is a positive root.
On the one hand, since $\alpha$ is spanned by the 
positive simple roots $\alpha_{1},\dots,\alpha_{\ell}$, 
we have
\begin{align*}
    \alpha 
    = \sum_{1\le k\le\ell} 
        \langle\alpha,\Lambda_{k}\rangle \alpha_{k}
\end{align*}
with $\langle\alpha,\Lambda_{k}\rangle \in\mathbb{N}$. 
On the other hand, since $\langle\alpha_{1},\lambda\rangle,
\dots,\langle\alpha_{\ell},\lambda\rangle$ are the 
barycentric coordinates of 
$\lambda\in\mathfrak{S}\cap\overline{\mathfrak{a}^{+}}$,
we have
\begin{align}\label{convex combination lambda}
    \lambda 
    =  \sum_{1\le k\le\ell}         
        \langle\alpha_{k},\lambda\rangle \Lambda_{k}
\end{align}
which is a convex combination. 
In particular, $\langle\alpha_{j},\lambda\rangle>0$ 
for all $\lambda\in\mathfrak{S}_{j}$. 
Hence the inner product
\begin{align*}
    \langle\alpha,\lambda\rangle
    = \sum_{1\le k\le\ell} 
        \underbrace{\langle\alpha,\Lambda_{k}\rangle
        }_{\ge0}
        \underbrace{\langle\alpha_{k},\lambda\rangle
        }_{\ge0}
        \underbrace{\langle\alpha_{k},\Lambda_{k}\rangle
        }_{=1}
\end{align*}
cannot vanish unless 
$\langle\alpha,\Lambda_{j}\rangle=0$. \\

{\it(ii)} By symmetry, we may assume again that $w=\id$. 
By taking the inner product of $\Lambda_{j}$ with
\eqref{convex combination lambda}, we obtain
\begin{align*}
    \langle\Lambda_{j},\lambda\rangle
    =\!\sum_{1\le k \le\ell}\!
        \langle\Lambda_{j},\Lambda_{k}\rangle
        \langle\alpha_{k},\lambda\rangle
    = |\Lambda_{j}|^2 
        \underbrace{\langle\alpha_{j},\lambda\rangle
        }_{\ge \frac{1}{\ell}}
      + \sum_{k\neq j} 
        \underbrace{
        \langle\Lambda_{j},\Lambda_{k}\rangle}_{\ge0}
        \underbrace{\langle\alpha_{k},\lambda\rangle
        }_{\ge0}
    \ge \frac{1}{\ell} |\Lambda_{j}|^2 ,
\end{align*}
according to the property \eqref{property dual basis}, 
and the fact that $\langle\alpha_{j},\lambda\rangle$ 
is the largest barycentric coordinates for
$\lambda\in\mathfrak{S}_{j}$. 
\end{proof}

Now, consider the tiling of the unit sphere obtained 
by projecting \eqref{tiling S gothic}:
\begin{align*}
    S(\mathfrak{a})
    = \bigcup_{w\in{W}}\bigcup_{1\le{j}\le\ell}w.S_{j}
\end{align*}
where $S_{j}$ are the projections of the barycentric
subdivisions $\mathfrak{S}_{j}$ on the unit sphere (see Figure 2).

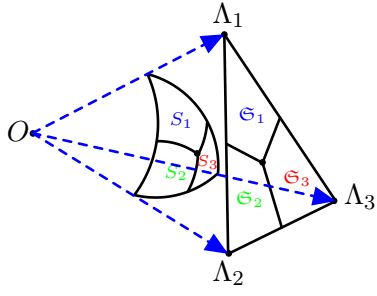
\begin{figure}[b]
    \vspace{-1cm}
    \centering
    \begin{tikzpicture}[line cap=round,line join=round,>=triangle 45,x=1.0cm,y=1.0cm, scale=0.7]
    \clip(-3.7929519970565737,-4.521248620259734) rectangle (8.648522372035192,5.214808192118163);
    \draw [line width=1.pt] (2.9507617363248198,0.15475396250089046)-- (3.6404547647133185,-0.20881041335112677);
    \draw [line width=1.pt] (3.6404547647133185,-0.20881041335112677)-- (3.973134704848195,0.6501522010641021);
    \draw [line width=1.pt] (3.6404547647133185,-0.20881041335112677)-- (4.006958916900773,-1.437959161361985);
    \draw [shift={(0.991027775282721,-0.26929961402305813)},line width=1.pt]  plot[domain=-0.07149224915632768:1.2951668884127652,variable=\t]({1.*1.809588514755417*cos(\t r)+0.*1.809588514755417*sin(\t r)},{0.*1.809588514755417*cos(\t r)+1.*1.809588514755417*sin(\t r)});
    \draw [shift={(-0.6300828027250419,0.5412556749808219)},line width=1.pt]  plot[domain=-0.635046782316099:0.4182243295792285,variable=\t]({1.*2.28608130852233*cos(\t r)+0.*2.28608130852233*sin(\t r)},{0.*2.28608130852233*cos(\t r)+1.*2.28608130852233*sin(\t r)});
    \draw [shift={(1.6114222917050398,0.8775466335624026)},line width=1.pt]  plot[domain=4.487112201170634:5.4452040821711964,variable=\t]({1.*1.7455836826856694*cos(\t r)+0.*1.7455836826856694*sin(\t r)},{0.*1.7455836826856694*cos(\t r)+1.*1.7455836826856694*sin(\t r)});
    \draw [line width=1.pt] (2.914584356650141,2.2349532937949)-- (2.987171397456459,-1.9388015525683633);
    \draw [line width=1.pt] (2.987171397456459,-1.9388015525683633)-- (5.031706380167742,-0.934680821414303);
    \draw [line width=1.pt] (5.031706380167742,-0.934680821414303)-- (2.914584356650141,2.2349532937949);
    \draw [shift={(0.8287959747820903,0.005410822034455465)},line width=1.pt]  plot[domain=5.780131647565373:6.250188889300229,variable=\t]({1.*1.5859206207166106*cos(\t r)+0.*1.5859206207166106*sin(\t r)},{0.*1.5859206207166106*cos(\t r)+1.*1.5859206207166106*sin(\t r)});
    \draw [shift={(1.1054860524878751,0.6510210033479509)},line width=1.pt]  plot[domain=5.793946023549244:6.224944672708353,variable=\t]({1.*1.4852904534760263*cos(\t r)+0.*1.4852904534760263*sin(\t r)},{0.*1.4852904534760263*cos(\t r)+1.*1.4852904534760263*sin(\t r)});
    \draw [shift={(1.658866207899445,-1.0678112975516147)},line width=1.pt]  plot[domain=0.9322906758433399:1.5898765127924257,variable=\t]({1.*1.2712695749279506*cos(\t r)+0.*1.2712695749279506*sin(\t r)},{0.*1.2712695749279506*cos(\t r)+1.*1.2712695749279506*sin(\t r)});
    \begin{scriptsize}
    \draw [fill] (-0.72,0.35) circle (1.5pt);
    \draw (-1,0.4) node {\large $O$};
    \draw [fill] (3,-1.93) circle (1.5pt);
    \draw (3,-2.3) node {\large $\Lambda_{2}$};
    \draw [fill] (2.915,2.23) circle (1.5pt);
    \draw (3,2.6) node {\large $\Lambda_{1}$};
    \draw [fill] (5,-0.93) circle (1.5pt);
    \draw (5.5,-0.85) node {\large $\Lambda_{3}$};
    \draw [color=blue] (3.45,0.7) node {$\mathfrak{S}_{1}$};
    \draw [color=green] (3.4,-0.9) node {$\mathfrak{S}_{2}$};
    \draw [color=red] (4.3,-0.5) node {$\mathfrak{S}_{3}$};
    \draw [color=blue] (2.1,0.6) node {$S_1$};
    \draw [color=green] (2,-0.4) node {$S_2$};
    \draw [color=red] (2.6,-0.2) node {\tiny $S_3$};
    \draw [fill] (3.64,-.2) circle (1.5pt);
    \draw [fill] (2.4,-.03) circle (1.5pt);
    \draw [->,line width=1.pt,dash pattern=on 3pt off 3pt,color=blue] (-0.7284086688607592,0.35576067574813774) -- (2.987171397456459,-1.938801552568363);
    \draw [->,line width=1.pt,dash pattern=on 3pt off 3pt,color=blue] (-0.7284086688607592,0.35576067574813774) -- (2.914584356650141,2.2349532937949);
    \draw [->,line width=1.pt,dash pattern=on 3pt off 3pt,color=blue] (-0.7284086688607592,0.35576067574813774) -- (5.031706380167742,-0.934680821414303);
    \end{scriptsize}
    \end{tikzpicture}
    \vspace{-1cm}
    \caption{Example of the projection in $A_3$}
    \end{figure}

We establish next a smooth version 
of the partition of unity
\begin{align*}
    \sum_{w \in W} \sum_{1\le j\le \ell}
    \mathbf{1}_{w.S_{j}} \big(
    {\textstyle \frac{\lambda}{|\lambda|}}\big)
    =1 \quad \textnormal{a.e.}.
\end{align*}
Let $\chi: \mathbb{R}\rightarrow [0,1]$ be a smooth 
cut-off function such that $\chi(r)=1$ when $r\ge0$
and $\chi(r)=0$ when $r\le -c_1$, where $c_1>0$ 
will be specified in \cref{remark constants}. 
For every $w\in W$ and $1\le j\le \ell$, we define
\begin{align*}
    \widetilde{\chi}_{w.S_{j}} (\lambda)
    =\!\prod_{1\le k\le \ell, k\neq j}\!
        \chi\Big( 
        \frac{\langle w.\alpha_{k},\lambda \rangle}
        {|\lambda|}\Big)
        \chi\Big( 
        \frac{\langle w.\alpha_{j},\lambda\rangle
        -\langle w.\alpha_{k},\lambda \rangle}
        {|\lambda|} \Big)
    \quad\forall\,\lambda\in
    \mathfrak{a}\smallsetminus\lbrace0\rbrace,
\end{align*}
and
\begin{align*}
    \widetilde{\chi}
    = \sum_{w \in W} \sum_{1\le j\le \ell}         
        \widetilde{\chi}_{w.S_{j}},
\end{align*}
which satisfy the following properties.

\begin{proposition}\label{subsection decompostion symbols}
Let $w\in W$ and $1\le j\le \ell$. For all 
$\lambda\in\mathfrak{a}\smallsetminus\lbrace0\rbrace$, we have 
\begin{enumerate}[leftmargin=*, label=(\roman*)]
    \item $\widetilde{\chi}_{w.S_{j}}(w.\lambda) 
        =\widetilde{\chi}_{S_{j}} (\lambda)$ and 
        $\widetilde{\chi}$ is $W$-invariant. 
        \vspace{5pt}
    \item $\widetilde{\chi}_{w.S_{j}}=1$ on $w.S_{j}$
        and $\widetilde{\chi} \ge 1$ on
        $\mathfrak{a}\smallsetminus\lbrace0\rbrace$.
        \vspace{5pt}
    \item $\widetilde{\chi}_{w.S_{j}}$ and
        $\widetilde{\chi}$ are homogeneous symbols 
        of order $0$.
\end{enumerate}
\end{proposition}

\begin{proof}
{\it(i)} follows immediately from the definitions.
In order to prove {\it(ii)}, we may assume that 
$w=\id$ by symmetry. For all $\lambda\in S_{j}$, 
we have
\begin{align*}
    \langle\alpha_{k},\lambda\rangle \ge 0
    \quad\textnormal{and}\quad 
    \langle\alpha_{j},\lambda\rangle 
    \ge \langle\alpha_{k},\lambda\rangle
\end{align*}
for every $1\le k\le \ell$ with $k\neq j$, hence
$\widetilde{\chi}_{S_{j}}(\lambda)=1$. 
We deduce straightforwardly that $\widetilde{\chi}\ge1$ 
on $\mathfrak{a}\smallsetminus\lbrace0\rbrace$. 
{\it(iii)} is obvious, since $\chi\big(
\frac{\langle w.\alpha_{k},\lambda\rangle}
{|\lambda|}\big)$ and 
$\chi\big(\frac{\langle w.\alpha_{j},\lambda\rangle 
-\langle w.\alpha_{k},\lambda\rangle}{|\lambda|}\big)$
are homogeneous symbols of order $0$ for all
$\lambda\in\mathfrak{a}\smallsetminus\lbrace0\rbrace$ 
and $1\le k\le \ell$.
\end{proof}

For every $w\in W$ and $1\le j\le \ell$, we set
\begin{align*}
\textstyle
    \chi_{w.S_{j}}
    = \frac{\widetilde{\chi}_{w.S_{j}}}{\widetilde{\chi}}
\end{align*}
on $\mathfrak{a}\smallsetminus\lbrace0\rbrace$. 
It follows from \cref{subsection decompostion symbols} that
$\chi_{w.S_{j}}(w.\lambda)=\chi_{S_{j}}(\lambda)$ and that
$\chi_{w.S_{j}}$ is a homogeneous symbol of order $0$. 
In particular, we have
\begin{align}\label{smooth partition}
    \sum_{w \in W}\sum_{1\le j\le \ell}
    \chi_{w.S_{j}} = 1
    \quad\textnormal{on}\quad \mathfrak{a}\smallsetminus\lbrace0\rbrace.
\end{align}
In addition, the vectors in the support of 
$\chi_{w.S_{j}}$ satisfy further properties, 
which require some preliminaries.

\begin{lemma}\label{lemma c2}
There exists $c_2>0$ such that, 
if $\lambda\in\mathfrak{a}$ satisfies
\begin{align*}
    -c_2 |\lambda| 
    \le \langle\alpha_{k},\lambda\rangle
    \le \langle\alpha_{j},\lambda\rangle 
        + c_2 |\lambda|
    \quad \forall\,k\in 
    \lbrace1,\dots,\ell\rbrace
    \smallsetminus\lbrace j\rbrace,
\end{align*}
for some $1\le j\le\ell$, then
$\langle\alpha_{j},\lambda\rangle \ge c_2 |\lambda|$.
\end{lemma}

\begin{proof}
By homogeneity, we may reduce to $|\lambda|=1$. 
Since all norms are equivalent on $\mathfrak{a}$, 
there exists $c_3>0$ such that
\begin{align}\label{equivalent norm c3}
    \sum_{1\le{k}\le\ell} 
    |\langle\alpha_{k},\lambda\rangle| \ge c_3
    \quad\forall\,\lambda\in{S(\mathfrak{a})}.
\end{align}
Set $c_2 = \frac{c_3}{2\ell}$. On the one hand, if 
\begin{align*}
    -c_{2}\le\langle\alpha_{k},\lambda\rangle\le2c_{2}
    \quad\forall\,{k}\in\lbrace1,\dots,\ell\rbrace 
    \smallsetminus\lbrace{j}\rbrace,
\end{align*}
then $\langle\alpha_{j},\lambda\rangle\ge 2c_2$. 
Otherwise,
\begin{align*}
    \sum_{1\le j\le \ell} |\langle\alpha_{k},\lambda\rangle|
    = \underbrace{
        |\langle\alpha_{j},\lambda\rangle|
        }_{<2c_2}
    +\sum_{k\neq j}         
        \underbrace{
        |\langle\alpha_{k},\lambda\rangle|}_{\le 2c_2}
    < 2\ell c_2 = c_3,
\end{align*}
which contradicts \eqref{equivalent norm c3}. 
On the other hand, if
\begin{align*}
    2c_2 
    \le \langle\alpha_{k},\lambda\rangle 
    \le \langle\alpha_{j},\lambda\rangle + c_2
\end{align*}
for some $k\in\lbrace 1,\dots,\ell\rbrace
\smallsetminus \lbrace j\rbrace$, then
$\langle\alpha_{j},\lambda\rangle \ge c_2$ is obvious.
\end{proof}

\begin{remark}\label{remark constants}
We clarify in this remark all constants appearing
in this subsection. Denote by $L_1$ the highest root 
length and by $L_2$ the sum of lengths of the dual basis
\begin{align*}
    L_1 
    = \max_{\alpha\in\Sigma^{+}} \sum_{1\le k\le \ell}
        \langle\alpha,\Lambda_{k}\rangle
    \quad\textnormal{and}\quad
    L_2 = \sum_{1\le k\le \ell} |\Lambda_{k}|.
\end{align*}
In addition, we denote by $M_1$ and $M_2$ the shortest
and the longest generators
\begin{align*}
    M_1 = \min_{1\le k\le \ell} |\Lambda_{k}|
    \quad\textnormal{and}\quad
    M_2 = \max_{1\le k\le \ell} |\Lambda_{k}|.
\end{align*}
Then we choose $c_1>0$ such that $c_1<c_2\min\lbrace
\frac{1}{L_1},\frac{M_{1}^{2}}{M_2L_2}\rbrace$, where
$c_2=\frac{c_3}{2\ell}$ with $c_3$ defined in
\eqref{equivalent norm c3}. Let $c_4=c_2 - L_1 c_1$ and
$c_5 = M_{1}^{2} c_2 - M_2 L_2 c_1$. 
Notice that $L_{1}\in\mathbb{N}^{*}$, $c_1<c_2$, $c_4>0$ 
and $c_5>0$. All these constants depend only on the
geometric structure of the roots system corresponding 
to $\mathbb{X}$.
\end{remark}

The following result is an analog of 
\cref{properties on S gothic} for the wider regions
$\supp\chi_{\omega.S_{j}}$.
\begin{proposition}\label{properties on Sj}
Let $w\in W$ and $1\le j\le \ell$. Then 
\begin{enumerate}[leftmargin=*,label=(\roman*)]
    \item a root $\alpha\in\Sigma$ satisfies either     
        $\langle\alpha,w.\Lambda_{j}\rangle=0$ or
        \begin{align}\label{estimate c4}
            |\langle\alpha,\lambda\rangle| 
            \ge c_4 |\lambda|
            \quad\forall\,
            \lambda\in\supp\chi_{w.S_{j}},
        \end{align}
    \item $|\langle{w}.\Lambda_{j},\lambda\rangle| 
    \ge c_5 |\lambda|$ for every
    $\lambda\in\supp\chi_{w.S_{j}}$.
\end{enumerate}
\end{proposition}

\begin{proof}
{\it(i)} By symmetry, we may assume that $w=\id$ and 
that $\alpha$ is a positive root. 
Notice that $\langle\alpha,\Lambda_{j}\rangle$ is a
nonnegative integer, suppose that
$\langle\alpha,\Lambda_{j}\rangle >0$ and let us prove
\eqref{estimate c4}.
As
\begin{align*}
    -c_1|\lambda| 
    \le \langle\alpha_{k},\lambda\rangle
    \le \langle\alpha_{j},\lambda\rangle 
        + c_{1} |\lambda|
    \quad \forall\,\lambda\in\supp\chi_{S_{j}},\
    \forall\,k\in \lbrace1,\dots,\ell\rbrace
    \smallsetminus\lbrace{j}\rbrace,
\end{align*}
we have indeed
\begin{align*}
    \langle\alpha,\lambda\rangle
    =&\ \sum_{1\le k\le\ell} 
        \langle\alpha,\Lambda_{k}\rangle
        \langle\alpha_{k},\lambda\rangle \\[5pt]
    =&\ \underbrace{
        \langle\alpha,\Lambda_{j}\rangle}_{\ge1}
        \underbrace{
        \langle\alpha_{j},\lambda\rangle
        }_{\ge c_2|\lambda|}
        + \sum_{k\neq j}
        \langle\alpha,\Lambda_{k}\rangle
        \underbrace{\langle\alpha_{k},\lambda\rangle
        }_{\ge -c_1|\lambda|}
    \ge (c_{2}-L_1 c_{1})|\lambda|=c_{4}|\lambda|,
\end{align*}
according to \cref{lemma c2} since $c_1<c_2$. \\

{\it(ii)} By symmetry, we assume again $w=\id$. 
By taking the inner product of $\Lambda_{j}$ with
\eqref{convex combination lambda}, we obtain, for 
every $\lambda\in\supp\chi_{S_{j}}$,
\begin{align*}
    \langle\Lambda_{j},\lambda\rangle
    &=\sum_{1\le k \le\ell}
        \langle\Lambda_{j},\Lambda_{k}\rangle
        \langle\alpha_{k},\lambda\rangle
    =\underbrace{|\Lambda_{j}|^2}_{\ge M_{1}^{2}}
        \underbrace{
        \langle\alpha_{j},\lambda\rangle
        }_{\ge c_2|\lambda|}
      + \sum_{k\neq j} 
        \underbrace{
        \langle\Lambda_{j},\Lambda_{k}\rangle
        }_{\le |\Lambda_{j}||\Lambda_{k}|}
        \underbrace{
        \langle\alpha_{k},\lambda\rangle
        }_{\ge -c_1|\lambda|}\\
    &\ge (M_{1}^{2} c_2 - M_2 L_2 c_1) |\lambda|
    = c_5 |\lambda|.
\end{align*}
\end{proof}

\begin{remark}
The partition of unity \eqref{smooth partition} plays an important role
in the kernel analysis carried out in \cref{section kernel estiamtes}.
It allows us to overcome a well-known problem in spherical Fourier 
analysis in higher rank, namely the fact that the Plancherel density
is not a symbol in general. This new tool should certainly help
solving other problems.
\end{remark}

\section{Pointwise estimates of the wave kernel}
\label{section kernel estiamtes}

In this section, we derive pointwise estimates 
for the $K$-bi-invariant convolution kernel 
$\omega_{t}^{\sigma}$ of the operator 
$W_{t}^{\sigma}=(-\Delta)^{-\frac{\sigma}{2}} 
e^{it \sqrt{- \Delta}}$ on the symmetric space $\mathbb{X}$:
\begin{align*}
    W_{t}^{\sigma}  f(x) 
    = f*\omega_{t}^{\sigma}(x) 
    = \int_{G}dy\ \omega_{t}^{\sigma} (y^{-1}x) f(y)
\end{align*}
for suitable exponents $\sigma \in \mathbb{C}$. 
By using the inverse formula of the spherical Fourier transform,
we have
\begin{align*}
    \omega_{t}^{\sigma} (x) 
    =\, C_0 \int_{\mathfrak{a}}d\lambda\ 
        |\mathbf{c}(\lambda)|^{-2} \varphi_{\lambda}(x) 
        (|\lambda |^2+|\rho|^2)^{-\frac{\sigma}{2}} 
        e^{it\sqrt{|\lambda|^2+|\rho|^2}}
\end{align*}
Let us point out that the analysis of this oscillatory integral 
carried out on hyperbolic spaces or on symmetric spaces 
$G/K$ with $G$ complex (see \cite{AnPi2014,Zha2021}) 
does not hold in general since the Plancherel density
$|\mathbf{c}(\lambda)|^{-2}$ is no more a differential symbol. 
We write 
\begin{align*}
    \omega_{t}^{\sigma} (x) 
    =\, \frac{1}{\Gamma(\sigma)} 
        \int_{0}^{+\infty}ds\  s^{\sigma-1} 
        \underbrace{ 
        C_0 \int_{\mathfrak{a}}d\lambda\ 
        |\mathbf{c}(\lambda)|^{-2} \varphi_{\lambda}(x) 
        e^{-(s-it) \sqrt{|\lambda|^2+|\rho|^2}}
        }_{p_{s-it}(x)}.
\end{align*}
according to the formula
\begin{align*}
    r^{-\sigma} 
    = \frac{1}{\Gamma(\sigma)} 
    \int_{0}^{+ \infty}\frac{ds}{s}\ 
    s^{\sigma} e^{-sr}
    \quad\forall\,r>0.
\end{align*}
Here $p_{s-it}$ is the $K$-bi-invariant convolution kernel
of the Poisson operator $\mathcal{P}_{s-it}$. 
Let us split up $\omega_{t}^{\sigma}(x) 
=\omega_{t}^{\sigma,0}(x)+\omega_{t}^{\sigma,\infty}(x)$ with
\begin{align*}
    \omega_{t}^{\sigma,0}(x) 
    = \frac{1}{\Gamma(\sigma)} 
        \int_{0}^{1}ds\ s^{\sigma-1} p_{s-it}(x)
\end{align*}
and
\begin{align*}
    \omega_{t}^{\sigma,\infty}(x) 
    = \frac{1}{\Gamma(\sigma)} 
        \int_{1}^{+\infty}ds\ s^{\sigma-1} p_{s-it}(x).
\end{align*}

We shall see in \cref{subsection 2} that the kernel
$\omega_{t}^{\sigma,0} (x)$ has a logarithmic singularity
on the sphere $|x|=t$ when $\sigma = \frac{d+1}{2}$. 
To bypass this problem, we consider the analytic family 
of operators
\begin{align}\label{analytic family tilde W}
    \widetilde{W}_{t}^{\sigma,0}
    =  \underbrace{
        \frac{e^{\sigma^2}}
        {\Gamma(\frac{d+1}{2}-\sigma)\Gamma(\sigma)}
        }_{C_{\sigma,d}}
        \int_{0}^{1}ds\ s^{\sigma-1} \mathcal{P}_{s-it}
\end{align}
in the vertical strip $0\le\re\sigma\le\frac{d+1}{2}$ 
and the corresponding kernels
\begin{align*}
    \widetilde{\omega}_{t}^{\sigma,0}(x)
    = C_{\sigma,d} \int_{0}^{1}ds\ 
        s^{\sigma-1} p_{s-it}(x)
    \quad \forall\,x\in\mathbb{X}.
\end{align*}
Notice that the Gamma function $\Gamma(\frac{d+1}{2}-\sigma)$ 
allows us to deal with the boundary point 
$\sigma=\frac{d+1}{2}$, while the exponential function 
ensures boundedness at infinity in the vertical strip.
More precisely, by using the inequality 
\begin{align*}
    |\Gamma(z)|
    \ge\Gamma(\re{z})\,
    \big(\cosh(\pi\im{z})\big)^{-\frac{1}{2}}
    \quad\forall\,z\in\mathbb{C}
    \,\,\textnormal{with}\,\,\re{z}
    \ge\textstyle{\frac{1}{2}}
\end{align*}
(see for instance \cite[Eq.\,5.6.7]{DLMF}), we can estimate
\begin{align}\label{constant C sigma d}
\textstyle
    |C_{\sigma,d}|
    \lesssim |\sigma|\,|\sigma-\frac{d+1}{2}|\,
        e^{\pi\,|\im\sigma|-(\im\sigma)^2}
\end{align}
for all $\sigma\in\mathbb{C}$ with
$0\le\re\sigma\le\frac{d+1}{2}$.

We divide the argument into three parts depending whether
$|t|$ and $\frac{|x|}{|t|}$ are small or large. 
When $|t|$ is large but $\frac{|x|}{|t|}$ is sufficiently
small, we estimate $\widetilde{\omega}_{t}^{\sigma,0}$ 
in \cref{theorem subsection 1} by combining the method 
of stationary phase with our barycentric decomposition 
of Weyl chambers; when $\frac{|x|}{|t|}$ is large, we
estimate $\widetilde{\omega}_{t}^{\sigma,0}$ in
\cref{theorem subsection 2} by using the Hadamard
parametrix along the lines of \cite{CGM2001};
$\omega_{t}^{\sigma,\infty}(x)$ is easily handled by a
standard stationary phase argument, 
see \cref{theorem subsection 3}.
\vspace{5pt}

\subsection{Estimates of 
$\widetilde{\omega}_{t}^{\sigma,0}(x)$ when $|t|$ is 
large and $\frac{|x|}{|t|}$ is sufficiently small.}
\label{subsection 1}

According to the integral expression 
\eqref{integral expression of phi lambda} of the 
spherical functions, we write
\begin{align*}
    \widetilde{\omega}_{t}^{\sigma,0}(x)
    = C_{\sigma,d}\,C_0 
      \int_{K}dk\,e^{\langle\rho,A(kx)\rangle}
      \int_{0}^{1}ds\,s^{\sigma-1} I(s,t,x),
\end{align*}
where
\begin{align*}
    I (s,t,x)
    = \int_{\mathfrak{a}} d\lambda\     
        |\mathbf{c}(\lambda)|^{-2}
        e^{-s\sqrt{|\lambda|^2+|\rho|^2}}
        e^{it\psi_{t}(\lambda)}
\end{align*}
is an oscillatory integral with phase
\begin{align}\label{phase}
\textstyle
    \psi_{t}(\lambda) 
    = \sqrt{|\lambda|^2+|\rho|^2} 
    + \langle\frac{A(kx)}{t},\lambda\rangle.
\end{align}
Let us split up 
\begin{align*}
    I(s,t,x)
    = I^{-}(s,t,x) + I^{+}(s,t,x)
    = \int_{\mathfrak{a}}d\lambda\ 
        \chi_{0}^{\rho}(\lambda) \cdots
        \ +\ \int_{\mathfrak{a}}d\lambda\ 
        \chi_{\infty}^{\rho}(\lambda) \cdots
\end{align*}
by using smooth radial cut-off functions $\chi_{0}^{\rho}$
and $\chi_{\infty}^{\rho}=1-\chi_{0}^{\rho}$, where
$\chi_{0}^{\rho}(\lambda)$ equals $1$ when
$|\lambda|\le|\rho|$ and vanishes if 
$|\lambda|\ge2|\rho|$. 
Then we have the following estimates for $I^{-}$ 
and $I^{+}$.
\begin{proposition}\label{proposition I0- and I0+}
There exists $0<C_{\Sigma}\le\frac{1}{2}$ such that 
the following estimates hold when 
$0<s<1$, $|t|\ge1$ and
$\frac{|x|}{|t|} \le C_{\Sigma}$:
\begin{align}\label{estimate of I0-}
    |I^{-} (s,t,x)| 
    \lesssim |t|^{-\frac{D}{2}}\,(1+|x|)^{\frac{D-\ell}{2}},
\end{align}
and
\begin{align}\label{estimate of I0+}
    |I^{+} (s,t,x)| \lesssim |t|^{-N},
\end{align}
for every $N \in \mathbb{N}$.
\end{proposition}

\begin{remark}
$C_{\Sigma}$ is a small constant depending only on 
the geometric structure of the root system $\Sigma$,  
which will be specified later in the proof of
\eqref{estimate of I0+}. 
Notice that the upper bounds \eqref{estimate of I0-} 
of $I^{-}$ and \eqref{estimate of I0+} of $I^{+}$ 
hold uniformly in $s\in(0,1)$.
\end{remark}

\begin{proof}
[Proof of the estimate \eqref{estimate of I0-}.]
Recall that
\begin{align*}
    I^{-} (s,t,x)
    = \int_{\mathfrak{a}} d\lambda\ 
        a_{0}(s,\lambda)
        e^{it\psi_{t}(\lambda)}
\end{align*}
is an oscillatory integral with amplitude
\begin{align*}
    a_{0}(s,\lambda) 
        = \chi_{0}^{\rho}(\lambda) 
        |\mathbf{c}(\lambda)|^{-2}
        e^{-s\sqrt{|\lambda|^2+|\rho|^2}}
\end{align*}
and phase $\psi_{t}(\lambda)$ which is defined by \eqref{phase}.
The amplitude $a_{0}(s,\lambda)$ is compactly 
supported for $|\lambda|\le2|\rho|$, and the phase 
$\psi_{t}$ has, in the support of $\chi_{0}^{\rho}$, 
a single nondegenerate critical point $\lambda_{0}$ 
which is given by
\begin{align}\label{critical point A/t}
\textstyle
    (|\lambda_{0}|^2+|\rho|^2)^{-\frac{1}{2}} 
    \lambda_{0} = -\frac{A}{t}
\end{align}
where $A=A(kx)$, and which satisfies
\begin{align}\label{critical point lambda 0}
\textstyle
    |\lambda_{0}| 
    = |\rho| \frac{|A|}{|t|}
        (1-\frac{|A|^2}{t^2})^{-\frac{1}{2}}
    \le |\rho| \frac{|x|}{t}
        (1-\frac{|x|^2}{t^2})^{-\frac{1}{2}}
    < \frac{|\rho|}{\sqrt{3}},
\end{align}
as $|A|\le|x|$ and $\frac{|x|}{t} \le C_{\Sigma} < \frac{1}{2}$. 
We conclude by resuming straightforwardly the computations 
carried out in the proof of \cite[Theorem 3.1]{Zha2021}.
For the sake of completeness and for the reader's convenience,
we include a detailed study of the oscillatory integral $I^{-}$ 
in \cref{Appendix A} (see \cref{Appendix A-lemma}).
\end{proof}

Let us turn to the oscillatory integral
\begin{align*}
    I^{+} (s,t,x)
    = \int_{\mathfrak{a}} d\lambda\ 
        \chi_{\infty}^{\rho}(\lambda) 
        |\mathbf{c}(\lambda)|^{-2}
        e^{-s\sqrt{|\lambda|^2+|\rho|^2}}
        e^{it\psi_{t}(\lambda)},
\end{align*}
which vanishes unless $|\lambda|>|\rho|$. According to
\eqref{critical point lambda 0}, $\psi_{t}$ has no critical
point in the support of $\chi_{\infty}^{\rho}$. 
In rank one or in higher rank with $G$ complex, one can
handle this integral by performing several integrations by parts. 
This approach fails in general since 
the Plancherel density $|\mathbf{c}(\lambda)|^{-2}$ 
is not a differential symbol.
To get around this problem, we split up the Weyl chamber
according to the barycentric decomposition carried out in
\cref{subsection decomposition}, and perform integrations
by parts based along a well chosen directional derivative in
each component.

\begin{proof}
[Proof of the estimate \eqref{estimate of I0+}.]
According to the partition of unity 
\eqref{smooth partition}, we split up
\begin{align*}
    I^{+} (s,t,x)
    =  \sum_{w\in W} \sum_{1\le j\le \ell}
        I_{w.S_{j}} (i\tau,x)
\end{align*}
with $\tau=s-it$, and we estimate
\begin{align}\label{I wSj}
    I_{w.S_{j}} (i\tau,x)
    = \int_{\mathfrak{a}} d\lambda\, 
      \chi_{w.S_{j}}(\lambda)\, 
      \chi_{\infty}^{\rho}(\lambda)\, 
      |\mathbf{c}(\lambda)|^{-2}\, 
      e^{-\tau\psi_{i\tau}(\lambda)}
\end{align}
by performing integrations by parts based on
\begin{align}\label{IBP}
    e^{-\tau\psi_{i\tau}(\lambda)}
    = \textstyle{-\frac{1}{\tau} 
    \frac{1}{\partial_{w.\Lambda_{j}} 
    \psi_{i\tau}(\lambda)}\,\partial_{w.\Lambda_{j}}} 
    e^{-\tau\psi_{i\tau} (\lambda)}.
\end{align}
Notice that 
\begin{align*}
\textstyle
    \partial_{w.\Lambda_{j}} \psi_{i\tau}(\lambda)
    = \langle w.\Lambda_{j},
        \frac{\lambda}{\sqrt{|\lambda|^2+|\rho|^2}}
        -i \frac{A(kx)}{\tau} \rangle
\end{align*}
is a symbol of order $0$, which satisfies in addition
\begin{align*}
    |\partial_{w.\Lambda_{j}} \psi_{i\tau}(\lambda)|
    \ge&\ \textstyle{
        \frac{|\langle{w}.\Lambda_{j},\lambda\rangle|}
        {\sqrt{|\lambda|^2+|\rho|^2}}
        - \big|\langle w.\Lambda_{j},\frac{A(kx)}{\tau}
        \rangle \big|} \\
    \ge&\ \textstyle{c_5     
        \frac{|\lambda|}{\sqrt{|\lambda|^2+|\rho|^2}}
        - |\Lambda_{j}| \frac{|A(kx)|}{|\tau|}
    \ge \frac{c_5}{\sqrt{2}}
        - M_2 \frac{|x|}{|t|}}
\end{align*}
on $(\supp \chi_{w.S_{j}})\cap(\supp\chi_{\infty}^{\rho})$
according to Proposition \ref{properties on Sj}, where
the constants $c_{5}$ and $M_2$ are specified in
\cref{remark constants}. 
By choosing 
$C_{\Sigma}=\min\lbrace\frac{c_5}{2M_2},\frac{1}{2}\rbrace$
, we obtain
\begin{align*}
\textstyle
    |\partial_{w.\Lambda_{j}} \psi_{i\tau}(\lambda)| 
    \ge \frac{\sqrt{2}-1}{2} c_5 > 0.
\end{align*}
Let us return to \eqref{I wSj}, which becomes
\begin{align*}
    I_{w.S_{j}} (i\tau,x)
    = {\textstyle \tau^{-N}} 
        \int_{\mathfrak{a}} & d\lambda\ 
        e^{-\tau\psi_{i\tau}(\lambda)} \\
    & \times {\textstyle
        \big\lbrace \partial_{w.\Lambda_{j}} 
        \circ \frac{1}{\partial_{w.\Lambda_{j}} 
        \psi_{i\tau}(\lambda)} \big\rbrace^{N}
        \big\lbrace \chi_{w.S_{j}}(\lambda) 
        \chi_{\infty}^{\rho}(\lambda) 
        |\mathbf{c}(\lambda)|^{-2} \big\rbrace},
\end{align*}
after $N$ integrations by parts based on \eqref{IBP}. 
If some derivatives hit $\chi_{\infty}^{\rho}(\lambda)$,
the above integral is reduced to the spherical shell
$|\lambda|\asymp|\rho|$ and thus converges. 
Assume that no derivative is applied to
$\chi_{\infty}^{\rho}(\lambda)$ and that
\begin{itemize}[leftmargin=*]\vspace{5pt}
    \item $N_1$ derivatives are applied to the factors
        $\frac{1}{\partial_{w.\Lambda_{j}}
        \psi_{i\tau}(\lambda)}$,\vspace{5pt}
    \item $N_2$ derivatives are applied to 
        $\chi_{w.S_{j}}(\lambda)$,\vspace{5pt}
    \item $N_3$ derivatives are applied to
        $|\mathbf{c}(\lambda)|^{-2}$,\vspace{5pt}
\end{itemize}
with $N=N_1+N_2+N_3$. The contribution of the first 
item is $O(|\lambda|^{-N_1})$, as 
$\partial_{w.\Lambda_{j}} \psi_{i\tau}(\lambda)$ is a 
symbol or order $0$, which stays away from $0$. 
Similarly, the contribution of the second item is
$O(|\lambda|^{-N_2})$, as $\chi_{w.S_{j}}(\lambda)$ 
is a symbol of order $0$ according to 
\cref{subsection decompostion symbols}. 
As far as the third item is concerned, the derivatives
$(\partial_{w.\Lambda_{j}})^{N_{3}}$ are applied to the
various factors in \eqref{Plancherel density}. 
According to \cref{properties on Sj}, for every $\lambda$
in the support of $\chi_{w.S_{j}}$, any root
$\alpha\in\Sigma$ satisfies either
$\langle\alpha,w.\Lambda_{j}\rangle=0$ or
$|\langle\alpha,\lambda\rangle|\gtrsim|\lambda|$. 
On the one hand, if $\langle\alpha,w.\Lambda_{j}\rangle=0$,
all derivatives
\begin{align*}
    (\partial_{w.\Lambda_{j}})^{N_{\alpha}}
    |\mathbf{c}_{\alpha}
    (\langle\alpha,\lambda\rangle)|^{-2}
    \quad\forall\,N_{\alpha}\in\mathbb{N}^{*}
\end{align*}
vanish. 
On the other hand, if $\langle\alpha,w.\Lambda_{j}\rangle\neq0$, 
we use the fact that $|\mathbf{c}_{\alpha}|^{-2}$ is a symbol 
on $\mathbb{R}$ of order $m_{\alpha}+m_{2\alpha}$, 
together with \eqref{estimate c4}, in order to estimate
\begin{align*}
    \big| (\partial_{w.\Lambda_{j}})^{N_{\alpha}}
    |\mathbf{c}_{\alpha}
    (\langle\alpha,\lambda\rangle)|^{-2} \big|
    \lesssim 
        |\langle\alpha,\lambda\rangle
        |^{m_{\alpha}+m_{2\alpha}-N_{\alpha}}
    \asymp
        |\lambda|^{m_{\alpha}+m_{2\alpha}-N_{\alpha}}
    \quad\forall\,N_{\alpha}\in\mathbb{N}
\end{align*}
for $\lambda\in(\supp\chi_{w.S_{j}})
\cap(\supp\chi_{\infty}^{\rho})$. 
Hence
\begin{align*}
    (\partial_{w.\Lambda_{j}})^{N_{3}} 
    |\mathbf{c} (\lambda)|^{-2}
    = O(|\lambda|^{d-\ell-N_{3}})
    \quad \forall\,\lambda\in (\supp\chi_{w.S_{j}})
    \cap(\supp\chi_{\infty}^{\rho}).
\end{align*}
In conclusion,
\begin{align*}
    |I_{w.S_{j}} (i\tau,x)| 
    \lesssim |
    \tau|^{-N} \int_{\mathfrak{a}}d\lambda\
        |\lambda|^{d-\ell-N_{1}-N_{2}-N_{3}}
    \lesssim |t|^{-N}
\end{align*}
provided that $N>d$, and consequently 
\begin{align*}
    I^{+}(s,t,x) =O(|t|^{-N}).
\end{align*}
\end{proof}

We deduce from \eqref{estimate of I0-} and 
\eqref{estimate of I0+} that, for all $0<s<1$,
$|t|\ge1$ and $\frac{|x|}{|t|}\le{C}_{\Sigma}$,
\begin{align}\label{estimate of I}
    |I(s,t,x)|
    \lesssim |t|^{-\frac{D}{2}}\,(1+|x|)^{\frac{D-\ell}{2}}
\end{align}
uniformly in $s$. 
Notice that
\begin{align*}
    \frac{\partial}{\partial{s}} I(s,t,x)
    = -\int_{\mathfrak{a}} d\lambda\,     
        |\mathbf{c}(\lambda)|^{-2}\,
        \sqrt{|\lambda|^2+|\rho|^2}\,
        e^{-s\sqrt{|\lambda|^2+|\rho|^2}}\,
        e^{it\psi_{t}(\lambda)}
\end{align*}
has the same phase as $I(s,t,x)$. 
Hence the estimate $\eqref{estimate of I}$ holds for
$\frac{\partial}{\partial{s}} I(s,t,x)$ by similar computations.
Since
\begin{align*}
    \int_{0}^{1}ds\,s^{\sigma-1} I(s,t,x)
    = \big[{\textstyle\frac{1}{\sigma}}\,
        s^{\sigma}\,I(s,t,x)\big]_{0}^{1}
    - {\textstyle\frac{1}{\sigma}}
        \int_{0}^{1}ds\,s^{\sigma}\,
        {\textstyle\frac{\partial}{\partial{s}}}\,
        I(s,t,x),
\end{align*}
we deduce that 
\begin{align*}
    \Big| C_{\sigma,d}
        \int_{0}^{1}ds\,s^{\sigma-1} I(s,t,x) \Big|
    \lesssim |t|^{-\frac{D}{2}}\,(1+|x|)^{\frac{D-\ell}{2}}
\end{align*}
according to \eqref{constant C sigma d}.
Then we obtain the following kernel estimate of
$\widetilde{\omega}_{t}^{\sigma,0}$.

\begin{theorem}\label{theorem subsection 1}
There exists $0<C_{\Sigma}\le\frac{1}{2}$ such that 
the following estimate holds, when $|t|\ge1$ and
$\frac{|x|}{|t|} \le C_{\Sigma}$, uniformly in the 
vertical strip $0\le\re\sigma\le\frac{d+1}{2}$:
\begin{align}
    |\widetilde{\omega}_{t}^{\sigma,0}(x)|
    \lesssim |t|^{-\frac{D}{2}} (1+|x|)^{\frac{D-\ell}{2}}
    \varphi_{0}(x).
\end{align}
\end{theorem}


\subsection{Estimates of 
$\widetilde{\omega}_{t}^{\sigma,0}(x)$ 
in the remaining range}\label{subsection 2}

Recall that $\tau=s-it$ with $t\in\mathbb{R}^{*}$
and $s\in(0,1)$ throughout this subsection.
We are looking for pointwise estimates of
\begin{align*}
    \widetilde{\omega}_{t}^{\sigma,0}(x)
    = C_{\sigma,d} \int_{0}^{1} ds\ s^{\sigma-1}
        \underbrace{
            C_0 \int_{\mathfrak{a}} d\lambda\ 
            |\mathbf{c}(\lambda)|^{-2} \varphi_{\lambda}(x)
            \overbrace{
                e^{-\tau \sqrt{|\lambda|^2+|\rho|^2}}
            }^{\widetilde{p}_{\tau}(\lambda)}
        }_{p_{\tau}(x)}
    \quad \forall\, x\in\mathbb{X},
\end{align*}
where $p_{\tau}(x)$ is the Poisson kernel and 
$\widetilde{p}_{\tau}(\lambda)$ denotes its spherical Fourier
transform. This subsection focuses on pointwise estimates of
$p_{\tau}$ along the lines of \cite[pp.1054-1063]{CGM2001}.

\begin{remark}
Notice that the exponential factor ensures the convergence of 
the integral defining $p_{\tau}$, but yields a large negative 
power $s^{-d}$. Then $\widetilde{\omega}_{t}^{\sigma,0}$ 
converges under the strong smoothness assumption $\re\sigma\ge{d}$. 
We will sharpen it to $\re\sigma=\frac{d+1}{2}$. 
Notice that the stationary phase method carried out in the
previous subsection fails since the critical point can be 
very large when $\frac{|x|}{|t|}$ is not bounded from above.
\end{remark}

As in \cite{CGM2001}, let us denote by
$p_{\tau}^{\mathbb{R}}(v) = \frac{\tau}{\pi(\tau^2+v^2)}$ 
the Poisson kernel on $\mathbb{R}$. We may write
\begin{align*}
    \widetilde{p}_{\tau}(\lambda)
    = e^{-\tau\sqrt{|\lambda|^2+|\rho|^2}}
    = \int_{\mathbb{R}} dv\ p_{\tau}^{\mathbb{R}}(v)
        \cos(v\sqrt{|\lambda|^2+|\rho|^2})
    \quad \forall\lambda\in\mathfrak{a}.
\end{align*}
Consider a smooth even cut-off function 
$\chi: \mathbb{R}\rightarrow[0,1]$, which is supported in
$[-\sqrt{2},\sqrt{2}]$, and equals $1$ on $[-1,1]$. 
Denote by $\chi_{T}=\chi(\frac{\cdot}{2T})$ with 
$T=\sqrt{2}$ (uniformly in $t$) when $|t|\le1$ or $T=\sqrt{2}|t|$ when $|t|\ge1$. 
Then $\chi_{T}$ is supported in 
$[-2\sqrt{2}T,2\sqrt{2}T]\subset(-3T,3T)$. 
We denote by $a_{\tau}$ and $b_{\tau}$ the $K$-bi-invariant 
kernels of operators
\begin{align*}
    A_{\tau}
    = \int_{-\infty}^{+\infty} dv\,
        \chi_{T}(v)\,
        p_{\tau}^{\mathbb{R}}(v)\,
        \cos(v\sqrt{-\Delta})
\end{align*}
and
\begin{align*}
    B_{\tau}
    = \int_{-\infty}^{+\infty} dv\, 
        \lbrace1-\chi_{T}(v)\rbrace\,
        p_{\tau}^{\mathbb{R}}(v)\,
        \cos(v\sqrt{-\Delta}).
\end{align*}
Then $p_{\tau}=a_{\tau}+b_{\tau}$ and $a_{\tau}$ is 
supported in a ball of radius $3T$ in $\mathbb{X}$ by 
finite propagation speed. $b_{\tau}$ is easily estimated 
by straightforward computations, see \cref{b tau}. 
In order to analyze $a_{\tau}$, we will expand
$\cos(v\sqrt{-\Delta})$ by using the Hadamard parametrix.

Let $\lbrace{R_{+}^{z}\,|\,z\in\mathbb{C}}\rbrace$ be the 
analytic family of Riesz distributions on $\mathbb{R}$ 
defined by
\begin{align*}
    R_{+}^{z}(r)
    = \begin{cases}
        \Gamma (z)^{-1} r^{z-1}
        &\textnormal{if \,} r>0, \\[5pt]
        0 &\textnormal{if \,} r\le0,
    \end{cases}
\end{align*}
for $\re{z}>0$. 
The $K$-bi-invariant convolution kernel $\Phi_{v}$ 
of the operator $\cos(v\sqrt{-\Delta})$ has the asymptotic expansion
\begin{equation}\label{Expansion}\begin{aligned}
    \Phi_{v}(\exp H)
    &= J(H)^{-\frac{1}{2}}
        \sum_{k=0}^{[d/2]} 4^{-k}\,|v|\,U_{k}(H)\,
        R_{+}^{k-\frac{d-1}{2}}(v^2-|H|^2)\\
    &+ E_{\Phi}(v,H)
\end{aligned}\end{equation}
where
\begin{align*}
    J(H) = \prod_{\alpha\in\Sigma^{+}}
        \Big(\frac{\sinh\langle\alpha,H\rangle}
        {\langle\alpha,H\rangle}\Big)^{m_{\alpha}}
\end{align*}
denotes the Jacobian of the exponential map from $\mathfrak{p}$
equipped with Lebesgue measure to $\mathbb{X}$ equipped with
Riemannian measure. 
Moreover, the coefficients satisfy
\begin{align}\label{Estimates U}
    \nablaP^{n} U_{k}= O(1)
\end{align}
for every $ k,n\in\mathbb{N}$,
and the remainder is estimated as
\begin{align}\label{Estimates E Phi}
    |E_{\Phi}(v,H)| 
    \lesssim (1+v)^{3(\frac{d}{2}+1)} 
        e^{-\langle\rho,H\rangle}.
\end{align}
The Hadamard parametrix has been described and applied in various
settings, see for instance \cite{Ber1977,Hor1994,CGM2001}. 
For the reader's convenience, we give in \cref{Appendix B}
some details about this construction in the particular case of 
noncompact symmetric spaces.
By resuming the proof of Lemma 3.3 in \cite{CGM2001}
(see \cref{Appendix C} for details), we deduce the following
expansion of the $K$-bi-invariant convolution kernel $a_{\tau}$
of the operator $A_{\tau}$:
\begin{align}\label{a tau expansion}
    a_{\tau}(\exp H) 
    =& \frac{\tau}{\pi}\,J(H)^{-\frac{1}{2}}\,
        \sum_{k=0}^{[d/2]} 4^{-k}\, U_{k}(H)\,
        {\textstyle\Gamma\big(\frac{d+1}{2}-k\big)}\,
        (|H|^2+\tau^2)^{k-\frac{d+1}{2}}
        \nonumber\\
    +& E(\tau,H)
\end{align}
where 
\begin{align}\label{error part}
    |E(\tau,H)| 
    \lesssim T^{3(\frac{d}{2}+1)}\,
    (\log{T}-\log{s})\,
    e^{-\langle\rho,H\rangle}
    \quad\forall\,H\in\frakACL.
\end{align}

\begin{remark}
As a consequence, we may deduce that
\begin{align*}
    |a_{\tau}(\exp H)| 
    \lesssim 
        s^{-\frac{d+1}{2}}\,e^{-\langle\rho,H\rangle}
        \begin{cases}
            |t|^{-\frac{d-1}{2}} 
            \quad& \textnormal{if $|t|$ is small},\\[5pt]
            |t|^{3(\frac{d}{2}+1)}\,\log{|t|}
            \quad& \textnormal{if $|t|$ is large},
        \end{cases}
\end{align*}
for all $H\in\frakACL$. However, we cannot apply
straightforwardly such estimates to study the kernel
$\omega_{t}^{\sigma}$, since it kills the imaginary part 
of $\sigma$ and yields a logarithmic singularity on the 
sphere $|x|=t$ when $\sigma\in\mathbb{C}$ with
$\re\sigma=\frac{d+1}{2}$.
\end{remark}

The following proposition concerning the estimate of 
$b_{\tau}$ will be proved by straightforward computations.
\begin{proposition}\label{b tau}
Let $N>d$ be an even integer. Then 
\begin{align}\label{b tau estimate}
    |b_{\tau}(x)| 
    \lesssim 
    (1+|t|)^{-N}\,
    \varphi_{0}(x)
\end{align}
for every $x\in\mathbb{X}$ and for every $\tau=s-it$ with 
$s\in(0,1]$ and $t\in\mathbb{R}^{*}$.
\end{proposition}

\begin{proof}
Let us study
\begin{align*}
    B_{\tau}(\lambda)
    = \frac{2\tau}{\pi} \int_{0}^{+\infty}dv\,
        \lbrace1-\chi_{T}(v)\rbrace\,
        \frac{1}{\tau^2+v^2}\, 
        \cos(v\sqrt{|\lambda|^2+|\rho|^2}),
\end{align*}
which vanishes unless $v>2T$. 
By performing $N$ integrations by parts based on 
\begin{align*}
    \cos(v\sqrt{|\lambda|^2+|\rho|^2})
    = -\frac{1}{|\lambda|^2+|\rho|^2}\,   
        \frac{\partial^2}{\partial v^2}\,
        \cos(v\sqrt{|\lambda|^2+|\rho|^2}),
\end{align*}
we obtain
\begin{align*}
    B_{\tau}(\lambda)
    &= \frac{2\tau}{\pi}\,(-1)^{-\frac{N}{2}}
        (|\lambda|^2+|\rho|^2)^{-\frac{N}{2}}\\
     &\times\int_{0}^{+\infty}dv\, 
        \cos(v\sqrt{|\lambda|^2+|\rho|^2})\,
        \textstyle{ 
        \big(\frac{\partial}{\partial v})^{N}\,
        (\frac{1-\chi_{T}(v)}{\tau^2+v^2}\big)}.
\end{align*}
Since $v>2T$, we have $|\tau^2+v^2|\gtrsim v^2$ uniformly 
in $\tau=s-it$. Hence
\begin{align*}
    B_{\tau}(\lambda) 
    \lesssim |\tau|\,(|\lambda|^2+|\rho|^2)^{-\frac{N}{2}}\,
        \int_{2T}^{+\infty}dv\,v^{-2-N}\,
    \lesssim |T|^{-N}\,
        (|\lambda|^2+|\rho|^2)^{-\frac{N}{2}}.
\end{align*}
By the inverse formula of the spherical Fourier transform, 
we deduce
\begin{align*}
    |b_{\tau}(x)|
    =&\ \Big|\int_{\mathfrak{a}}d\lambda\,    
        |\mathbf{c}(\lambda)|^{-2}\,
        \varphi_{\lambda}(x)\, 
        B_{\tau}(\lambda)\Big|\\[5pt]
    \lesssim&\ |T|^{-N}\,\varphi_{0}(x)\, 
        \int_{\mathfrak{a}}d\lambda\,
        |\mathbf{c}(\lambda)|^{-2}\,
        (|\lambda|^2+|\rho|^2)^{-\frac{N}{2}}
\end{align*}
where the last integral converges provided that $N>d$.
\end{proof}

According to the asymptotic expansion \eqref{a tau expansion}
of $a_{\tau}$ and to the estimate \eqref{b tau estimate} of
$b_{\tau}$, we establish the pointwise estimates of
$\widetilde{\omega}_{t}^{\sigma,0}$ in the case where
$\frac{|x|}{|t|}$ is bounded from below.

\begin{theorem}\label{theorem subsection 2}
Let $\sigma\in\mathbb{C}$ with $\re\sigma=\frac{d+1}{2}$. 
The following estimates hold for all $t\in\mathbb{R}^{*}$ 
and $x\in\mathbb{X}$.
\begin{enumerate}[leftmargin=*, label=(\roman*)]
    \item 
    If $0<|t|<1$, then
    \begin{align*}
        |\widetilde{\omega}_{t}^{\sigma,0}(x)|
        \lesssim |t|^{-\frac{d-1}{2}}\,
        (1+|x^{+}|)^{\frac{\max\lbrace{d,D}\rbrace-\ell}{2}}\, 
        e^{-\langle\rho,x^{+}\rangle}.
    \end{align*}
    \item 
    If $|t|\ge1$ and $\frac{|x|}{|t|}>C_{\Sigma}$, then
    \begin{align*}
    |\widetilde{\omega}_{t}^{\sigma,0}(x)|
    \lesssim |t|^{-N_1}\,(1+|x^{+}|)^{N_2}\, 
    e^{-\langle\rho,x^{+}\rangle},
    \end{align*}
    for every $N_{1}\in\mathbb{N}$ and 
    $N_2\ge{N}_{1}+2(d+1)+ \frac{\max\lbrace{d,D}\rbrace-\ell}{2}$.
\end{enumerate}
\end{theorem}

\begin{proof}
Recall that we are looking for a pointwise estimate of
the kernel
\begin{align*}
    \widetilde{\omega}_{t}^{\sigma,0}(x)
    = C_{\sigma,d}\,\int_{0}^{1}ds\,
        s^{\sigma-1}\,p_{\tau}(x),
\end{align*}
where $\tau=s-it$ with $s\in(0,1]$ and $t\in\mathbb{R}^{*}$.
According to the Cartan decomposition, for every
$x\in\mathbb{X}$, there exist $k_{1},k_{2} \in K$ and
$x^{+}\in\overline{\mathfrak{a}^{+}}$ such that
$x=k_{1}(\exp{x^{+}})k_{2}$. Then
\begin{align*}
    p_{\tau}(x) 
    = a_{\tau}(\exp x^{+}) + b_{\tau}(\exp x^{+})
\end{align*}
by the $K$-bi-invariance. According to the expansion
\eqref{a tau expansion}, we split up
\begin{align*}
    \widetilde{\omega}_{t}^{\sigma,0}(x)
    &=I_{1}(t,x^{+})+I_{2}(t,x^{+})+I_{3}(t,x^{+})\\[5pt]
    &=\frac{1}{\pi}\,J(x^{+})^{-1/2}\,
        \sum_{k=0}^{[d/2]} 4^{-k}\,U_{k}(x^{+})\,
        {\textstyle\Gamma\big(\frac{d+1}{2}-k\big)}\, 
        I_{1,k}(t,x^{+})\\[5pt]
    &+ C_{\sigma,d}\,\int_{0}^{1}ds\,s^{\sigma-1}     
        E(\tau,x^{+})
        + C_{\sigma,d}\,\int_{0}^{1}ds\,s^{\sigma-1} 
        b_{\tau}(\exp x^{+})
\end{align*}
where
\begin{align*}
    I_{1,k}(t,x^{+})
    = C_{\sigma,d}\,\int_{0}^{1}ds\,s^{\sigma-1}\, 
        \tau(|x^{+}|^2+\tau^2)^{k-\frac{d+1}{2}}
\end{align*}
satisfies
\begin{align*}
    |I_{1,k}(t,x^{+})|
    \lesssim 1+|t|^{k-\frac{d-1}{2}}
\end{align*}
according to next lemma. Hence
\begin{align}\label{estimate I1}
    I_{1}(t,x^{+})\lesssim
    (\sqrt{|t|}+|t|^{-\frac{d-1}{2}})\,J(x^{+})^{-1/2}
    \quad\forall\,t\in\mathbb{R}^{*}.
\end{align}
The last two terms $I_{2}(t,x^{+})$ and $I_{3}(t,x^{+})$ are easily
handled: on the one hand, we have
\begin{equation}\label{estimate I2}\begin{aligned}
    |I_{2}(t,x^{+})|
    &\lesssim \int_{0}^{1}ds\,s^{\re\sigma-1}\,|E(\tau,x^{+})|\\
    &\lesssim (1+|t|)^{3(\frac{d}{2}+1)}\,\log{(2+|t|)}\,
    e^{-\langle\rho,x^{+}\rangle},
\end{aligned}\end{equation}
according to \eqref{error part}; 
on the other hand, \eqref{b tau estimate} yields
\begin{align}\label{estimate I3}
    |I_{3}(t,x^{+})|
    \lesssim \int_{0}^{1}ds\,s^{\re\sigma-1}\,
        |b_{\tau}(\exp x^{+})|
    \lesssim (1+|t|)^{-N}\,\varphi_{0}(\exp x^{+})
\end{align}
for all $\sigma\in\mathbb{C}$ with $\re\sigma=\frac{d+1}{2}$.
By summing up the estimates \eqref{estimate I1},
\eqref{estimate I2} and \eqref{estimate I3}, 
we deduce, on the one hand, 
\begin{align*}
    |\widetilde{\omega}_{t}^{\sigma,0}(x)|
    \lesssim
        |t|^{-\frac{d-1}{2}}\, 
        (1+|x^{+}|)^{\frac{\max\lbrace d,D\rbrace-\ell}{2}}\, e^{-\langle\rho,x^{+}\rangle}
\end{align*}
if $|t|<1$, and on the other hand,
\begin{align*}
    |\widetilde{\omega}_{t}^{\sigma,0}(x)| 
    \lesssim |t|^{3(\frac{d}{2}+1)}\,\log{(2+|t|)}\,
    (1+|x^{+}|)^{\frac{\max\lbrace d,D\rbrace-\ell}{2}}\,
    e^{-\langle\rho,x^{+}\rangle}
\end{align*}
if $|t|\ge1$. 
Since $\frac{|x|}{|t|}$ is bounded from below, 
we obtain finally
\begin{align*}
    |\widetilde{\omega}_{t}^{\sigma,0}(x)| 
    \lesssim |t|^{-N_1}\,(1+|x^{+}|)^{N_2}\,
        e^{-\langle\rho,x^{+}\rangle}
    \quad \forall\,|t|\ge1
\end{align*}
for every $N_{1}\in\mathbb{N}$ and  
$N_2\ge{N}_{1}+2(d+1)+ \frac{\max\lbrace d,D\rbrace-\ell}{2}$.
\end{proof}

\begin{remark}
Notice that the above method works only in small time, 
or in large time under the assumption that
$\frac{|x|}{|t|}$ is bounded from below. 
The large polynomial growth in $|x^{+}|$ appearing in the estimate 
is not crucial for further computations because of the exponential
decay $e^{-\langle\rho,x^{+}\rangle}$.
\end{remark}

\begin{lemma}
For every integer $0\le{k}<\frac{d+1}{2}$, the integral
\begin{align*}
    I_{1,k}(t,x^{+})
    = C_{\sigma,d}\,\int_{0}^{1}ds\,s^{\sigma-1}\, 
        \tau(|x^{+}|^2+\tau^2)^{k-\frac{d+1}{2}}
\end{align*}
satisfies
\begin{align*}
    |I_{1,k}(t,x^{+})|
    \lesssim 1+|t|^{k-\frac{d-1}{2}}
    \quad\forall\,t\in\mathbb{R}^{*},\
    \forall\,x\in\frakACL
\end{align*}
uniformly in $\sigma\in\mathbb{C}$ with $\re\sigma=\frac{d+1}{2}$.
\end{lemma}

\begin{proof}
Since $\tau=s-it$, we write $I_{1,k}(t,x^{+})=P_{1}+P_{2}$ with
\begin{align*}
    P_{1}=
    C_{\sigma,d}\,\int_{0}^{1}ds\,s^{\sigma}\, 
        (s^2+|x^{+}|^2-t^2-2sti)^{k-\frac{d+1}{2}}
\end{align*}
and
\begin{align*}
    P_{2}=
    C_{\sigma,d}\,(-it)\,\int_{0}^{1}ds\,s^{\sigma-1}\, 
        (s^2+|x^{+}|^2-t^2-2sti)^{k-\frac{d+1}{2}}.
\end{align*}
As
\begin{align*}
\big|s^2+|x^{+}|^2-t^2-2sti\big|
=\sqrt{s^4+2s^2(|x^{+}|^2+t^2)+(|x^{+}|^2-t^2)^2},
\end{align*}
notice that
\begin{numcases}
{\big|s^2+|x^{+}|^2-t^2-2sti\big|\ge}
    s^2 \label{lemma inequality s^2}
    \\[5pt]
    s|t| \label{lemma inequality st}
    \\[5pt]
    \big||x^{+}|^2-t^2\big| \label{lemma inequality x-t}
\end{numcases}

$P_{1}$ is easily estimated. 
By using \eqref{lemma inequality st}, we obtain first
\begin{align*}
    |P_{1}|
    \lesssim
    |t|^{k-\frac{d+1}{2}}\,\int_{0}^{1}ds\,s^{k}
    \le
    |t|^{k-\frac{d+1}{2}}
    \quad\forall\,t\in\mathbb{R}^{*}.
\end{align*}
By using in addition \eqref{lemma inequality s^2},
we obtain next, for $|t|<1$,
\begin{align*}
    |P_{1}|
    \lesssim
    |t|^{k-\frac{d+1}{2}}\,\int_{0}^{|t|}ds\,s^{k}
    +\int_{|t|}^{1}\,ds\,s^{2k-\frac{d+1}{2}}
    \lesssim 1+|t|^{2k-\frac{d-1}{2}}
\end{align*}
We deduce that 
\begin{align}\label{estimate of P1}
    |P_{1}|\lesssim
    \begin{cases}
        1+|t|^{2k-\frac{d-1}{2}}
        &\quad\textnormal{if \,}|t|<1,\\[5pt]
        |t|^{k-\frac{d+1}{2}}
        &\quad\textnormal{if \,}|t|\ge1.
    \end{cases}
\end{align}
\vspace{10pt}

Let us turn to $P_2$. Consider first the easy case where
$1\le{k}<\frac{d+1}{2}$. By using \eqref{lemma inequality st}
again, we get
\begin{align}\label{estimate of P2 k>0}
    |P_{2}|\lesssim
    |t|\cdot|t|^{k-\frac{d+1}{2}}\,
    \int_{0}^{1}ds\,s^{\re\sigma-1+k-\frac{d+1}{2}}
    \lesssim |t|^{k-\frac{d-1}{2}}
\end{align}
for all $\sigma\in\mathbb{C}$ with $\re\sigma=\frac{d+1}{2}$.
In order to estimate $P_2$ in the remaining case where $k=0$, 
we write
\begin{align*}
    P_{2}=
    C_{\sigma,d}\,(-it)\,\int_{0}^{1}ds\,s^{i\im\sigma-1}\,
    {\textstyle
    \big(\frac{s}{s^2+|x^{+}|^2-t^2-2sti}\big)^{\frac{d+1}{2}}}.
\end{align*}
By performing an integration by parts, $P_{2}$ becomes the sum of 
$P_{2}^{-}$ and $P_{2}^{+}$ where
\begin{align*}
\textstyle
    P_{2}^{-}=
    \big[
        \frac{C_{\sigma,d}}{\im\sigma}\,(-t)\,s^{i\im\sigma}
        \big(\frac{s}{s^2+|x^{+}|^2-t^2-2sti}\big)^{\frac{d+1}{2}}
    \big]_{0}^{1}
\end{align*}
and
\begin{align*}
    P_{2}^{+}=
    {\textstyle\frac{C_{\sigma,d}}{\im\sigma}}\,t\,
    \int_{0}^{1}ds\,s^{i\im\sigma}\,
    {\textstyle
    \frac{\partial}{\partial{s}}
    \big\lbrace
    \big(\frac{s}{s^2+|x^{+}|^2-t^2-2sti}\big)^{\frac{d+1}{2}}
    \big\rbrace}.
\end{align*}
By using \eqref{constant C sigma d} together with 
\eqref{lemma inequality s^2} in small time and 
\eqref{lemma inequality st} in large time, we obtain
\begin{align}\label{estimate of P2-}
    |P_{2}^{-}|\lesssim    
    \begin{cases}
        1
        &\quad\textnormal{if \,}|t|<1,\\[5pt]
        |t|^{-\frac{d-1}{2}}
        &\quad\textnormal{if \,}|t|\ge1.
    \end{cases}
\end{align}
Since
\begin{align*}
\textstyle
    \frac{\partial}{\partial{s}}
    \big\lbrace
    \big(\frac{s}{s^2+|x^{+}|^2-t^2-2sti}\big)^{\frac{d+1}{2}}
    \big\rbrace
    =
    \underbrace{\textstyle
    \frac{d+1}{2}
    \Big(\frac{s}{s^2+|x^{+}|^2-t^2-2sti}\Big)^{\frac{d-1}{2}}
    }_{O(|t|^{-\frac{d-1}{2}})}\,
    \frac{|x^{+}|^2-t^2-s^2}{(s^2+|x^{+}|^2-t^2-2sti)^2},
\end{align*}
we have
\begin{align}\label{estimate of P2+}
    |P_{2}^{+}|\lesssim
    |t|^{-\frac{d-1}{2}}
    \underbrace{
    |t|\int_{0}^{1}ds\,
    {\textstyle
    \frac{\big||x^{+}|^2-t^2-s^2\big|}
    {\big|s^2+|x^{+}|^2-t^2-2sti\big|^2}
    }}_{Q}.
\end{align}
It remains for us to estimate $Q$, which is bounded by 
the sum of
\begin{align*}
    Q_{1}
    = |t|\int_{0}^{1}ds\,
    {\textstyle \frac{s^2}
    {\big|s^2+|x^{+}|^2-t^2-2sti\big|^2}}
    \quad\textnormal{and}\quad
    Q_{2}
    = |t|\int_{0}^{1}ds\,
    {\textstyle \frac{\big||x^{+}|^2-t^2\big|}
    {\big|s^2+|x^{+}|^2-t^2-2sti\big|^2}}.
\end{align*}
$Q_{1}$ is estimated as $P_{1}$. 
According to \eqref{lemma inequality st} and 
\eqref{lemma inequality s^2}, we have
\begin{align*}
    Q_{1}\lesssim
    \begin{cases}
        |t|\,\int_{0}^{1}ds\,|t|^{-2}=|t|^{-1}\le1
        &\quad\textnormal{if \,}\,|t|\ge1,\\[5pt]
        |t|\,\int_{0}^{|t|}ds\,|t|^{-2}
        +|t|\,\int_{|t|}^{1}ds\,s^{-2}\le2
        &\quad\textnormal{if \,}\,|t|<1.
    \end{cases}
\end{align*}
Let us finally estimate $Q_2$.
On the one hand, if $\big||x^{+}|^2-t^2\big|\ge|t|$, 
by using \eqref{lemma inequality x-t}, we get
\begin{align*}
    |Q_{2}|\lesssim
    |t|\int_{0}^{1}ds\,
    \big||x^{+}|^2-t^2\big|^{-1}
    \lesssim1.
\end{align*}
On the other hand, if $\big||x^{+}|^2-t^2\big|\le|t|$,
we have $Q_{2}=\mathrm{O}(1)$ since
\begin{align*}
    \Big|
    |t|\int_{0\le{s}\le\frac{||x^{+}|^2-t^2|}{|t|}}ds\,
    {\textstyle \frac{\big||x^{+}|^2-t^2\big|}
    {\big|s^2+|x^{+}|^2-t^2-2sti\big|^2}}
    \Big|
    &\lesssim
    |t|\int_{0\le{s}\le\frac{||x^{+}|^2-t^2|}{|t|}}ds\,
    \big||x^{+}|^2-t^2\big|^{-1}\\
    &\le1
\end{align*}
according to \eqref{lemma inequality x-t}, and
\begin{align*}
    \Big|
    |t|\int_{\frac{||x^{+}|^2-t^2|}{|t|}\le{s}\le1}ds\,
    {\textstyle \frac{\big||x^{+}|^2-t^2\big|}
    {\big|s^2+|x^{+}|^2-t^2-2sti\big|^2}}
    \Big|
    &\lesssim
    {\textstyle
    \frac{||x^{+}|^2-t^2|}{|t|}}
    \int_{\frac{||x^{+}|^2-t^2|}{|t|}\le{s}\le1}ds\,s^{-2}\\
    &\le2
\end{align*}
according to \eqref{lemma inequality st}. 
Hence $Q=\mathrm{O}(1)$ and we deduce from \eqref{estimate of P2+} that
$|P_{2}^{+}|\lesssim|t|^{-\frac{d-1}{2}}$ for all $t\in\mathbb{R}^{*}$.
By combining with \eqref{estimate of P2-} and \eqref{estimate of P2 k>0},
we obtain
\begin{align*}
    |P_{2}|\lesssim |t|^{k-\frac{d-1}{2}}
    \quad\forall\,t\in\mathbb{R}^{*}.
\end{align*}
Together with \eqref{estimate of P1},
this concludes the proof.
\end{proof}

\subsection{Estimates of
$\omega_{t}^{\sigma,\infty}$}\label{subsection 3}

We establish in this last subsection the pointwise 
estimates of $\omega_{t}^{\sigma,\infty}$. Recall that
\begin{align*}
    \omega_{t}^{\sigma,\infty}(x)
    = \frac{1}{\Gamma(\sigma)} \int_{1}^{+\infty}ds\,
        s^{\sigma-1}\,p_{s-it}(x)\,
    \quad\forall\,x\in\mathbb{X},\ 
    \forall\,t\in\mathbb{R}^{*}.
\end{align*}
According to the integral expression 
\eqref{integral expression of phi lambda} of 
the spherical function, we may write
\begin{align*}
    \omega_{t}^{\sigma,\infty}(x) 
    = \frac{C_{0}}{\Gamma(\sigma)}
        \int_{K}dk\,e^{\langle\rho,A(kx)\rangle}\,
        \int_{1}^{+\infty}ds\, 
        s^{\sigma-1}\,I(s,t,x),
\end{align*}
where, let us recall,
\begin{align*}
    I(s,t,x)
    = \int_{\mathfrak{a}}d\lambda\, 
        |\mathbf{c}(\lambda)|^{-2}\,
        e^{-s\sqrt{|\lambda|^2+|\rho|^2}}\,
        e^{it \psi_{t}(\lambda)}.
\end{align*}
We have considered this oscillatory integral in the case where $s\in(0,1)$.For $s\ge1$, the factor 
$e^{-s\sqrt{|\lambda|^2+|\rho|^2}}$ plays an important role.
On the one hand, for $\lambda$ close to the critical point 
of $\psi_{t}(\lambda)$, this factor produces an exponential
decay in $s$, which ensures the convergence of the integral
over $s\in(1,+\infty)$. For $\lambda$ away from the critical
point, it produces an exponential decay in $|\lambda|$, which
ensures the convergence of the integral over
$\lambda\in\mathfrak{a}$. Let us elaborate.

\begin{theorem}\label{theorem subsection 3}
The following estimate holds,
uniformly in the strip $0\le\re\sigma\le\frac{d+1}{2}$,
for all $t\in\mathbb{R}^{*}$ and $x\in\mathbb{X}$:
\begin{align}\label{theorem subsection 3-1}
    |\omega_{t}^{\sigma,\infty}(x)| 
    \lesssim \varphi_{0}(x).
\end{align}
Moreover, if $|t|\ge1$,
\begin{align}\label{theorem subsection 3-2}
     |\omega_{t}^{\sigma,\infty}(x)| 
     \lesssim |t|^{-\frac{D}{2}}\, 
     (1+|x|)^{\frac{D}{2}}\,\varphi_{0}(x).
\end{align}
\end{theorem}

\begin{proof}
The global estimate \eqref{theorem subsection 3-1} is 
obtained by a straightforward computation. On the one hand,
\begin{align*}
    \int_{|\lambda|\le1}d\lambda\, 
    |\mathbf{c}(\lambda)|^{-2}\, 
    e^{-s\sqrt{|\lambda|^2+|\rho|^2}}\,
    \le e^{-s|\rho|}
        \underbrace{
        \int_{|\lambda|\le1} d\lambda\, 
        |\lambda|^{D-\ell}}_{<+\infty}.
\end{align*}
On the other hand,
\begin{align*}
    \int_{|\lambda|\ge1} d\lambda\, 
    |\mathbf{c}(\lambda)|^{-2}\, 
    e^{-s\sqrt{|\lambda|^2+|\rho|^2}} 
    \le e^{-\frac{s}{2}|\rho|}
        \underbrace{
        \int_{|\lambda|\ge1} d\lambda\, 
        |\lambda|^{d-\ell}\,
        e^{-\frac{s}{2}|\lambda|}}_{<+\infty}.
\end{align*}
Hence
\begin{align}
    |\omega_{t}^{\sigma,\infty}(x)|
    \lesssim \varphi_{0}(x) 
        \underbrace{
        \int_{1}^{+\infty} ds\, s^{\re\sigma-1}\, 
        e^{-\frac{s}{2}|\rho|}}_{< +\infty}.
\end{align}

The estimate \eqref{theorem subsection 3-2} follows from
\eqref{theorem subsection 3-1} if $\frac{|x|}{|t|}$ is 
bounded from below. Let us prove it if $\frac{|x|}{|t|}$ is
bounded from above, let say by $\frac{1}{2}$. 
We study the oscillatory integral $I$ along the 
lines of \cref{subsection 1}. Let split up again
\begin{align*}
    I (s,t,x)
    = I^{-} (s,t,x) + I^{+} (s,t,x)
    = \int_{\mathfrak{a}} d\lambda\, 
        \chi_{0}^{\rho}(\lambda) \cdots
    \,+\, 
        \int_{\mathfrak{a}} d\lambda\,
        \chi_{\infty}^{\rho}(\lambda) \cdots
\end{align*}
according to cut-off functions $\chi_{0}^{\rho}$ and
$\chi_{\infty}^{\rho}=1-\chi_{0}^{\rho}$, which have been
defined in \cref{subsection 1}.
Recall that $\chi_{0}^{\rho}(\lambda)=1$ when
$|\lambda|\le|\rho|$ and vanishes if $|\lambda|\ge2|\rho|$.
\vspace{10pt}

On the one hand, $I^{-}$ is estimated by studying 
the oscillatory integral
\begin{align*}
    I^{-} (s,t,x)
    = \int_{\mathfrak{a}} d\lambda\, 
        \underbrace{
        \chi_{0}^{\rho}(\lambda)\, 
        |\mathbf{c}(\lambda)|^{-2}\, 
        e^{-s\sqrt{|\lambda|^2+|\rho|^2}}
        }_{a_{0}(s,\lambda)}
        e^{it \psi_{t}(\lambda)}
\end{align*}
where the amplitude $a_{0}$ is compactly supported 
for $|\lambda|\le2|\rho|$, and in this range, the phase $\psi_{t}$, 
defined by \eqref{phase}, has a single critical
point, which is nondegenerate and small if 
$\frac{|x|}{|t|}\le\frac{1}{2}$. 
According to \cref{Appendix A-lemma}, we obtain
\begin{align}\label{estimate I infinity -}
    |I^{-}(s,t,x)|
    \lesssim
    |t|^{-\frac{D}{2}}\,(1+|x|)^{\frac{D-\ell}{2}}\,
    e^{-\frac{|\rho|}{2}s}.
\end{align}
On the other hand, 
\begin{align*}
    I^{+} (s,t,x)
    = \int_{\mathfrak{a}} d\lambda\,
        \chi_{\infty}^{\rho}(\lambda)\,
        |\mathbf{c}(\lambda)|^{-2}\, 
        e^{-s\sqrt{|\lambda|^2+|\rho|^2}}\,
        e^{it \psi_{t}(\lambda)}
\end{align*}
is easily estimated with no barycentric decomposition.
Let
\begin{align*}
\textstyle
    \widetilde{\psi}_{\infty}(\lambda)
    = \frac{|\lambda|^2}{\sqrt{|\lambda|^2+|\rho|^2}}
        + \langle\frac{A(kx)}{t},\lambda\rangle 
    \quad\forall\,\lambda\in\supp{\chi}_{\infty}^{\rho}.
\end{align*}
Then $\widetilde{\psi}_{\infty}$ is a symbol of order $1$, 
and satisfies 
\begin{align*}
\textstyle
    |\widetilde{\psi}_{\infty}(\lambda)|
    =
    \underbrace{
        \frac{|\lambda|^2}{\sqrt{|\lambda|^2+|\rho|^2}}
    }_{\ge\frac{|\lambda|}{\sqrt{2}}}
    -
    \underbrace{
        \frac{\langle\lambda_{0},\lambda\rangle}
        {\sqrt{|\lambda_{0}|^2+|\rho|^2}}
    }_{\le\frac{|\lambda|}{\sqrt{3}}}
    \ge (\frac{1}{\sqrt{2}}-\frac{1}{\sqrt{3}})|\rho|>0
\end{align*}
on $\supp{\chi}_{\infty}^{\rho}$ according to 
\eqref{critical point A/t} and \eqref{critical point lambda 0}.
By performing $N$ integrations by parts based on
\begin{align*}
\textstyle 
    e^{it \psi_{t}(\lambda)}
    = \frac{1}{it}\, 
        \widetilde{\psi}_{\infty}(\lambda)^{-1}\,
        \sum_{j=1}^{\ell} \lambda_{j}\,
        \frac{\partial}{\partial\lambda_{j}}\,
        e^{it \psi_{t}(\lambda)},
\end{align*}
we write 
\begin{align*}
    I_{\infty}^{+} (s,t,x) 
    &=(it)^{-N}
        \int_{\mathfrak{a}} d\lambda\, 
        e^{it \psi_{t}(\lambda)}\\
    &\times
        {\textstyle\big\lbrace
        -\sum_{j=1}^{\ell}
        \frac{\partial}{\partial\lambda_{j}}\circ
        \frac{\lambda_{j}}{\widetilde{\psi}_{\infty}(\lambda)}
        \big\rbrace^{N}
        \big\lbrace 
        \chi_{\infty}^{\rho}(\lambda)\,
        |\mathbf{c}(\lambda)|^{-2}\, 
        e^{-s\sqrt{|\lambda|^2+|\rho|^2}}
        \big\rbrace}.
\end{align*}
If some derivatives hit $\chi_{\infty}^{\rho}(\lambda)$, 
the range of the above integral is reduced to a spherical 
shell where $|\lambda|\asymp|\rho|$, and
\begin{align*}
    I^{+} (s,t,x) \asymp |t|^{-N} e^{-s|\rho|}.
\end{align*}
Assume next that no derivative is applied to
$\chi_{\infty}^{\rho}$ and
\begin{itemize}[leftmargin=*]
    \item 
    $N_1$ derivatives are applied to the factors
    $\lambda_{j}/\widetilde{\psi}_{\infty}(\lambda)$, which are
    inhomogeneous symbols of order $0$, producing contributions
    which are $O(|\lambda|^{-N_1})$, 
    \vspace{5pt}
    \item 
    $N_2$ derivatives are applied to the factor
    $|\mathbf{c}(\lambda)|^{-2}$ which is not a symbol in
    general, producing a contribution which is
    $O(|\lambda|^{d-\ell})$, 
    \vspace{5pt}
    \item 
    $N_3$ derivatives are applied to the factor
    $e^{-s\sqrt{|\lambda|^2+|\rho|^2}}$, 
    producing a contribution which is 
    $O(s^{N_3} e^{-s\sqrt{|\lambda|^2+|\rho|^2}})$,
    \vspace{5pt}
\end{itemize}
with $N_1+N_2+N_3=N$. 
Then we get the upper bound
\begin{align*}
    |t|^{-N}\,s^{N_3} 
    \int_{|\lambda|>|\rho|} d\lambda\,
    |\lambda|^{d-\ell-N_1}\,
    e^{-s\sqrt{|\lambda|^2+|\rho|^2}},
\end{align*}
which yields
\begin{align*}
    |I^{+} (s,t,x)| 
    \lesssim |t|^{-N}\,s^{N}\,e^{-\frac{s}{2}|\rho|}
        \underbrace{
        \int_{|\lambda|>|\rho|}d\lambda\,
        |\lambda|^{d-\ell}\, 
        e^{-\frac{s}{2}|\lambda|}}_{< +\infty}.
\end{align*}
Together with \eqref{estimate I infinity -}, we obtain
\begin{align*}
    |I(s,t,x)| 
    \lesssim |t|^{-\frac{D}{2}}\, 
    (1+|x|)^{\frac{D-\ell}{2}}\, 
    s^{N}e^{-\frac{s}{2}|\rho|}.
\end{align*}
for all $s\ge1$ and for $N\ge\frac{D}{2}$.
We deduce
\begin{align*}
    |\omega_{t}^{\sigma,\infty}(x)| 
    \lesssim |t|^{-\frac{D}{2}}\, 
        (1+|x|)^{\frac{D-\ell}{2}}\,
        \varphi_{0}(x)\,
        \underbrace{
        \int_{1}^{+\infty}ds\, 
        s^{\re\sigma-1+N}\, 
        e^{-\frac{|\rho|}{2}s}}_{<+\infty}
\end{align*}
for all $x\in\mathbb{X}$ and $|t|\ge1$.
\end{proof}

\section{Dispersive estimates}
\label{section dispersive estimates}

In this section, we prove our second main result about the
$L^{q'}\rightarrow{L^{q}}$ estimates for the operator
$W_{t}^{\sigma}=(-\Delta)^{-\frac{\sigma}{2}}
e^{it\sqrt{-\Delta}}$. We introduce the following criterion
based on the Kunze-Stein phenomenon, which is crucial for 
the proof of dispersive estimates.
\begin{lemma}\label{Kunze-Stein}
Let $\kappa$ be a reasonable $K$-bi-invariant function on $G$.
Then
\begin{align*}
    \OpNorme{\,\cdot*\kappa}{q'}{q}
    \le \Big\lbrace \int_{G} dx\, 
        \varphi_{0}(x)\,
        |\kappa(x)|^{\frac{q}{2}}\, 
        \Big\rbrace^{\frac{2}{q}}
\end{align*}
for every $q\in[2,+\infty)$.
In the limit case $q=\infty$,
\begin{align*}
\textstyle
    \OpNorme{\,\cdot*\kappa}{1}{\infty}
    = \sup_{x\in{G}} |\kappa(x)|.
\end{align*}
\end{lemma}

\begin{remark}
This lemma has been proved in several contexts. 
For $q=2$, it is the so-called Herz's criterion, 
see for instance \cite{Cow1997}. 
For $q>2$, the proof carried out on Damek-Ricci spaces
\cite[Theorem 4.2]{APV2011} is adapted straightforwardly 
in the higher rank case. 
\end{remark}

\begin{theorem}[Small time dispersive estimate]
\label{small time dispersive estimate}
Let $d\ge3$ and $0<|t|<1$. Then
\begin{align*}
    \OpNorme{(-\Delta)^{-\frac{\sigma}{2}}
        e^{it\sqrt{-\Delta}}}{q'}{q}
    \lesssim |t|^{-(d-1)(\frac{1}{2}-\frac{1}{q})}
\end{align*}
for all $2<q<+\infty$ and
$\sigma\ge(d+1)(\frac{1}{2}-\frac{1}{q})$.
\end{theorem}

\begin{proof}
We divide the proof into two parts, corresponding to the 
kernel decomposition $\omega_{t}^{\sigma} =
\omega_{t}^{\sigma,0}+\omega_{t}^{\sigma, \infty}$. 
According to \cref{Kunze-Stein} and to the pointwise estimate
in \cref{theorem subsection 3}, we obtain on the one hand
\begin{align*}
    \OpNorme{\,\cdot * \omega_{t}^{\sigma,\infty}}{q'}{q}
        &\le\Big\lbrace 
        \int_{G} dx\, \varphi_{0}(x)\,
        |\omega_{t}^{\sigma,\infty}(x)|^{\frac{q}{2}} 
        \Big\rbrace^{\frac{2}{q}} \\[5pt]
        &\lesssim\Big\lbrace 
        \int_{\mathfrak{a}^{+}} dx^{+}\, 
        |\varphi_{0}(x^{+})|^{\frac{q}{2}+1}\,\delta(x^{+})
        \Big\rbrace^{\frac{2}{q}} \\[5pt]
        &\lesssim\Big\lbrace 
        \int_{\mathfrak{a}^{+}} dx^{+}\,
        (1+|x^{+}|)^{\frac{D-\ell}{2}(\frac{q}{2}+1)}\,
        e^{-(\frac{q}{2}-1)\langle\rho,x^{+}\rangle}
        \Big\rbrace^{\frac{2}{q}} \\
        &< +\infty\vphantom{\Big|}
\end{align*}
for all $q>2$. 
On the other hand, we use an analytic interpolation between the $L^{2}\rightarrow L^{2}$ and $L^{1}\rightarrow L^{\infty}$
estimates for the family of operators
$\widetilde{W}_{t}^{\sigma,0}$ defined by 
\eqref{analytic family tilde W} in the vertical strip
$0\le\re\sigma\le\frac{d+1}{2}$. 
When $\re\sigma=0$, the spectral theorem yields
\begin{align*}
    \OpNorme{\widetilde{W}_{t}^{\sigma,0}}{2}{2}
    = \OpNorme{e^{it \sqrt{- \Delta}}}{2}{2}    
    = 1
\end{align*}
for all $t \in \mathbb{R}^{*}$. 
According to \cref{theorem subsection 2}, 
when $\re\sigma=\frac{d+1}{2}$,
\begin{align*}
    \OpNorme{\widetilde{W}_{t}^{\sigma,0}}{1}{\infty}
    \lesssim \|\widetilde{\omega}_{t}^{\sigma,0}
                \|_{L^{\infty}(\mathbb{X})}
    \lesssim |t|^{-\frac{d-1}{2}}.
\end{align*}
By Stein's interpolation theorem applied to the 
analytic family of operators $\widetilde{W}_{t}^{\sigma,0}$, 
we conclude for $\sigma=(d+1)(\frac{1}{2}-\frac{1}{q})$ that
\begin{align*}
    \OpNorme{W_{t}^{\sigma}}{q'}{q}
    \lesssim |t|^{-(d-1)(\frac{1}{2}-\frac{1}{q})},
\end{align*}
for all $0<|t|<1$ and $2<q<+\infty$.
\end{proof}

\begin{theorem}[Large time dispersive estimate]
\label{large time dispersive estimate}
Assume that $|t|\ge1$, $2<q<+\infty$ and
$\sigma\ge(d+1)(\frac{1}{2}-\frac{1}{q})$. Then
\begin{align*}
    \OpNorme{(-\Delta)^{-\frac{\sigma}{2}}
        e^{it\sqrt{-\Delta}}}{q'}{q}
    \lesssim |t|^{-\frac{D}{2}}.
\end{align*}
\end{theorem}

\begin{proof}
We divide the proof into three parts, corresponding to the
kernel decomposition
\begin{align*}
    \omega_{t}^{\sigma}
    = \mathds{1}_{B(0,C_{\Sigma}|t|)}\,\omega_{t}^{\sigma,0}
    + \mathds{1}_{\mathbb{X}\backslash{B(0,C_{\Sigma}|t|)}\,}
        \omega_{t}^{\sigma,0}
    + \omega_{t}^{\sigma,\infty}
\end{align*}
where the constant $C_{\Sigma}$ has been specified 
in the proof of \cref{theorem subsection 1}.
The first and the last terms are estimated by straightforward
computations. By combining \cref{Kunze-Stein} with the 
pointwise kernel estimates in \cref{theorem subsection 1} 
and \cref{theorem subsection 3}, we obtain indeed 
\begin{align*}
    &\OpNorme{\,\cdot*\lbrace\mathds{1}_{B(0,C_{\Sigma}|t|)}\,
        \omega_{t}^{\sigma,0}\rbrace}{q'}{q} 
        \hspace{-2cm}\\[5pt]
    &\lesssim\Big\lbrace 
        \int_{G}dx\,\varphi_{0}(x)\,
        |\mathds{1}_{B(0,C_{\Sigma}|t|)}(x)\,
        \omega_{t}^{\sigma,0}(x)|^{\frac{q}{2}} 
        \Big\rbrace^{\frac{2}{q}} \\[5pt]
    &\lesssim|t|^{-\frac{D}{2}}
        \underbrace{\Big\lbrace 
        \int_{|x^{+}|<C_{\Sigma}|t|}dx^{+}\,
        (1+|x^{+}|)^{\frac{D-\ell}{2}(q+1)}\,
        e^{-(\frac{q}{2}-1)\langle\rho,x^{+}\rangle}
        \Big\rbrace^{\frac{2}{q}}}_{<+\infty}
\end{align*}
and
\begin{align*}
    &\OpNorme{\,\cdot*\omega_{t}^{\sigma,\infty}}{q'}{q}
        \hspace{-2cm}\\[5pt]
    &\lesssim\Big\lbrace 
        \int_{G}dx\,\varphi_{0}(x)\,
        |\omega_{t}^{\sigma,\infty}(x)|^{\frac{q}{2}} 
        \Big\rbrace^{\frac{2}{q}} \\[5pt]
    &\lesssim|t|^{-\frac{D}{2}} 
        \underbrace{\Big\lbrace 
        \int_{\mathfrak{a}^{+}}dx^{+}\,
        (1+|x^{+}|)^{\frac{D-\ell}{2}
            + (D-\frac{\ell}{2})\frac{q}{2}}\,
        e^{-(\frac{q}{2}-1)\langle\rho,x^{+}\rangle}
        \Big\rbrace^{\frac{2}{q}}}_{<+\infty}.
\end{align*}
Here $2\le q<+\infty$ and the above estimates are uniform
in the strip $0\le\re\sigma\le\frac{d+1}{2}$.
As far as the middle term is concerned, we use again the analytic
interpolation for the family of operators associated with the
convolution kernel
$\mathds{1}_{\mathbb{X}\backslash{B(0,C_{\Sigma}|t|)}\,}
\widetilde{\omega}_{t}^{\sigma,0}$.
On the one hand, if $\re\sigma=0$, then
\begin{align*}
    &\OpNorme{\,\cdot*
    \mathds{1}_{\mathbb{X}\backslash{B(0,C_{\Sigma}|t|)}\,}
    \widetilde{\omega}_{t}^{\sigma,0}}{2}{2}\\
    &\le\,
    \OpNorme{\,\cdot*\widetilde{\omega}_{t}^{\sigma,0}}{2}{2}
    +
    \OpNorme{\,\cdot*
    \mathds{1}_{B(0,C_{\Sigma}|t|)\,}
    \widetilde{\omega}_{t}^{\sigma,0}}{2}{2}
    \lesssim\,1.
\end{align*}
On the other hand, if $\re\sigma=\frac{d+1}{2}$, we deduce 
from \cref{theorem subsection 2} that
\begin{align*}
    \OpNorme{\,\cdot*
    \mathds{1}_{\mathbb{X}\backslash{B(0,C_{\Sigma}|t|)}\,}
    \widetilde{\omega}_{t}^{\sigma,0}}{1}{\infty}
    &=\sup_{x\in\mathbb{X}}
    |\mathds{1}_{\mathbb{X}\backslash{B(0,C_{\Sigma}|t|)}}(x)\,
    \widetilde{\omega}_{t}^{\sigma,0}(x)|
    \lesssim |t|^{-N}
\end{align*}
for any $N\in\mathbb{N}$.
By using Stein's interpolation theorem between the above
$L^{2}\rightarrow L^{2}$ and $L^{1}\rightarrow L^{\infty}$
estimates, we obtain
\begin{align*}
    \OpNorme{\,\cdot*
    \mathds{1}_{\mathbb{X}\backslash{B(0,C_{\Sigma}|t|)}\,}
    \omega_{t}^{\sigma,0}}{q'}{q} 
    \lesssim |t|^{-N},
\end{align*}
for all $|t|\ge1$, $2<q<+\infty$ and for any $N\in\mathbb{N}$.
This concludes the proof.
\end{proof}

\begin{remark}
The standard $TT^{*}$ method used to prove the
Strichartz inequality breaks down in the critical case.
In order to take care of these endpoints, we need the 
dyadic decomposition method carried out in \cite{KeTa1998}
and the following stronger dispersive property, which is 
obtained by interpolation arguments.
\end{remark}

\begin{corollary}
Let $d\ge3$, $2 < q, \widetilde{q} < + \infty$ and 
$\sigma\ge(d+1)\,\max(\frac{1}{2}-\frac{1}{q},
\frac{1}{2}-\frac{1}{\widetilde{q}})$. Then there exists a
constant $C>0$ such that the following dispersive estimates
hold:
    \begin{align*}
    \OpNorme{(-\Delta)^{-\frac{\sigma}{2}}
        e^{it\sqrt{-\Delta}}}{\widetilde{q}'}{q}
    \le C
        \begin{cases}
            |t|^{-(d-1)\max(\frac{1}{2}-\frac{1}{q},     
            \frac{1}{2}-\frac{1}{\widetilde{q}})} 
            &\textnormal{if \,}0<|t|<1,\\[5pt]
            |t|^{-\frac{D}{2}} 
            &\textnormal{if \,}|t|\ge1.
        \end{cases}
    \end{align*}
\end{corollary}

\section{Strichartz inequality and applications}
\label{section applications}

In this section, we use the dispersive properties proved in
the previous section to establish the Strichartz inequality.
This inequality serves as a tool for finding minimal
regularity conditions on the initial data ensuring
well-posedness of related semi-linear wave equations.
Such results were previously known to hold for real 
hyperbolic spaces \cite{AnPi2014} (actually for all noncompact
symmetric spaces of rank one) and for noncompact symmetric 
spaces $G/K$ with $G$ complex \cite{Zha2021}.
For simplicity, we may assume that $\ell\ge2$ throughout this
section, thus $d\ge4$.

Let $\sigma \in \mathbb{R}$ and $1<q<+\infty$. 
Recall that the Sobolev space $H^{\sigma,q}(\mathbb{X})$ is 
the image of $L^q(\mathbb{X})$ under the operator 
$(- \Delta)^{- \frac{\sigma}{2}}$, equipped with the norm
\begin{align*}
    \|f\|_{H^{\sigma,q}(\mathbb{X})} 
    = \|(-\Delta)^{\frac{\sigma}{2}}\,f\|_{L^q(\mathbb{X})}.
\end{align*}
If $\sigma=N$ is a nonnegative integer, then
$H^{\sigma,q}(\mathbb{X})$ coincides with the classical 
Sobolev space
\begin{align*}
    W^{N,q} (\mathbb{X}) 
    = \lbrace{f\in L^q(\mathbb{X})\,|\,
        \nabla^{j} f\in L^q(\mathbb{X})\,
        \forall\, 1 \le j \le N}\rbrace,
\end{align*}
defined by means of covariant derivatives. 
We refer to \cite{Tri1992} for more details about function
spaces on Riemannian manifolds. 
Let us state the Strichartz inequality and some applications.
The proofs are adapted straightforwardly from 
\cite{APV2012,AnPi2014} and are therefore omitted.

\subsection{Strichartz inequality}
\label{subsection Strichartz}

We study the linear inhomogeneous wave equation on $\mathbb{X}$
\begin{align}\label{Cauchy problem}
    \begin{cases}
        \partial_{t}^2\,u(t,x)-\Delta\,u(t,x)=F(t,x), \\[5pt]
    u(0,x) =f(x),\,\partial_{t}|_{t=0} u(t,x)=g(x)
    \end{cases}
\end{align}
whose solution is given by Duhamel's formula:
\begin{align*}
    u(t,x) 
    = (\cos{t\sqrt{-\Delta}}) f(x)
    + {\textstyle 
        \frac{\sin t\sqrt{-\Delta}}{\sqrt{-\Delta}}} 
        g(x)
    + \int_{0}^{t}ds\, 
        {\textstyle
        \frac{\sin(t-s)\sqrt{-\Delta}}{\sqrt{-\Delta}}}
        F(s,x).
\end{align*}
Recall that a couple $(p,q)$ is called admissible if
$(\frac{1}{p},\frac{1}{q})$ belongs to the triangle
\begin{align*}
\textstyle
    \Big\lbrace 
    \Big(\frac{1}{p},\frac{1}{q}\Big) 
    \in\Big(0,\frac{1}{2}\Big]\times\Big(0,\frac{1}{2}\Big) 
    \,\Big|\,
    \frac{1}{p}\ge\frac{d-1}{2} 
    \Big(\frac{1}{2}-\frac{1}{q}\Big)
    \Big\rbrace \bigcup 
    \Big\lbrace\Big(0,\frac{1}{2}\Big)\Big\rbrace.
\end{align*}

\begin{figure}[h]\label{n>4}
    \centering
    \begin{tikzpicture}[scale=0.47][line cap=round,line join=round,>=triangle 45,x=1.0cm,y=1.0cm]
    \draw[->,color=black] (0.,0.) -- (8.805004895364174,0.);
    \foreach \x in {,2.,4.,6.,8.,10.}
    \draw[shift={(\x,0)},color=black] (0pt,-2pt);
    \draw[color=black] (8,0.08808504628984362) node [anchor=south west] {\large $\frac{1}{p}$};
    \draw[->,color=black] (0.,0.) -- (0.,8);
    \foreach \y in {,2.,4.,6.,8.}
    \draw[shift={(0,\y)},color=black] (-2pt,0pt);
    \draw[color=black] (0.11010630786230378,8) node [anchor=west] {\large $\frac{1}{q}$};
    \clip(-5,-2) rectangle (11.805004895364174,9.404473957611293);
    \fill[line width=2.pt,color=red,fill=red,fill opacity=0.15000000596046448] (0.45855683542994374,5.928142080369191) -- (5.85410415683891,5.8847171522290775) -- (5.864960388873938,4.136863794589543) -- cycle;
    \draw [line width=1.pt,dash pattern=on 5pt off 5pt,color=red] (0.,6.)-- (6.,6.);
    \draw [line width=1.pt,color=red] (6.,6.)-- (6.,4.);
    \draw [line width=1.pt,color=red] (6.,4.)-- (0.,6.);
    \draw [line width=0.8pt] (6.,4.)-- (6.,0.);
    \draw [->,line width=0.5pt,color=red] (1.8569381602310457,-0.8701903243695805) -- (2.891937454136701,4.811295161325333);
    \begin{scriptsize}
    \draw [fill=red] (0.,6.) circle (5pt);
    \draw[color=black] (-0.5709441083587729,6) node {\large $\frac{1}{2}$};
    \draw[color=black] (-1.5,4) node {\large $\frac{1}{2}- \frac{1}{d-1}$};
    \draw[color=black] (-0.5,-0.2) node {\large $0$};
    \draw[color=black] (6.2,-0.7) node {\large $\frac{1}{2}$};
    \draw[color=red] (1.8,-1.3) node {\large $\frac{1}{p} = \frac{d-1}{2} \left( \frac{1}{2} - \frac{1}{q} \right)$};
    \draw [color=red] (6.,6.) circle (5pt);
    \draw [fill=red] (6.,4.) circle (5pt);
    \end{scriptsize}
    \end{tikzpicture}
    \caption{Admissibility in dimension $d \ge 4$.}
    \end{figure}
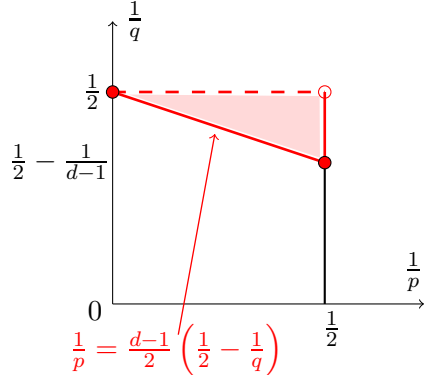

\begin{theorem}\label{thm strichartz}
Let $(p,q)$ and $(\tilde{p},\tilde{q})$ be two admissible
couples, and let
\begin{align*}
\textstyle
    \sigma\ge \frac{d+1}{2} 
    \left(\frac{1}{2}-\frac{1}{q} \right) 
    \quad\textnormal{and}\quad 
    \widetilde{\sigma}\ge \frac{d+1}{2} 
    \left(\frac{1}{2}-\frac{1}{\tilde{q}}\right).
\end{align*}
Then all solutions $u$ to the Cauchy problem 
\eqref{Cauchy problem} satisfy the following 
Stri\-chartz inequality:
\begin{align}\label{Strichartz}
    \|\nabla_{\mathbb{R}\times\mathbb{X}} u 
    \|_{L^p(I; H^{-\sigma,q}(\mathbb{X}))} 
    \lesssim 
        \|f\|_{H^1(\mathbb{X})} 
        + \|g\|_{L^2(\mathbb{X})} 
        + \|F\|_{L^{\tilde{p}'}
        (I;H^{\tilde{\sigma},\tilde{q}'}(\mathbb{X}))}.
\end{align}
\end{theorem}

\begin{remark}
As have already been observed on hyperbolic spaces, 
the admissible set for $\mathbb{X}$ is much larger than the
admissible set for $\mathbb{R}^{d}$, which corresponds only 
to the lower edge of the triangle. This is due to large scale
dispersive effects in negative curvature.
\end{remark}

The admissible range in \eqref{Strichartz} can be widened 
by using the Sobolev embedding theorem.
\begin{corollary}\label{corollary strichartz}
Assume that $(p,q)$ and $(\tilde{p},\tilde{q})$ are 
two couples corresponding to the square 
\begin{align*}
\textstyle
    \Big[0,\frac{1}{2}\Big] 
    \times \Big(0,\frac{1}{2}\Big)
    \bigcup\Big\lbrace\Big(0,\frac{1}{2}\Big)\Big\rbrace,
\end{align*}

\begin{figure}[h]
    \hspace{20pt}
    \centering
    \begin{tikzpicture}[scale=0.5][line cap=round,line join=round,>=triangle 45,x=1.0cm,y=1.0cm]
    \draw[->,color=black] (6.,0.) -- (8,0.);
    \foreach \x in {,1.,2.,3.,4.,5.,6.,7.,8.,9.,10.,11.}
    \draw[shift={(\x,0)},color=black] (0pt,-2pt);
    \draw[color=black] (8,0.1) node [anchor=south west] {\large $\frac{1}{p}$};
    \draw[->,color=black] (0.,6.) -- (0.,8);
    \foreach \y in {,1.,2.,3.,4.,5.,6.,7.}
    \draw[shift={(0,\y)},color=black] (-2pt,0pt);
    \draw[color=black] (0.1,8.1) node [anchor=west] {\large $\frac{1}{q}$};
    \clip(-4,-2) rectangle (11.868487651785372,7.652018714581741);
    \fill[line width=2.pt,color=red,fill=red,fill opacity=0.15000000596046448] (0.45855683542994374,5.928142080369191) -- (5.85410415683891,5.8847171522290775) -- (5.864960388873938,4.136863794589543) -- cycle;
    \fill[line width=2.pt,color=cyan,fill=cyan,fill opacity=0.10000000149011612] (0.13636385492209063,5.773698170331237) -- (5.837359246154291,3.8701454719028483) -- (5.876010062366035,0.10169089125781497) -- (0.16535196708089844,0.12101629936368694) -- cycle;
    \draw [line width=1.pt,dash pattern=on 3pt off 3pt,color=red] (0.,6.)-- (6.,6.);
    \draw [line width=1.pt,color=red] (6.,6.)-- (6.,4.);
    \draw [line width=1.pt,color=red] (6.,4.)-- (0.,6.);
    \draw [line width=1.pt,color=cyan] (0.,6.)-- (0.,0.);
    \draw [line width=1.pt,color=cyan] (6.,4.)-- (6.,0.);
    \draw [line width=1.pt,dash pattern=on 3pt off 3pt,color=cyan] (6.,0.)-- (0.,0.);
    \begin{scriptsize}
    \draw [fill=red] (0.,6.) circle (5pt);
    \draw[color=black] (-0.5709441083587729,6) node {\large $\frac{1}{2}$};
    \draw [color=red] (6.,6.) circle (5pt);
    \draw [fill=red] (6.,4.) circle (5pt);
    \draw [color=cyan] (0.,0.) circle (5pt);
    \draw[color=black] (-0.5,-0.2) node {\large $0$};
    \draw [color=cyan] (6.,0.) circle (5pt);
    \draw[color=black] (6,-1) node {\large $\frac{1}{2}$};
    \draw[color=black] (-1.5,4) node {\large $\frac{1}{2}- \frac{1}{d-1}$};
    \end{scriptsize}
    \end{tikzpicture}
    \vspace{-0.5cm}
    \caption{Extended admissibility in dimension $d\ge4$.}
    \end{figure}
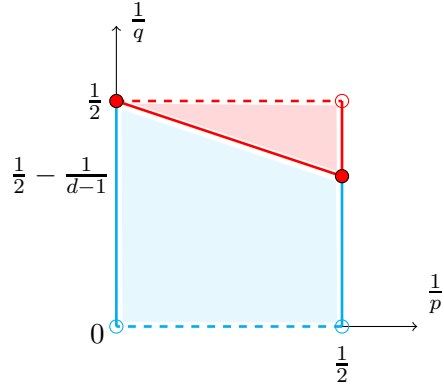

\noindent Let $\sigma, \tilde{\sigma} \in \mathbb{R}$ 
such that $\sigma \ge \sigma(p,q)$, where
\begin{align*}
\textstyle
    \sigma(p,q) 
    = \frac{d+1}{2} \Big(\frac{1}{2}-\frac{1}{q} \Big) 
    + \max\Big\lbrace0,\frac{d-1}{2} 
        \Big(\frac{1}{2}-\frac{1}{q}\Big)-\frac{1}{p} \Big\rbrace,
\end{align*}
and similarly 
$\tilde{\sigma} \ge \sigma(\tilde{p}, \tilde{q})$. 
Then the Strichartz inequality \eqref{Strichartz} holds for 
all solutions to the Cauchy problem \eqref{Cauchy problem}.
\end{corollary}

\subsection{Global well-posedness 
in $L^{p}(\mathbb{R},L^q(\mathbb{X}))$}
\label{subsection GWP}

By combining the classical fixed point scheme with the 
previous Strichartz inequalities, one obtains the global
well-posedness for the semi-linear equation
\begin{align}\label{semi-linear equation}
    \begin{cases}
        \partial_{t}^2 u(t,x)-\Delta u(t,x)=F(u(t,x)),\\[5pt]
        u(0,x)=f(x),\,\partial_{t}|_{t=0}u(t,x)=g(x).
    \end{cases}
\end{align}
on $\mathbb{X}$ with small initial data $f$ and $g$ and 
power-like nonlinearities $F$ satisfying
\begin{align*}
    |F(u)| \lesssim |u|^{\gamma}
    \quad\textnormal{and}\quad
    |F(u)-F(v)|
    \lesssim (|u|^{\gamma-1}+|v|^{\gamma-1})|u-v|
\end{align*}
where $\gamma>1$.
Let $\gamma_{c}=1+\frac{4}{d-1}$ be the conformal power.
The global existence of solutions to the semi-linear wave
equation \eqref{semi-linear equation} on $\mathbb{R}^{d}$
is related to the Strauss conjecture: the critical power 
$\gamma_{0}$, i.e., the infimum of all $\gamma\in(1,\gamma_{c}]$
such that \eqref{semi-linear equation} has global solutions
for small initial data, is the positive root of the quadratic equation
\begin{align*}
    (d-1)\gamma_{0}^{2} - (d+1)\gamma_{0}-2 =0
    \quad (d\ge2).
\end{align*}
In other words,
\begin{align*}
\textstyle
    \gamma_{0}
    = \frac{1}{2}+\frac{1}{d-1} + 
        \sqrt{(\frac{1}{2}+\frac{1}{d-1})^2+\frac{2}
            {d-1}}>1.
\end{align*}
We refer to \cite{Joh1979,Kat1980,Stra1981,GLS1997,Tat2001} 
and the references therein for more details about the Strauss
conjecture in the Euclidean setting.
In negative curvature, the global existence for small 
initial data has been proved, for any $\gamma\in(1,\gamma_{c}]$, 
on real hyperbolic spaces of dimension $d=3$
\cite{MeTa2011,MeTa2012} and then of any dimension $d\ge2$
\cite{AnPi2014}. In other words, there is no phenomenon  
analogous to the Strauss conjecture on such spaces.
Similar results were extended later to Damek-Ricci spaces, 
which contain all noncompact symmetric spaces of rank one 
\cite{APV2015},
and have been established recently on simply connected
complete Riemannian manifolds with strictly negative sectional
curvature \cite{SSW2019}, and on non-trapping asymptotically
hyperbolic manifolds \cite{SSWZ2020}.
Next theorem shows that the same phenomenon holds 
on general  noncompact symmetric spaces. More precisely, 
we prove that the semi-linear wave equation
\eqref{semi-linear equation} on $\mathbb{X}$ is globally
well-posed.\\

To state the following theorem, we need to introduce some
notation. Consider the following powers
\begin{align*}
    \gamma_1 =& 1+\frac{3}{d}
    \hspace{50pt}
    \gamma_2 = 1+\frac{2}{\frac{d-1}{2}+\frac{2}{d-1}}\\[5pt]
    \gamma_3 
        =& \begin{cases}
            1 + \frac{4}{d-2} 
            &\textnormal{if \,} d \le 5,\\[5pt]
            \frac{d-1}{2} + \frac{3}{d+1} - 
            \sqrt{( \frac{d-3}{2} + \frac{3}{d+1})^2 - 
            4\frac{d-1}{d+1}}  
            &\textnormal{if \,} d \ge 6,
    \end{cases}
\end{align*}
and the following curves
\begin{align*}
    \sigma_1(\gamma) 
        =&\ \frac{d+1}{4}-\frac{(d+1)(d+5)}{8d}
        \frac{1}{\gamma-\frac{d+1}{2d}},\\[5pt]
    \sigma_2(\gamma) 
        =&\ \frac{d+1}{4} - \frac{1}{\gamma -1},\quad 
    \sigma_3(\gamma) 
        = \frac{d}{2} - \frac{2}{\gamma -1}.
\end{align*}

\begin{theorem}\label{GWP}
The semi-linear Cauchy problem \eqref{semi-linear equation} is 
globally well-posed for small initial data in
$H^{\sigma,2}(\mathbb{X})\times H^{\sigma-1,2}(\mathbb{X})$
provided that
\begin{align*}
\begin{cases}
    \sigma>0 
    &\textnormal{if \,} 
        1< \gamma \le \gamma_1, \\[5pt]
    \sigma\ge\sigma_1 (\gamma) 
    &\textnormal{if \,}  
        \gamma_1< \gamma \le \gamma_2, \\[5pt]
    \sigma\ge\sigma_2 (\gamma) 
        &\textnormal{if \,}  
        \gamma_2\le \gamma \le \gamma_c, \\[5pt]
    \sigma\ge\sigma_3 (\gamma) 
    &\textnormal{if \,}  
        \gamma_c \le \gamma \le \gamma_3.
\end{cases}
\end{align*}
More precisely, in each case, there exists $2\le p,q<+\infty$
such that for any small initial data $(f,g)$ in
$H^{\sigma,2}(\mathbb{X})\times{H}^{\sigma-1,2}(\mathbb{X})$,
the Cauchy problem \eqref{semi-linear equation} has a unique solution in the Banach space
\begin{align*}
    \mathcal{C}(\mathbb{R};H^{\sigma,2}(\mathbb{X})) 
    \cap
    \mathcal{C}^{1}(\mathbb{R};H^{\sigma-1,2}(\mathbb{X})) 
    \cap
    L^{p}(\mathbb{R};L^{q}(\mathbb{X})).
\end{align*}
\end{theorem}

\section{Further results for Klein-Gordon equations}
\label{section Klein-Gordon}
The kernel estimates and dispersive estimates proved above
for the wave equation still hold if we replace the operator
$(-\Delta)^{-\frac{\sigma}{2}}e^{it\sqrt{-\Delta}}$ 
by $\mathbf{D}^{-\sigma}e^{it\mathbf{D}}$, where
\begin{align*}
    \mathbf{D}
    =\sqrt{-\Delta-|\rho|^2+\kappa^{2}}
    \quad\textnormal{with}\quad\,\kappa>0.
\end{align*}
Then, for every admissible couple $(p,q)$, the operator
\begin{align*}
    \mathbf{T}f(t,x)
    = \mathbf{D}_{x}^{-\frac{\sigma}{2}}e^{it\mathbf{D}_{x}} f(x)
\end{align*}
is again bounded from $L^2(\mathbb{X)}$ to 
$L^{p}(\mathbb{R};L^{q}(\mathbb{X}))$, 
and its adjoint
\begin{align*}
    \mathbf{T}^{*}F(x)
    = \int_{-\infty}^{+\infty}ds\,
        \mathbf{D}_{x}^{-\frac{\sigma}{2}}\,e^{-is\mathbf{D}_{x}}\,
        F(s,x)
\end{align*}
from $L^{p'}(\mathbb{R};L^{q'}(\mathbb{X}))$ to $L^2(\mathbb{X)}$.
While $L^2$ Sobolev spaces may be defined in terms of $\mathbf{D}$,
we need the operator
\begin{align*}
    \widetilde{\mathbf{D}}
    = \sqrt{-\Delta-|\rho|^2+\widetilde{\kappa}^{2}}
    \quad\textnormal{with}\quad\,\widetilde{\kappa}\ge|\rho|,
\end{align*}
in order to define $L^q$ Sobolev spaces when $q$ gets large.
As $\widetilde{\mathbf{D}}^{-\frac{\sigma}{2}}\circ
\mathbf{D}^{\frac{\sigma}{2}}$ is a topological
automorphism of $L^2(\mathbb{X})$, the operator
\begin{align*}
    \widetilde{\mathbf{T}}f(t,x)
    = \widetilde{\mathbf{D}}_{x}^{-\frac{\sigma}{2}}\,
        e^{it\mathbf{D}_{x}}f(x)
\end{align*}
is also bounded from $L^2(\mathbb{X)}$ to
$L^{p}(\mathbb{R};L^{q}(\mathbb{X}))$, and its adjoint
\begin{align*}
    \widetilde{\mathbf{T}}^{*}F(x)
    = \int_{-\infty}^{+\infty}ds\,
        \widetilde{\mathbf{D}}_{x}^{-\frac{\sigma}{2}}\,
        e^{-is\mathbf{D}_{x}}\,
        F(s,x)
\end{align*}
from $L^{p'}(\mathbb{R};L^{q'}(\mathbb{X}))$ to $L^2(\mathbb{X)}$,
hence 
\begin{align*}
    \widetilde{\mathbf{T}}\widetilde{\mathbf{T}}^{*}F(t,x)
    = \int_{-\infty}^{+\infty}ds\,
        \widetilde{\mathbf{D}}_{x}^{-\sigma}\,e^{i(t-s)\mathbf{D}_{x}}\,
        F(s,x)
\end{align*}
from $L^{p'}(\mathbb{R};L^{q'}(\mathbb{X}))$ to
$L^{\widetilde{p}}(\mathbb{R};L^{\widetilde{q}}(\mathbb{X}))$ 
for all admissible couples $(p,q)$ and $(\widetilde{p},\widetilde{q})$.
We deduce that \cref{thm strichartz} and \cref{corollary strichartz}
still hold for solutions to the inhomogeneous Klein-Gordon equation
\begin{align*}
    \begin{cases}
        \partial_{t}^2\,u(t,x)+\mathbf{D}^{2}\,u(t,x)=F(t,x), \\[5pt]
    u(0,x) =f(x),\,\partial_{t}|_{t=0} u(t,x)=g(x),
    \end{cases}
\end{align*}
and that the corresponding semi-linear equation is globally well-posed 
with low regularity data, see \cref{GWP}.

\appendix
\section{Oscillatory integral on $\mathfrak{a}$}
\label{Appendix A}

In this appendix, we prove the following lemma
which is used in the proofs of 
\cref{theorem subsection 1} and 
\cref{theorem subsection 3}.
Recall that $A=A(kx)$ is the $\mathfrak{a}$-component
of $kx\in\mathbb{X}$ in the Iwasawa decomposition, 
and that $C_{\Sigma}\in(0,\frac{1}{2}]$ is a 
fixed constant.

\begin{lemma}\label{Appendix A-lemma}
Let $s\in\mathbb{R}_{+}$ and $|t|\ge1$. 
Consider the oscillatory integral
\begin{align*}
    I^{-}(s,t,x)
    = \int_{\mathfrak{a}}d\lambda\,
    a_{0}(s,\lambda)\,e^{it\psi_{t}(\lambda)}
\end{align*}
where the phase is given by
\begin{align*}
\textstyle
    \psi_{t}(\lambda)
    = \sqrt{|\lambda|^2+|\rho|^2}
        + \langle\frac{A}{t},\lambda\rangle
\end{align*}
and the amplitude
\begin{align*}
    a_{0}(s,\lambda)
    =\chi_{0}^{\rho}(\lambda)\,
        |\mathbf{c}(\lambda)|^{-2}\,
        e^{-s\sqrt{|\lambda|^2+|\rho|^2}}
\end{align*}
vanishes unless $|\lambda|\le2|\rho|$.
Then, for all $x\in\mathbb{X}$ such that
$\frac{|x|}{|t|}\le{C_{\Sigma}}$, 
\begin{align*}
    |I^{-}(s,t,x)|
    \lesssim
    |t|^{-\frac{D}{2}}\,(1+|x|)^{\frac{D-\ell}{2}}\,
    e^{-\frac{|\rho|}{2}s}.
\end{align*}
\end{lemma}

\begin{remark}
The proof of this lemma is similar to the proof of
\cite[Theorem 3.1.(ii)]{Zha2021}, except that our
amplitude involves the general Plancherel density 
and in addition an exponential factor depending on $s$.  
\end{remark}

\begin{proof}
By symmetry, we may assume that $t\ge1$.
Recall that the critical point $\lambda_{0}$ of the
phase $\psi_t$ is given by
\begin{align}\label{Appendix-phase}
\textstyle
    (|\lambda_{0}|^2+|\rho|^2)^{-\frac{1}{2}} 
    \lambda_{0} = -\frac{A}{t}
\end{align}
and satisfies
\begin{align}\label{Appendix-phase estimate}
\textstyle
    |\lambda_{0}| 
    = |\rho| \frac{|A|}{|t|}
        (1-\frac{|A|^2}{t^2})^{-\frac{1}{2}}
    \le |\rho| \frac{|x|}{t}
        (1-\frac{|x|^2}{t^2})^{-\frac{1}{2}}
    < \frac{|\rho|}{\sqrt{3}},
\end{align}
as $|A|\le|x|$ and $\frac{|x|}{t}\le\frac{1}{2}$. 
Denote by 
\begin{align*}
    B(\lambda_{0},\eta)
    =\lbrace\lambda\in\mathfrak{a}\,|\,
        |\lambda-\lambda_{0}|\le\eta\rbrace
\end{align*}
the ball in $\mathfrak{a}$ centered at $\lambda_{0}$, 
where the radius $\eta$ will be specified later.
Notice that $|\lambda|<|\rho|+\eta$ for all 
$\lambda\in{B}(\lambda_{0},\eta)$.
Let $P_{\lambda}$ be the projection onto the vector
spanned by $\frac{\lambda}{|\lambda|}$.
Then $|\lambda|^2P_{\lambda}=\lambda\otimes\lambda$
and the Hessian matrix of $\psi_{t}$ is given by
\begin{align*}
    \Hess \psi_{t}(\lambda)
    =&\,
    (|\lambda|^2+|\rho|^2)^{-\frac{1}{2}}\,I_{\ell}
    - (|\lambda|^2+|\rho|^2)^{-\frac{3}{2}}\,
    \lambda\otimes\lambda\\[5pt]
    =&\,
    (|\lambda|^2+|\rho|^2)^{-\frac{3}{2}}\,
    \lbrace|\rho|^2P_{\lambda}+
    (|\lambda|^2+|\rho|^2)(I_{\ell}-P_{\lambda})
    \rbrace\\[5pt]
    =&\,
    (|\lambda|^2+|\rho|^2)^{-\frac{3}{2}}\,
    \left(\begin{array}{c|c}
    |\rho|^2  & 0 \\[5pt] \hline
    { } & { } \\
    0 &  ( |\lambda|^2 + |\rho|^2) I_{\ell-1} \\
    { } & { } \\
    \end{array}\right)
\end{align*}
which is a positive definite symmetric matrix.
Hence $\lambda_{0}$ is a nondegenerate critical point.
Since $\nablaA\psi_{t}(\lambda_{0})=0$, we write
\begin{align*}
    &\psi_{t}(\lambda)-\psi_{t}(\lambda_{0})\\
    &=  (\lambda-\lambda_{0})^{T}\,
    \Big\lbrace
    \underbrace{\int_{0}^{1}du\,(1-u)\,   
    \Hess\psi_{t}\big(\lambda_{0}+u(\lambda-\lambda_{0})\big)
    }_{\mathcal{M}(\lambda)}
    \Big\rbrace\,
    (\lambda-\lambda_{0}),
\end{align*}
where $\mathcal{M}(\lambda)$ belongs, for every
$\lambda\in{B}(\lambda_{0},\eta)$, to a compact subset
of the set of positive definite symmetric
matrices.
We introduce a new variable 
$\mu=\mathcal{M}(\lambda)^{\frac{1}{2}}
(\lambda-\lambda_{0})$, then 
$|\mu|^2=\psi_{t}(\lambda)-\psi_{t}(\lambda_{0})$
and $\mu=0$ if and only if $\lambda=\lambda_{0}$.
There exist $0<\widetilde{\eta}_{1}\le\widetilde{\eta}_{2}$ 
such that $\mu\in{B}(0,\widetilde{\eta}_{1})$ implies 
$\lambda\in{B}(\lambda_{0},\eta)$, and 
$\lambda\in{B}(\lambda_{0},\eta)$ implies  
$\mu\in{B}(0,\widetilde{\eta}_{2})$.
Notice that for every $k\in\mathbb{N}$, there exists
$C_{k}>0$ such that
\begin{align}\label{Appendix-estimate of M}
    |\nablaA^{k} \mathcal{M}(\lambda)^{\frac{1}{2}}|\le C_k
    \quad\forall\,\lambda\in{B}(\lambda_{0},\eta).
\end{align}
Denote by $j(\lambda)$ the Jacobian matrix such that 
$d\mu= \det[\,j(\lambda)\,]d\lambda$, then we can choose 
$\eta>0 $ small enough such that
\begin{align}\label{Appendix-estimate of j}
    \det[\,j(\lambda)\,] 
    > \frac{1}{2} \det[\,\mathcal{M}(\lambda)^{\frac{1}{2}}\,]
    \quad\forall\,\lambda\in{B}(\lambda_{0},\eta).
\end{align}
Now, we split up 
\begin{align*}
    I^{-}(s,t,x) 
    &=
    I_{0}^{-}(s,t,x)+I_{\infty}^{-}(s,t,x)\\[5pt]
    &=\int_{\mathfrak{a}}d\lambda\,
        \chi_{0}^{\eta}(\lambda)\,
        a_{0}(s,\lambda)\,
        e^{it \psi_{t} (\lambda)}
        +\int_{\mathfrak{a}}d\lambda\,
        \chi_{\infty}^{\eta}(\lambda)\,a_{0}(s,\lambda)\,
        e^{it \psi_{t} (\lambda)}
\end{align*}
where $\chi_{0}^{\eta}:\mathfrak{a}\rightarrow[0,1]$ is 
a smooth cut-off function which vanishes unless $|\lambda-\lambda_{0}|\le\frac{\eta}{2}$,
$\chi_{0}^{\eta}(\lambda)=1$ if 
$|\lambda-\lambda_{0}|\le\frac{\eta}{4}$,
and $\chi_{\infty}^{\eta}=1-\chi_{0}^{\eta}$. \\
\paragraph{\bf Estimate of $I_{0}^{-}$.}
We estimate $I_{0}^{-}$ by using the stationary phase analysis
described in \cite[Chap.VIII 2.3]{Ste1993}. Notice that
$\supp\chi_{0}^{\eta}\subset{B}(\lambda_{0},\eta)$.
By substituting $\psi_{t}(\lambda)=|\mu|^{2}+\psi_{t}(\lambda_{0})$, 
we get
\begin{align*}
    I_{0}^{-}(s,t,x) 
    = e^{it\psi_{t}(\lambda_{0})}\,\int_{\mathfrak{a}}d\mu\, 
        \widetilde{a}(s,\lambda(\mu))\,e^{it |\mu|^{2}}
\end{align*}
where the amplitude
\begin{equation*}\begin{aligned}
    \widetilde{a}(s,\lambda(\mu)) 
    &= \chi_{0}^{\eta}(\lambda(\mu))\, 
        \chi_{0}^{\rho}(\lambda(\mu))\,
        \boldsymbol{\pi}(\lambda(\mu))^{2}\\
     &\times   |\mathbf{b}(\lambda(\mu))|^{-2}\,
        e^{-s\sqrt{|\lambda(\mu)|^2+|\rho|^2}}\,
        \det[\,j(\lambda(\mu))\,]^{-1}
\end{aligned}\end{equation*}
is smooth and compactly supported in $B(0,\widetilde{\eta}_{2})$.
We deduce, from \eqref{Appendix-estimate of M} and 
\eqref{Appendix-estimate of j} that $\widetilde{a}(s,\lambda(\mu))$ 
is bounded, together with all its derivatives. 
Let 
$\chi_{\widetilde{\eta}_{2}}\in\mathcal{C}_{c}^{\infty}(\mathfrak{a})$ 
be a bump function such that $\chi_{\widetilde{\eta}_{2}}=1$ on
$B(0,\widetilde{\eta}_{2})$. Then
\begin{align*}
    I_{0}^{-}(s,t,x) 
    = e^{ it \psi_{t} (\lambda_{0})} \int_{\mathfrak{a}}d\mu\,
        \chi_{\widetilde{\eta}_{2}}(\mu)\,
        e^{it|\mu|^{2}}\,e^{-|\mu|^{2}} 
        \lbrace{e^{|\mu|^{2}}\,\widetilde{a}(s,\lambda(\mu))}\rbrace.
\end{align*}
Consider the Taylor expansion at the origin
\begin{align*}
    e^{|\mu|^{2}}\,\tilde{a}(s,\lambda(\mu))
    = \sum_{|k|\le2M} c_{k}\,\mu^{k} + R_{2M}(\mu),
\end{align*}
where $2M$ is the smallest even integer $>D$. By expanding
\begin{align*}
\boldsymbol{\pi}(\lambda(\mu))^2=\prod_{\alpha\in\Sigma_r^+}\bigl[\langle\alpha,\lambda_0\rangle+\langle\alpha,\mathcal{M}(\lambda(\mu))^{-\frac12}\mu\rangle\bigr]^2,
\end{align*}
we see that its coefficients satisfy
\begin{align}\label{Appendix-ck}
\textstyle
    |c_{k}|
    \lesssim
    |\lambda_0|^{N_{|k|}}\,
    (1+s)^{|k|}\,
    e^{-s\sqrt{|\lambda_{0}|^2+|\rho|^2}}
    \lesssim
    (\frac{|x|}{t})^{N_{|k|}}\,
    (1+s)^{|k|}\,e^{-|\rho|s},
\end{align}
where $N_{|k|}=\max\,\{0,2\,|\Sigma_r^+|-|k|\}$, while the remainder satisfies
\begin{align}\label{Appendix-Taylor remainder}
    |\nablaA^{n}\,R_{2M}(\mu)|
    \lesssim 
    |\mu|^{2M+1-n}\,
    (1+s)^{2M+1+n}\,e^{-|\rho|s}
    \quad\forall\,0\le{n}\le{2M+1}.
\end{align}
By substituting this expansion in the above integral,
$I_{0}^{-}(s,t,x)$ is the sum of following three terms:
\begin{align*}
    I_1 
    = \sum_{|k|\le2M} c_k 
        \int_{\mathfrak{a}}d\mu\, 
        \mu^{k}\,e^{it|\mu|^{2}}\,e^{-|\mu|^{2}} \\[5pt]
    I_2 
    = \int_{\mathfrak{a}}d\mu\,
        \chi_{\widetilde{\eta}_{2}}(\mu)\,R_{2M}(\mu)\,
        e^{it|\mu|^{2}}\,e^{-|\mu|^{2}},
\end{align*}
and
\begin{align*}
    I_3  
    = \sum_{|k|\le2M} c_k  
        \int_{\mathfrak{a}}d\mu\, 
        \lbrace\chi_{\widetilde{\eta}_{2}}(\mu)-1\rbrace\,\mu^{k}\,
        e^{it|\mu|^{2}}\,e^{-|\mu|^{2}}.
\end{align*}
To estimate $I_{1}$, we write 
\begin{align*}
    I_1 
    = \sum_{|k|\le2M} c_k \prod_{j=1}^{\ell} 
        \int_{-\infty}^{+\infty}d\mu_{j}\,
        e^{it\mu_{j}^{2}}\,e^{-\mu_{j}^2}\,\mu_{j}^{k_{j}}
\end{align*}
where
\begin{align*}
    \int_{-\infty}^{+\infty}d\mu_{j}\,
        e^{-(1-it)\mu_{j}^{2}}\,\mu_{j}^{k_{j}}=0
\end{align*}
if $k_{j}$ is odd, while
\begin{align*}
    \int_{-\infty}^{+\infty}d\mu_{j}\,
        e^{-(1-it)\mu_{j}^{2}}\,\mu_{j}^{k_{j}}
    = 2\,(1-it)^{-\frac{k_{j}+1}{2}}
        \int_{0}^{+\infty}dz_{j}\,e^{-z_{j}^{2}}\,
        z_{j}^{k_{j}}
\end{align*}
by a change of contour if $k_{j}$ is even.
We deduce from \eqref{Appendix-ck}
\begin{align*}
  |I_1|&\lesssim\sum_{|k|\le2|\Sigma_r^+|}t^{-\frac{|k|+\ell}2}\bigl(\tfrac{|x|}t\bigr)^{2|\Sigma_r^+|-|k|}(1+s)^{|k|}e^{-|\rho|s}\\
  &+\sum_{2|\Sigma_r^+|<|k|\le2M}t^{-\frac{|k|+\ell}2}(1+s)^{|k|}e^{-|\rho|s}\\
  &\lesssim t^{-\frac D2}(1+|x|)^{\frac{D-\ell}2}e^{-\frac{|\rho|}2s}
  \end{align*}
since $D=\ell+2|\Sigma_r^+|$ and $\frac{|x|}{t}\le{C}_{\Sigma}$.
Next, we perform $M$ integrations by parts based on
\begin{align}\label{Appendix-IBP I2}
\textstyle
    e^{it|\mu|^{2}}
    = -\frac{i}{2t}\,\sum_{j=1}^{\ell}
        \frac{\mu_{j}}{|\mu|^2}\,\frac{\partial}{\partial\mu_{j}}\, 
        e^{it|\mu|^{2}}
\end{align}
and obtain
\begin{align*}
    |I_{2}|
    \lesssim t^{-M}\,(1+s)^{3M+1}\,e^{-|\rho|s}
    \lesssim t^{-M}\,e^{-\frac{|\rho|}{2}s}
\end{align*}
according to \eqref{Appendix-Taylor remainder}.
Finally, as $\mu\mapsto\mu^{k}\,e^{-|\mu|^{2}}\,(\chi_{\widetilde{\eta}_2}-1)$
is exponentially decreasing and vanishes near the origin, we perform
$N\ge\frac{D}{2}$ integrations by parts based on \eqref{Appendix-IBP I2} 
again and obtain
\begin{align*}
    |I_{3}| \lesssim
        t^{-N}\,e^{-\frac{|\rho|}{2}s}.
\end{align*}
By summing up the estimates of $I_1$, $I_2$ and $I_3$, we deduce that
\begin{equation}\label{Appendix-I-}\begin{aligned}
    |I_{0}^{-}(s,t,x)| \,
      \lesssim\,
        t^{-\frac{D}{2}}\,(1+|x|)^{\frac{D-\ell}{2}}\,
        e^{-\frac{|\rho|}{2}s}.
\end{aligned}\end{equation}
\paragraph{\bf Estimate of $I_{\infty}^{-}$.}
Since the phase $\psi_{t}$ has a unique critical point $\lambda_{0}$ 
which is defined by \eqref{Appendix-phase} and satisfies 
\eqref{Appendix-phase estimate}, then for all 
$\lambda\in\supp\chi_{\infty}^{\eta}$, 
we have $\nablaA\psi_{t}(\lambda)\neq0$.
In order to get large time decay, we estimate 
\begin{align*}
    I_{\infty}^{-}(s,t,x)=
    \int_{\mathfrak{a}}d\lambda\,
        \chi_{\infty}^{\eta}(\lambda)\,a_{0}(s,\lambda)\,
        e^{it \psi_{t} (\lambda)}
\end{align*}
by using several integrations by parts based on
\begin{align*}
\textstyle
    e^{it\psi_{t}(\lambda)} 
    = \frac{1}{it}\,\widetilde{\psi}_{0}(\lambda)^{-1}\,
        \sum_{j=1}^{\ell}\,
        \big(\frac{\lambda_{j}}{\sqrt{|\lambda|^2+|\rho|^2}}
            +\frac{A_{j}}{t}\big)
        \frac{\partial}{\partial\lambda_{j}}\,e^{it \psi_{t}(\lambda)},
\end{align*}
where
\begin{align*}
\textstyle
    \widetilde{\psi}_{0}(\lambda) 
    = \big|\frac{\lambda}{\sqrt{|\lambda|^2+|\rho|^2}}
            +\frac{A}{t}\big|^2
\end{align*}
is a smooth function, which is bounded from below on the compact set
$(\supp\chi_{\infty}^{\eta})\cap(\supp\chi_{0}^{\rho})$,
uniformly in $\frac{A}{t}$.
After performing $N$ such integrations by parts, 
$I_{\infty}^{-}(s,t,x)$ becomes
\begin{align*}
    &\const\,(it)^{-N}
    \int_{\mathfrak{a}}d\lambda\,e^{it\psi_{t}(\lambda)}\\
    &\times\Big\lbrace -\sum_{j=1}^{\ell}\,
        \frac{\partial}{\partial\lambda_{j}}\circ
            \big[\widetilde{\psi}_{0}(\lambda)^{-1}
                {\textstyle
                \big(
                \frac{\lambda_{j}}{\sqrt{|\lambda|^2+|\rho|^2}}
                +\frac{A_{j}}{t}
                \big)}
            \big]
        \Big\rbrace^{N}
        \Big\lbrace\chi_{\infty}^{\eta}(\lambda)\,a_{0}(s,\lambda)\Big\rbrace
\end{align*}
where the last integral is bounded from above by
$(1+s)^{N}\,e^{-|\rho|s}\lesssim\,e^{-\frac{|\rho|}{2}s}$.
Hence
\begin{align}\label{Appendix-I+}
    |I_{\infty}^{-}(s,t,x)|
    \lesssim t^{-N}\,e^{-\frac{|\rho|}{2}s}
\end{align}
for every $N\in\mathbb{N}$.
By combining \eqref{Appendix-I-} and \eqref{Appendix-I+},
we conclude that
\begin{align*}
    |I^{-}(s,t,x)| \lesssim
        t^{-\frac{D}{2}}\,(1+|x|)^{\frac{D-\ell}{2}}\,
        e^{-\frac{|\rho|}{2}s}.
\end{align*}
\end{proof}

\section{Hadamard parametrix on symmetric spaces}
\label{Appendix B}

Let $\Phi_{v}$ be the $K$-bi-invariant convolution 
kernel of the operator $\cos(v\sqrt{-\Delta})$
whose spherical Fourier transform is given by
$\widetilde{\Phi}_{v}(\lambda)
=\cos(v\sqrt{|\lambda|^2+|\rho|^2})$.
Then $\Phi_{v}(x)$ solves the following Cauchy
problem
\begin{align*}
    \begin{cases}
        \partial_{v}^{2}\,U(v,x)-\Delta_{x}\,U(v,x)=0, \\[5pt]
        U(0,x) = \delta_{0}(x),\
        \partial_{v}|_{v=0} U(v,x)=0.
    \end{cases}
\end{align*}
We are looking for the asymptotic expansion of the kernel $\Phi_{v}$.
Recall that $J$ denotes the Jacobian of the exponential map 
from $\mathfrak{p}$ equipped with the Lebesgue measure to
$\mathbb{X}$ equipped with the Riemannian measure. It satisfies
\begin{align*}
    J(H)^{-\frac{1}{2}}
    &= \prod_{\alpha\in\Sigma^{+}}
    \Big( \frac{\langle\alpha,H\rangle}
    {\sinh\langle\alpha,H\rangle} \Big)^{\frac{m_{\alpha}}{2}}\\
    &\asymp \Big\lbrace\prod_{\alpha\in\Sigma^{+}}
    (1+\langle\alpha,H\rangle)^{\frac{m_{\alpha}}{2}}
    \Big\rbrace
    e^{-\langle\rho,H\rangle}
    \quad \forall H\in\overline{\mathfrak{a}^{+}}.
\end{align*}

Let $f$ be a $K$-bi-invariant function on G, then $f$ may also be considered as an $\Ad{K}$-invariant function on $\mathfrak{p}$ or a $W$-invariant function on $\mathfrak{a}$. 
Recall that $\DeltaP$ and $\DeltaA$ denote the usual Laplacian
on the Euclidean spaces $\mathfrak{p}$ and 
$\mathfrak{a}\subset\mathfrak{p}$.
The radial part of the Laplacian $\Delta$ on $\mathbb{X}$
is defined by
\begin{align*}
    \Deltarad f(H)
    = \DeltaA f(H) + \sum_{\alpha\in\Sigma^{+}}
       m_{\alpha} \coth\langle\alpha,H\rangle\partial_{\alpha} f(H)
    \quad \forall H\in\mathfrak{a}^{+},
\end{align*}
and that of $\DeltaP$ is given by
\begin{align*}
    \DeltaradP f(H)
    = \DeltaA f(H) + \sum_{\alpha\in\Sigma^{+}}
       m_{\alpha} \langle\alpha,H\rangle^{-1}
       \partial_{\alpha} f(H)
    \quad \forall H\in\mathfrak{a}^{+},
\end{align*}
see \cite[Propositions 3.9 and 3.11]{Hel2000}.
The following proposition provides a relation between 
$\Deltarad$ and $\DeltaradP$, it allows us to simplify
the computations about the parametrix.

\begin{proposition}\label{Annexe-deltaa to deltap}
Let  $f\in\mathcal{C}^{\infty}(\mathfrak{a})$ be a 
$W$-invariant function. Then
\begin{align*}
    \big[J(H)^{\frac{1}{2}} \circ\Deltarad\circ
    J(H)^{-\frac{1}{2}}\big]  f(H) 
    = \big[\DeltaradP + \omega(H)\big]  f(H) 
    \quad \forall H\in\mathfrak{a},
\end{align*}
where
\begin{align*}
    \omega(H) 
    &=\sum_{\alpha\in\Sigma^{+}}
        \frac{m_{\alpha}}{2} 
        \Big(\frac{m_{\alpha}}{2}-1\Big) |\alpha|^2 
        \Big\lbrace \frac{1}{\langle\alpha,H\rangle^{2}}
        -\frac{1}{\sinh^2\langle\alpha,H\rangle} 
        \Big\rbrace\\[5pt]
    &+\sum_{\substack{\alpha\in\Sigma^{+}\\ 
        \textnormal{s.t.}\ 2\alpha\in \Sigma^{+}}}
        \frac{m_{\alpha}m_{2\alpha}}{2} |\alpha|^2
        \Big( \frac{m_{\alpha}}{2}-1\Big) |\alpha|^2 
        \Big\lbrace
        \frac{1}{\langle\alpha,H\rangle^{2}}
        -\frac{1}{\sinh^2\langle\alpha,H\rangle} 
        \Big\rbrace -|\rho|^2
\end{align*}
is a smooth $W$-invariant function, which is uniformly bounded
together with all its derivatives.
\end{proposition}

\begin{proof}
Notice that
\begin{align*}
    \big[J(H)^{\frac{1}{2}}\circ\Deltarad\circ
    J(H)^{-\frac{1}{2}}\big]f(H)
    =J(H)^{\frac{1}{2}}
        \big(\Deltarad J^{-\frac{1}{2}}\big)(H) f(H)&\\
        + \underbrace{\Deltarad f(H)
        + 2 J(H)^{\frac{1}{2}}
            \big(\nablaA J^{-\frac{1}{2}}\big)(H)
            \cdot \nablaA f(H)}_{\DeltaradP f(H)}&\,,
\end{align*}
since
\begin{align*}
    J(H)^{\frac{1}{2}}
    \big(\nablaA J^{-\frac{1}{2}}\big)(H)
    = \sum_{\alpha\in\Sigma^{+}} \frac{m_{\alpha}}{2}
        \Big\lbrace 
        \frac{1}{\langle\alpha,H\rangle}
        -\coth\langle\alpha,H\rangle
        \Big\rbrace \alpha.
\end{align*}
We deduce from the next lemma that
\begin{align*}
    J(H)^{\frac{1}{2}}
    \big(\Deltarad J^{-\frac{1}{2}}\big)(H)
    = \omega(H) \quad \forall H\in\mathfrak{a},
\end{align*}
and this concludes the proof.
\end{proof}

\begin{lemma}[Cancellations]
The following equations hold for all $H\in\mathfrak{a}$:
\begin{align*}
    \sum_{\alpha,\beta\in\Sigma^{+},\
        \mathbb{R}\alpha\neq\mathbb{R}\beta}
        m_{\alpha}m_{\beta}
        \frac{\langle\alpha,\beta\rangle}
        {\langle\alpha,H\rangle\langle\beta,H\rangle}
        =0
\end{align*}
\begin{align*}
    \sum_{\alpha,\beta\in\Sigma^{+},\
        \mathbb{R}\alpha\neq\mathbb{R}\beta}
        m_{\alpha}m_{\beta} 
        \langle\alpha,\beta\rangle
        \big(\coth\langle\alpha,H\rangle \coth\langle\beta,H\rangle-1\big)=0
\end{align*}
\end{lemma}

\begin{proof}
See \cite[Appendix]{HaSt2003} for a detailed proof of this "\textit{folklore}" result.
\end{proof}
\vspace{10pt}

Recall that $\lbrace{R_{+}^{z}\,|\,z\in\mathbb{C}}\rbrace$
denotes the analytic family of Riesz distributions 
on $\mathbb{R}$ defined by
\begin{align*}
    R_{+}^{z}(r) =
    \begin{cases}
    \Gamma (z)^{-1} r^{z-1}
    &\textnormal{if \,} r>0, \\[5pt]
    0 &\textnormal{if \,} r\le0.
    \end{cases}
\end{align*}
Consider the asymptotic expansion
\begin{align}\label{Appendix-Expansion}
    \Phi_{v}(\exp H)
    = J(H)^{-\frac{1}{2}}
        \sum_{k=0}^{+\infty} 4^{-k}\,
        |v|\, U_{k}(H)\,
        R_{+}^{k-\frac{d-1}{2}}(v^2-|H|^2)
\end{align}
where $U_{0}$ is a constant such that 
$U_{0}\,J(H)^{-\frac{1}{2}}\,|v|\,
R_{+}^{-\frac{d-1}{2}}(v^2-|H|^2)
\rightarrow\delta_{0}(H)$ as $v\rightarrow0$ 
and $U_{k}\in\mathcal{C}^{\infty}(\mathfrak{p})$ are
smooth $\Ad{K}$-invariant functions.
By expanding
\begin{align*}
    0=&\ J(H)^{\frac{1}{2}}
        \big[\partial_{v}^{2}-\Deltarad]
        \Phi_{v}(\exp H)\\[5pt]
    =&\ \sum_{k=0}^{+\infty} 4^{-k}\,
        \big[\partial_{v}^{2}-\DeltaradP-\omega(H)\big]
        \Big\lbrace
        |v|\, U_{k}(H)\,
        R_{+}^{k-\frac{d-1}{2}}(v^2-|H|^2)
        \Big\rbrace,
\end{align*}
we deduce
\begin{align}\label{Appendix-incduction}
    \big[(k+1)+\partial_{H}\big] U_{k+1}(H)
    = \big[\DeltaradP + \omega(H)\big] U_{k}(H),
\end{align}
for every $k\in\mathbb{N}$.
In other words, 
\begin{align}\label{Appendix-integral induction}
    U_{k+1}(H)
    = \int_{0}^{1}ds\,s^{k}\,
        \big[\DeltaradP + \omega(sH)\big] U_{k}(sH).
\end{align}
As $\omega$ and all its derivatives are uniformly bounded,
we obtain
\begin{align}\label{Appendix-estimate Uk}
    \nablaP^{n} U_{k} = O(1)
\end{align}
for any $k,n\in\mathbb{N}$.\\

Next, by resuming the proof of \cite[Proposition 27]{Ber1977} 
with our asymptotic expansion \eqref{Appendix-Expansion}, 
we deduce that the remainder of the truncated expansion 
\begin{equation}\label{Appendix-Truncated Expansion}\begin{aligned}
    \Phi_{v}(\exp H)
    &= J(H)^{-\frac{1}{2}}
        \sum_{k=0}^{N} 4^{-k}\,
        |v|\, U_{k}(H)\,
        R_{+}^{k-\frac{d-1}{2}}(v^2-|H|^2)\\
        &+ E_{N}(v,\exp{H}) 
\end{aligned}\end{equation}
is a solution to the inhomogeneous Cauchy problem
\begin{align*}
    \begin{cases}
        \big[\partial_{v}^{2}-\Deltarad \big]
        E_{N}(v,\exp{H}) 
        = J(H)^{-\frac{1}{2}}
        \widetilde{U}_{N}(v,H), \\[5pt]
         \lim_{v \rightarrow 0} E_{N}(v,\exp{H}) = 0,\
         \lim_{v \rightarrow 0} 
         \frac{\partial E_N}{\partial v}(v,\exp{H}) = 0,
    \end{cases}
\end{align*}
where
$\widetilde{U}_{N}(v,H) 
= - 4^{-N}\,|v|\,U_{N}(H)\,
R_{+}^{N-\frac{d-1}{2}}(v^2-|H|^2)$.
Hence, by Duhamel's formula
\begin{align*}
    E_{N}(v,\exp H)
    = \int_{0}^{v}du\,
        {\textstyle
        \frac{\sin(v-u)\sqrt{-\Deltarad}}
        {\sqrt{-\Deltarad}}}
    \lbrace
    J(H)^{-\frac{1}{2}}\,\widetilde{U}_{N}(u,H)
    \rbrace.
\end{align*}
According to next lemma and by $L^2$ conservation, we have
\begin{align*}
     |E_{N}(v,\exp H)|
     \lesssim&\ 
        e^{-\langle \rho,H \rangle}\,
        \|E_{N}(v,\cdot)\|_{H^{2\sigma+1}(\mathbb{X})}
        \\[5pt]
     \lesssim&\ 
        e^{-\langle \rho,H \rangle}\, 
        \int_{0}^{v} du\, 
        \|\widetilde{U}_{N}(u,\cdot)J^{-\frac{1}{2}}
        \|_{H^{2\sigma}(\mathbb{X})}
\end{align*}
provided that $2\sigma+1>\frac{d}{2}$, and
\begin{align*}
    \|\widetilde{U}_{N}(u,\cdot) J^{-\frac{1}{2}}
    \|_{H^{2\sigma}(\mathbb{X})}^2
    =&\ \|\Delta^{\sigma} 
        \lbrace 
        \widetilde{U}_{N}(u,\cdot) J^{-\frac{1}{2}}
        \rbrace
        \|_{L^{2}(\mathbb{X)}}^2\\[5pt]
    =&\ \const
        \int_{\mathfrak{p}} dX\,
        \big|J(X)^{\frac{1}{2}}(\Deltarad)^{\sigma} 
        \lbrace 
        J(X)^{-\frac{1}{2}} \widetilde{U}_{N}(u,X)
        \rbrace \big|^{2} \\[5pt]
    =&\ \const\, \int_{\mathfrak{p}} dX\
        \big| [\DeltaradP+\omega(X)]^{\sigma}
        (\widetilde{U}_{N}(u,X)) \big|^{2}\\[5pt]
    \lesssim&\ 
        u^{2}\, \sum_{j=0}^{2\sigma}
        \int_{\lbrace{X}\in\mathfrak{p}\,|\,|X|<u\rbrace}
        dX\ \big|
        \nablaP^{j} (u^2-|X|^2)^{N-\frac{d+1}{2}}\big|^{2}
        \\[5pt]
    \lesssim&\ (1+u)^{4N-d},
\end{align*}
since $\omega$ and $U_{N}$, together with all their derivatives 
are uniformly bounded.
Here we assume $N>\frac{d+1}{2}+2\sigma$ to avoid
possible singularities on the sphere $|X|=u$. 
We may set $2\sigma=\big[\frac{d+2}{4}\big]$ 
and $N>d+1$.
Finally, we obtain
\begin{align*}
    |E_{N}(v,\exp H)| 
    \lesssim e^{-\langle \rho,H \rangle}
    \int_{0}^{v} du\ (1+u)^{2N-\frac{d}{2}}
    \lesssim 
    (1+v)^{2N-\frac{d}{2}+1} e^{-\langle \rho,H \rangle}.
\end{align*}

\begin{lemma}[Sobolev embedding theorem for 
$K$-bi-invariant functions on $\mathbb{X}$]
Let $\sigma>\frac{d}{2}$ be an integer. Then
\begin{align*}
    | f(\exp H) | 
    \lesssim e^{-\langle \rho,H \rangle} 
    \|f\|_{H^{\sigma}(\mathbb{X})}
    \quad\forall\,H\in\frakACL
\end{align*}
for all $K$-bi-invariant functions 
$f\in{H}^{\sigma}(\mathbb{X})$.
\end{lemma}

\begin{proof}
See \cite[Lemma 2.3]{Ank1992}.
\end{proof}
\vspace{10pt}
Notice that, for all $N>\frac{d}{2}$, we have
\begin{align*}
    \Big| J(H)^{-\frac{1}{2}}\!
        \sum_{k=[d/2]+1}^{N}\!4^{-k}\,|v|\,U_{k}(H)
        R_{+}^{k-\frac{d-1}{2}}(v^2-|H|^2)\Big|
    \lesssim (1+v)^{2N-\frac{d+\ell}{2}}\,
        e^{-\langle\rho,H\rangle}.
\end{align*}
Then we deduce the following corollary.

\begin{corollary}\label{Appendix-corollary}
The $K$-bi-invariant convolution kernel $\Phi_{v}$
has the asymptotic expansion
\begin{equation}\label{Appendix-final expansion}\begin{aligned}
    \Phi_{v}(\exp H)
    &= J(H)^{-\frac{1}{2}}
        \sum_{k=0}^{[d/2]} 4^{-k}\,|v|\,U_{k}(H)\,
        R_{+}^{k-\frac{d-1}{2}}(v^2-|H|^2)\\
    &+ E_{\Phi}(v,H),
\end{aligned}\end{equation}
where the remainder satisfies
\begin{align}\label{Appendix-final estimate}
    |E_{\Phi}(v,H)| \lesssim
    (1+v)^{3(\frac{d}{2}+1)} e^{-\langle\rho,H\rangle}
    \quad\forall\,H\in\frakACL.
\end{align}
\end{corollary}
\vspace{10pt}

\section{Asymptotic expansion of the Poisson kernel}
\label{Appendix C}

The Hadamard parametrix described above provides an 
asymptotic development of the kernel of the truncated
Poisson operator
\begin{align*}
    A_{\tau}
    = \int_{-\infty}^{+\infty} dv\ 
    \chi_{T}(v) p_{\tau}^{\mathbb{R}}(v)
    \cos{(v\sqrt{-\Delta})}.
\end{align*}
Here $\tau=s-it$ with $s\in(0,1]$ and $t\in\mathbb{R}^{*}$,
\begin{align*}
    T=
    \begin{cases}
    \sqrt{2}
    &\textnormal{if \,} 0<|t|\le1,\\[5pt]
    \sqrt{2}|t|
    &\textnormal{if \,} |t|\ge1,
    \end{cases}
\end{align*}
$\chi:\mathbb{R}\rightarrow[0,1]$ is a smooth even
cut-off function such that $\chi=1$ on $[-1,1]$ and
$\supp\chi\subset[-2\sqrt{2},2\sqrt{2}]$,
$\chi_{T}(v)=\chi(\frac{v}{2T})$ is supported in 
$[-2\sqrt{2}T,2\sqrt{2}T]\subset(-3T,3T)$, and 
$p_{\tau}^{\mathbb{R}}(v) = \frac{1}{\pi} \frac{\tau}{\tau^2+v^2}$
is the Poisson kernel on $\mathbb{R}$ 
(with complex time $\tau$). Notice that $|\tau|\le{T}$.
By resuming and improving slightly \cite[Lemma 3.3]{CGM2001}, 
we deduce the following parametrix for the kernel of 
$A_{\tau}$.

\begin{proposition}\label{Appendix-prop atau}
The kernel $a_{\tau}$ of the operator $A_{\tau}$ is a
smooth $K$-bi-invariant function on $G$, which is supported
in the ball of radius $3T$ in $\mathbb{X}$. Moreover
\begin{align}
    a_{\tau}(\exp H) 
    &= \frac{\tau}{\pi}\,J(H)^{-\frac{1}{2}}\,
        \sum_{k=0}^{[d/2]} 4^{-k}\, U_{k}(H)\,
        \Gamma\Big(\frac{d+1}{2}-k\Big)\,
        (|H|^2+\tau^2)^{\frac{d+1}{2}-k}
        \nonumber\\
    &+ E(\tau,H)
\label{Appendix-atau expansion}\end{align}
where the remainder satisfies
\begin{align}\label{Appendix-atau error}
    |E(\tau,H)| 
    \lesssim |T|^{3(\frac{d}{2}+1)}\,
    (\log{T}-\log{s})\,e^{-\langle\rho,H\rangle}
    \quad\forall\,H\in\frakACL.
\end{align}
Here the coefficients $U_{k}$ are the same as in
\cref{Appendix-corollary} and are uniformly bounded.
\end{proposition}

\begin{remark}
The proof of \cref{Appendix-prop atau} is similar to the
proof of Lemma 3.3 in \cite{CGM2001}. Notice that the latter
statement contains a minor error in the Gamma factor and
that our estimates contain an additional exponential decay,
which is crucial for the dispersive estimates.\\
\end{remark}

Let us state and reprove some technical results borrowed
from \cite{CGM2001}.

\begin{lemma}
Let $n\ge1$ and $\gamma \in \mathbb{R}_{+}$. Then
\begin{align}
    &|z|^{2\gamma-n} \int_{0}^{3T} dr\ 
    r^{n-1} |r^2+z^2|^{-\gamma}\,\asymp\nonumber\\
    &\asymp\,
    \begin{cases}
        \,(\frac{|z|}{\re{z}})^{\gamma-1}
        &\textnormal{if}\
        \gamma>1\ \textnormal{and}\ n<2\gamma,\\[5pt]
        \,(\frac{T}{|z|})^{n-2} 
        +\log(\frac{|z|}{\re{z}}) 
        &\textnormal{if}\
        \gamma=1\ \textnormal{and}\ n>2, \\[5pt]
        \,1+\log (\frac{T}{\re{z}})
        &\textnormal{if}\ 
        \gamma=1\ \textnormal{and}\ n=2, \\[5pt]
        \,1+\log(\frac{|z|}{\re{z}})
        &\textnormal{if}\
        \gamma=1\ \textnormal{and}\ n<2.
    \end{cases}
\label{Appendix-lemma B1}\end{align}
for every $z\in\mathbb{C}$ such that $\re{z}>0$
and $|z|\le{T}$.
\end{lemma}

\begin{proof}
Write $z=|z|e^{i\theta}$ in polar coordinates, with
$\theta\in[-\frac{\pi}{2},\frac{\pi}{2}]$. 
By performing the change of variables $r=|z|w$, 
the left hand side of \eqref{Appendix-lemma B1} becomes
\begin{align*}
    I = \int_{0}^{\frac{3T}{|z|}} dw\
    w^{n-1} |w^2+ e^{i2\theta}|^{-\gamma}.
\end{align*}
Notice that $\frac{3T}{|z|}>2$ and that
\begin{align}\label{Appendix-dem B1 ineq}
    |w^2-1| 
    \le |w^2+ e^{i2\theta}|
    \le |w^2+1|.
\end{align}
Let us split up $I=I_{0}+I_{1}+I_{\infty}$ according to
\begin{align*}
    \int_{0}^{\frac{3T}{|z|}} dw\
    = \int_{0}^{\frac{1}{2}} dw\
    + \int_{\frac{1}{2}}^{2} dw\
    + \int_{2}^{\frac{3T}{|z|}} dw.
\end{align*}
The first and the last integrals are easily estimated.
According \eqref{Appendix-dem B1 ineq},
\begin{align*}
    \begin{cases}
    \frac{3}{4}\le|w^2+ e^{i2\theta}|\le\frac{5}{4}
    &\textnormal{if \,} 0<w\le\frac{1}{2},\\[5pt]
    \frac{3}{4}w^2\le|w^2+ e^{i2\theta}|\le\frac{5}{4}w^2
    &\textnormal{if \,} w\ge2,
    \end{cases}
\end{align*}
we deduce
\begin{align}\label{Appendix-B1 I0}
    I_{0} 
    = \int_{0}^{\frac{1}{2}} dw\
        w^{n-1} \asymp 1
\end{align}
and
\begin{align}\label{Appendix-B1 I infinity}
    I_{\infty} 
    = \int_{2}^{\frac{3T}{|z|}} dw\
        w^{n-2\gamma-1} 
    \asymp 
    \begin{cases}
    1
    &\textnormal{if \,} n<2\gamma,\\[5pt]
    1+\log\frac{T}{|z|}
    &\textnormal{if \,} n=2\gamma,\\[5pt]
    (\frac{T}{|z|})^{n-2\gamma}
    &\textnormal{if \,} n>2\gamma.
    \end{cases}
\end{align}
Let us turn to the remaining integral, where
$\frac{1}{2}\le{w}\le2$. In this case we use the following
improvement of \eqref{Appendix-dem B1 ineq}:
\begin{align*}
    |w^2+ e^{i2\theta}|^2
    &= w^2+1+2w^2\cos{2\theta}
    = (w^2-1)^{2}+4w^2\cos^{2}\theta\\
    &\asymp \big(w-\frac{1}{w}\big)^{2}+\cos^{2}\theta.
\end{align*}
By performing the change of variables $u=w-\frac{1}{w}$ 
and noticing that $\frac{du}{dw}=1+\frac{1}{w^2}\asymp1$, 
we get
\begin{align}
    I_{1} 
    &\asymp 
    \int_{-\frac{3}{2}}^{\frac{3}{2}} du\
        (u^2+\cos^2\theta)^{-\frac{\gamma}{2}}
    \asymp
    \int_{0}^{\frac{3}{2}} du\
        (u+\cos\theta)^{-\gamma}\nonumber\\[5pt]
    &\asymp
    \begin{cases}
        (\cos\theta)^{-\gamma-1}
        &\textnormal{if \,} \gamma>1,\\[5pt]
        1-\log(\cos\theta)
        &\textnormal{if \,} \gamma=1,\\[5pt]
        1
        &\textnormal{if \,} \gamma<1.
    \end{cases}\label{Appendix-B1 I1}
\end{align}
In conclusion, \eqref{Appendix-lemma B1} is obtained by
combining \eqref{Appendix-B1 I0}, \eqref{Appendix-B1 I infinity} and \eqref{Appendix-B1 I1}.
\end{proof}

\begin{lemma}
Let $z\in\mathbb{C}$ with $\re{z}>0$ and $u\in\mathbb{R}$. 
Then
\begin{align}\label{Appendix-lemma B2}
    \int_{0}^{+\infty} d(w^2)\
    R_{+}^{1-\varepsilon} (w^2-u^2)
    \frac{1}{\pi} \frac{\tau}{w^2+z^2} =
    \begin{cases}
        \frac{1}{\pi}\frac{z}{u^2+z^2}
        &\textnormal{if \,} \varepsilon=1,\\[5pt]
        \frac{1}{\sqrt{\pi}}
        \frac{z}{\sqrt{u^2+z^2}}
        &\textnormal{if \,} \varepsilon=\frac{1}{2}.
    \end{cases}
\end{align}
\end{lemma}

\begin{proof}
The case $\varepsilon=1$ follows immediately from the fact
the distribution $R_{+}^{0}$ is equal to the Dirac measure
at the origin. In the case $\varepsilon=\frac{1}{2}$, the
formula is proved first for $z>0$ and then extended
straightforwardly by analytic continuation to all
$z\in\mathbb{C}$ with $\re{z}>0$. Specifically, the left
hand side of \eqref{Appendix-lemma B2} becomes
\begin{align*}
    \pi^{-\frac{3}{2}} \int_{0}^{+\infty}
    \frac{d(w^2)}{w}
    \frac{z}{w^2+u^2+z^2}
    = 
    \underbrace{
    2 \pi^{-\frac{3}{2}} \int_{0}^{+\infty}
    \frac{dr}{r^2+1}
    }_{\frac{1}{\sqrt{\pi}}}
    \frac{z}{\sqrt{u^2+z^2}}
\end{align*}
after performing the change of variables 
$w=\sqrt{u^2+z^2}r$.
\end{proof}
\vspace{10pt}

\begin{proof}[Proof of \cref{Appendix-prop atau}]
According to the asymptotic expansion 
\eqref{Appendix-final expansion}, we write
\begin{align*}
    a_{\tau} (\exp{H})
    =   J(H)^{-\frac{1}{2}}\,
        \sum_{k=0}^{[d/2]}\,4^{-k}\, 
        U_{k}(H)\,I_{k}(\tau,H)
        + E(\tau,H)
\end{align*}
with
\begin{align*}
    I_{k}(\tau,H)
    = \int_{0}^{+\infty}d(v^2)\,
        p_{\tau}^{\mathbb{R}}(v)\,
        R_{+}^{k-\frac{d-1}{2}}(v^2-|H|^2)
\end{align*}
and
\begin{align*}
    E(\tau,H)
    &=J(H)^{-\frac{1}{2}} 
        \sum_{k=0}^{[d/2]}\,4^{-k}\,U_{k}(H)\,\\
     &\times\int_{0}^{+\infty} d(v^2)\,
            \lbrace\chi_{T}(v)-1\rbrace\,
            p_{\tau}^{\mathbb{R}}(v)\,
            R_{+}^{k-\frac{d-1}{2}}(v^2-|H|^2)\\[5pt]
    &+\ 2\int_{0}^{+\infty} dv\,
        \chi_{T}(v)\,p_{\tau}^{\mathbb{R}}(v)\,
        E_{\Phi}(v,H)
\end{align*}
Let $\varepsilon=1$ if $d$ is even and $\varepsilon=\frac{1}{2}$ 
if $d$ odd. Then
\begin{align*}
    I_{k}(\tau,H) 
    =
    \textstyle{\big(-\frac{\partial}
    {\partial(|H|^2)}\big)^{[\frac{d}{2}]-k}}
    \int_{0}^{+\infty} d(v^2)\,
        p_{\tau}^{\mathbb{R}}(v)\,
        R_{+}^{1-\varepsilon}(v^2-|H|^2)
\end{align*}
where
\begin{align*}
    \int_{0}^{+\infty} d(v^2)\,
        p_{\tau}^{\mathbb{R}}(v)\,
        R_{+}^{1-\varepsilon}(v^2-|H|^2)
    =
    \begin{cases}
        \frac{1}{\pi}\,\frac{\tau}{|H|^2+\tau^2}
        &\textnormal{if \,} \varepsilon=1,\\[5pt]
        \frac{1}{\sqrt{\pi}}\,
        \frac{\tau}{\sqrt{|H|^2+\tau^2}}
        &\textnormal{if \,} \varepsilon=\frac{1}{2},
    \end{cases}    
\end{align*}
according to \eqref{Appendix-lemma B2}. Then we obtain
\begin{align*}
    I_{k}(\tau,H) = 
    \frac{\tau}{\pi}\,
    \textstyle{\frac{\Gamma(\frac{d+1}{2}-k)}
    {(|H|^2+\tau^2)^{\frac{d+1}{2}-k}}}.
\end{align*}
Next, we estimate the remainder $E(\tau,H)$ whose second
part is easily handled. By using 
\eqref{Appendix-final estimate}, we have
\begin{align*}
    |E_{2}(\tau,H)|
    \le&\
    2\int_{0}^{+\infty} dv\,
        \chi_{T}(v)\,|p_{\tau}^{\mathbb{R}}(v)|\,
        |E_{\Phi}(v,H)|\\[5pt]
    \lesssim&\
        e^{-\langle\rho,H\rangle}
        \int_{0}^{3T} dv\,
        \frac{|\tau|}{|v^2+\tau^2|}\,(1+v^{3(\frac{d}{2}+1)}).
\end{align*}
where
\begin{align*}
    |\tau| \int_{0}^{3T} dv\,
    |v^2+\tau^2|^{-1}
    \lesssim 
    1+\textstyle{\log\frac{|\tau|}{\re{\tau}}}
\end{align*}
and
\begin{align*}
    |\tau| \int_{0}^{3T} dv\,
    |v^2+\tau^2|^{-1}\,v^{3(\frac{d}{2}+1)}
    \lesssim 
    \textstyle{|\tau|^{3(\frac{d}{2}+1)}\,
    \big\lbrace\big(\frac{T}{|\tau|}\big)^{\frac{3d+1}{2}}
    +\log\frac{|\tau|}{\re{\tau}}\big\rbrace}
\end{align*}
according to the formulas in \eqref{Appendix-lemma B1}. 
We deduce
\begin{align}\label{Appendix-estim E2}
    |E_{2}(\tau,H)|
    \lesssim T^{3(\frac{d}{2}+1)}(\log{T}-\log{s})
    e^{-\langle\rho,H\rangle}.
\end{align}
It remains to estimate
\begin{align*}
    E_{1}(\tau,H)
    &=J(H)^{-\frac{1}{2}} 
        \sum_{k=0}^{[d/2]}\,4^{-k}\,U_{k}(H)\\
    &\times\underbrace{
        \int_{0}^{+\infty} d(v^2)\,
            \lbrace\chi_{T}(v)-1\rbrace
            p_{\tau}^{\mathbb{R}}(v)\,
            R_{+}^{k-\frac{d-1}{2}}(v^2-|H|^2)
        }_{\widetilde{I}_{k}(\tau,H)}.
\end{align*}
By repeating the previous calculations for $I_{k}$,
\begin{align*}
    \widetilde{I}(\tau,H)
    = \frac{\tau}{\pi}
        &\textstyle{\big(-\frac{\partial}
        {\partial(|H|^2)}
        \big)^{[\frac{d}{2}]-k}}\\[5pt]
    &\int_{0}^{+\infty} d(v^2)\,
        \lbrace\chi_{T}(\sqrt{v^2+|H|^2})-1\rbrace\,
        \frac{1}{v^2+|H|^2+\tau^2}\,
        R_{+}^{1-\varepsilon}(v^2).
\end{align*}
Let's first consider the case where $\varepsilon=1$, 
i.e., $d$ is odd. Then
\begin{align*}
    &\widetilde{I}_{k}(\tau,H)
    =\frac{\tau}{\pi}\,
        \textstyle{\big(-\frac{\partial}
        {\partial(|H|^2)}\big)^{\frac{d-1}{2}-k}}
        \Big\lbrace
        \big(\chi_{T}(|H|)-1\big)\,
        \frac{1}{|H|^2+\tau^2}\,
        \Big\rbrace \\[5pt]
    &=\frac{\tau}{\pi}\,
        \sum_{j+j'=\frac{d-1}{2}-k}
        \textstyle{\frac{(\frac{d-1}{2}-k)!}{j!j'!}\,
        \big(-\frac{\partial}
        {\partial(|H|^2)}\big)^{j}}
        \big(\chi_{T}(|H|)-1\big)\,
        \textstyle{\big(-\frac{\partial}
        {\partial(|H|^2)}\big)^{j'}}
        \frac{1}{|H|^2+\tau^2}.
\end{align*}
On the one hand, the expression 
$\big(-\frac{\partial}
    {\partial(|H|^2)}\big)^{j}
    \big(\chi_{T}(|H|)-1\big)$
vanishes when $|H|\le2T$. 
In addition, it is $O(T^{-2j})$.
On the other hand,
\begin{align*}
    \textstyle{\big(-\frac{\partial}
    {\partial(|H|^2)}\big)^{j'}}
    \frac{1}{|H|^2+\tau^2}
    = \frac{j'!}{(|H|^2+\tau^2)^{j'+1}}
    = O (T^{-2j'-2})
\end{align*}
when $|H|\ge2T$. We deduce that
\begin{align*}
    |\widetilde{I}_{k}(\tau,H)|
    = O(T^{2k-d}).
\end{align*}
Let's then consider the case where $\varepsilon=\frac{1}{2}$,
i.e., $d$ is even. Then
\begin{align*}
    &\widetilde{I}_{k}(\tau,H)
    =\frac{2\tau}{\pi}\,
    \sum_{j+j'=\frac{d}{2}-k}
    \textstyle{\frac{(\frac{d}{2}-k)!}{j!}}\\[5pt]
    &\times\int_{0}^{+\infty} dv\,
        \textstyle{\big(-\frac{\partial}
        {\partial(|H|^2)}\big)^{j}}
        \lbrace\chi_{T}(\sqrt{v^2+|H|^2})-1\rbrace\,
        (v^2+|H|^2+\tau^2)^{-j'-1}.
\end{align*}
Again, the expression
$\big( -\frac{\partial}
    {\partial(|H|^2)}\big)^{j}
    \lbrace\chi_{T}(\sqrt{v^2+|H|^2})-1\rbrace$
is $O(T^{-2j})$, which vanishes when
$v^2+|H|^2\le4T^2$, as well as $v^2+|H|^2\ge9T^2$ if $j>0$.
It follows that the integral above is $O(T^{2k-d-1})$
if $j>0$, and that it is estimated by
\begin{align*}
    \int_{v^2+|H|^2\ge4T^2}dv\,
    \big|v^2+|H|^2+\tau^2\big|^{k-\frac{d}{2}-1}
    \lesssim&\,
    \int_{v+|H|\ge2T}dv\,
    (v+|H|)^{2k-d-2}\\[5pt]
    \asymp&\, T^{2k-d-1}
\end{align*}
if $j=0$. 
In any case, we obtain
\begin{align*}
    |\widetilde{I}_{k}(\tau,H)|
    \lesssim\, T^{2k-d}\lesssim 1
\end{align*}
and therefore
\begin{align*}
    |E_{1}(\tau,H)|
    \lesssim J(H)^{-\frac{1}{2}}
\end{align*}
since the coefficients $U_{k}$ are bounded.
By combining with \eqref{Appendix-estim E2},
we conclude that
\begin{align*}
    |E(\tau,H)| 
    \lesssim\,|E_{1}(\tau,H)|+|E_{1}(\tau,H)|
    \lesssim\,
    T^{3(\frac{d}{2}+1)}\,
    (\log{T}-\log{s})\,e^{-\langle\rho,H\rangle}
\end{align*}
for all $H\in\frakACL$.
\end{proof}

\vspace{20pt}
\printbibliography
\vspace{20pt}

\address{
\noindent \textbf{JEAN-PHILIPPE ANKER:}
anker@univ-orleans.fr\\
Institut Denis Poisson (UMR 7013),
Universit\'e d'Orl\'eans, Universit\'e de Tours \& CNRS,
B\^atiment de math\'ematiques - Rue de Chartres,
B.P. 6759 - 45067 Orl\'eans cedex 2 - France}
\vspace{10pt}

\address{
\noindent\textbf{HONG-WEI ZHANG:}
zhw.dimn@gmail.com\\
Ghent University, Department of Mathematics, Analysis, Logic and Discrete Mathematics - Krijgslaan 281, Building S8 - B9000 Ghent - Belgium
\\and\\
Institut Denis Poisson (UMR 7013),
Universit\'e d'Orl\'eans, Universit\'e de Tours \& CNRS,
B\^atiment de math\'ematiques - Rue de Chartres,
B.P. 6759 - 45067 Orl\'eans cedex 2 - France}

\end{document}